\documentclass[11pt]{amsart}

\usepackage{amsmath}
\usepackage{amsfonts}
\usepackage{amssymb}
\usepackage{graphicx}
\usepackage{mathrsfs}
\usepackage{pb-diagram}
\usepackage{epstopdf}

\theoremstyle{plain}
\newtheorem{prop}{Proposition}[section]
\newtheorem{theorem}[prop]{Theorem}

\newtheorem{conjecture}[prop]{Conjecture}
\newtheorem{corollary}[prop]{Corollary}
\newtheorem{definition}[prop]{Definition}

\newtheorem{lemma}[prop]{Lemma}
\newtheorem{remark}[prop]{Remark}

\numberwithin{equation}{section}

\newcommand{\conn}{\nabla}

\newcommand{\der}{\mathrm{d}}

\newcommand{\pairing}[2]{\left( #1\, , \, #2 \right)}

\newcommand{\nat}{\mathbb{N}}
\newcommand{\integer}{\mathbb{Z}}
\newcommand{\rat}{\mathbb{Q}}
\newcommand{\real}{\mathbb{R}}
\newcommand{\cpx}{\mathbb{C}}

\newcommand{\consti}{\mathbf{i}\,}
\newcommand{\conste}{\mathbf{e}}

\newcommand{\sphere}[1]{\mathbf{S}^{#1}}
\newcommand{\torus}[1]{\mathbf{T}^{#1}}
\newcommand{\proj}{\mathbb{P}}

\newcommand{\isomto}{\overset{\sim}{\longrightarrow}}

\newcommand{\Aut}{\mathrm{Aut}}

\newcommand{\Hom}{\mathrm{Hom}}
\newcommand{\Image}{\mathrm{Im}}
\newcommand{\Kernel}{\mathrm{Ker}}

\newcommand{\GL}{\mathrm{GL}}

\begin{document}
\title[SYZ for toric CY]{SYZ mirror symmetry for\\toric Calabi-Yau manifolds}

\author[K. Chan]{Kwokwai Chan}
\address{Department of Mathematics, The Chinese University of Hong Kong, Shatin, N.T., Hong Kong}
\email{kwchan@math.cuhk.edu.hk}

\author[S.-C. Lau]{Siu-Cheong Lau}
\address{Institute for the Physics and Mathematics of the Universe (IPMU), University of Tokyo, Kashiwa, Chiba 277-8583, Japan}
\email{siucheong.lau@ipmu.jp}

\author[N.C. Leung]{Naichung Conan Leung}
\address{Department of Mathematics and the Institute of Mathematical Sciences, The Chinese University of Hong Kong, Shatin, N.T., Hong Kong}
\email{leung@math.cuhk.edu.hk}

\begin{abstract}
We investigate mirror symmetry for toric Calabi-Yau manifolds from the perspective of the SYZ conjecture. Starting with a non-toric special Lagrangian torus fibration on a toric Calabi-Yau manifold $X$, we construct a complex manifold $\check{X}$ using T-duality modified by quantum corrections. These corrections are encoded by Fourier transforms of generating functions of certain open Gromov-Witten invariants. We conjecture that this complex manifold $\check{X}$, which belongs to the Hori-Iqbal-Vafa mirror family, is inherently written in canonical flat coordinates. In particular, we obtain an enumerative meaning for the (inverse) mirror maps, and this gives a geometric reason for why their Taylor series expansions in terms of the K\"ahler parameters of $X$ have integral coefficients. Applying the results in \cite{Chan10} and \cite{LLW10}, we compute the open Gromov-Witten invariants in terms of local BPS invariants and give evidences of our conjecture for several 3-dimensional examples including $K_{\proj^2}$ and $K_{\proj^1\times\proj^1}$.
\end{abstract}

\maketitle

\tableofcontents

\section{Introduction}

For a pair of mirror Calabi-Yau manifolds $X$ and $\check{X}$, the Strominger-Yau-Zaslow (SYZ) conjecture \cite{syz96} asserts that there exist special Lagrangian torus fibrations $\mu:X\to B$ and $\check{\mu}:\check{X}\to B$ which are fiberwise-dual to each other. In particular, this suggests an intrinsic construction of the mirror $\check{X}$ by fiberwise dualizing a special Lagrangian torus fibration on $X$. This process is called \textit{T-duality}.

The SYZ program has been carried out successfully in the semi-flat case \cite{kontsevich00, LYZ, boss01}, where the discriminant loci of special Lagrangian torus fibrations are empty (i.e. all fibers are regular) and the base $B$ is a smooth integral affine manifold. On the other hand, mirror symmetry has been extended to non-Calabi-Yau settings, and the SYZ construction has been shown to work in the toric case \cite{auroux07, auroux09, chan08}, where the discriminant locus appears as the boundary of the base $B$ (so that $B$ is an integral affine manifold with boundary).

In general, by fiberwise dualizing a special Lagrangian torus fibration $\mu:X\rightarrow B$ away from the discriminant locus, one obtains a manifold $\check{X}_0$ equipped with a complex structure $J_0$, the so-called semi-flat complex structure. In both the semi-flat and toric cases, $(\check{X}_0,J_0)$ already serves as the complex manifold mirror to $X$. However, when the discriminant locus $\Gamma$ appears inside the interior of $B$ (so that $B$ is an integral affine manifold with singularities), $\check{X}_0$ is contained in the mirror manifold $\check{X}$ as an open dense subset and the semi-flat complex structure $J_0$ does not extend to the whole $\check{X}$. It is expected that the genuine mirror complex structure $J$ on $\check{X}$ can be obtained by deforming $J_0$ using instanton corrections and wall-crossing formulas, which come from symplectic enumerative information on $X$ (see Fukaya \cite{fukaya05}, Kontsevich-Soibelman \cite{kontsevich-soibelman04} and Gross-Siebert \cite{gross07}\footnote{This is related to the so-called ``reconstruction problem". This problem was first attacked by Fukaya \cite{fukaya05} using heuristic arguments in the two-dimensional case, which was later given a rigorous treatment by Kontsevich-Soibelman in \cite{kontsevich-soibelman04}; the general problem was finally solved by the important work of Gross-Siebert \cite{gross07}.}). This is one manifestation of the mirror phenomenon that the complex geometry of the mirror $\check{X}$ encodes symplectic enumerative data of $X$.

To go beyond the semi-flat and toric cases, a good starting point is to work with non-toric special Lagrangian torus fibrations\footnote{Here, ``non-toric" means the fibrations are not those provided by moment maps of Hamiltonian torus actions on toric varieties.} on toric Calabi-Yau manifolds constructed by Gross \cite{gross_examples} \footnote{These fibrations were also constructed by Goldstein \cite{goldstein} independently, but they were further analyzed by Gross from the SYZ perspective.}, which serve as local models of Lagrangian torus fibrations on compact Calabi-Yau manifolds. Interior discriminant loci are present in these fibrations, leading to wall-crossing phenomenon of disk counting invariants and nontrivial quantum corrections of the mirror complex structure. In this paper, we construct the instanton-corrected mirrors of toric Calabi-Yau manifolds by running the SYZ program for these non-toric fibrations. This generalizes the work of Auroux \cite{auroux07, auroux09}, in which he considered non-toric Lagrangian torus fibrations on $\cpx^n$ and constructed instanton-corrected mirrors by studying the wall-crossing phenomenon of disk counting invariants.

What follows is an outline of our main results. We first need to fix some notations. Let $N \cong \integer^n$ be a lattice and $M=\textrm{Hom}(N,\integer)$ be its dual. For a $\integer$-module $R$, we let $N_R:=N\otimes_\integer R$, $M_R:=M\otimes_\integer R$ and denote by $\pairing{\cdot}{\cdot}:M_R\times N_R\to R$ the natural pairing.

Let $X=X_\Sigma$ be a toric manifold defined by a fan $\Sigma$ in $N_\real$, and $v_0,v_1,\ldots,v_{m-1}\in N$ be the primitive generators of the 1-dimensional cones of $\Sigma$. Suppose that $X$ is Calabi-Yau. This condition is equivalent to the existence of $\underline\nu\in M$ such that
$$\pairing{\underline\nu}{v_i}=1$$
for $i=0,1,\ldots,m-1$.  As in \cite{gross_examples}, we also assume that the fan $\Sigma$ has convex support, so that $X$ is a crepant resolution of an affine toric variety with Gorenstein canonical singularities.  Equip $X$ with a toric K\"ahler structure $\omega$.

We study the SYZ aspect of mirror symmetry for every toric Calabi-Yau manifold $X$, which is usually called ``local mirror symmetry" in the literature, because it was derived by considering certain limits in the K\"ahler and complex moduli spaces of Calabi-Yau hypersurfaces in toric varieties (see Katz-Klemm-Vafa \cite{KKV97}).  Chiang-Klemm-Yau-Zaslow \cite{CKYZ} verified by direct computations that closed Gromov-Witten invariants of a local Calabi-Yau match with the period integrals in the mirror side;  in \cite{HIV00}, Hori-Iqbal-Vafa wrote down the following formula for the mirror $\check{X}$ of $X$:
\begin{equation}\label{HIV}
\check{X}=\left\{(u,v,z_1,\ldots,z_{n-1})\in\cpx^2\times(\cpx^\times)^{n-1}:uv=\sum_{i=0}^{m-1}C_iz^{v_i}\right\},
\end{equation}
where $C_i\in\cpx$ are some constants (which determine the complex structure of $\check{X}$) and $z^{v_i}$ denotes the monomial $\prod_{j=1}^{n-1} z_j^{\pairing{\nu_j}{v_i}}$.  Here $\{\nu_j\}_{j=0}^{n-1} \subset M$ is the dual basis of $\{v_j\}_{j=0}^{n-1} \subset N$.  One of the aims of this paper is to explain why, from the SYZ viewpoint, the mirror $\check{X}$ should be written in this form.

We now outline our SYZ mirror construction. To begin with, fix a constant $K_2$, and let $D\subset X$ be the hypersurface $\{x\in X: w(x)-K_2=0\}$, where $w:X\rightarrow\cpx$ is the holomorphic function corresponding to the lattice point $\underline{\nu}\in M$. In Section 4, we consider a non-toric special Lagrangian torus fibration $\mu:X\to B$ constructed by Gross \cite{gross_examples}, where $B$ is a closed upper half space in $\real^n$.  As shown in \cite{gross_examples}, the discriminant locus $\Gamma$ of this fibration consists of $\partial B$ together with a codimension two subset contained in a hyperplane $H \subset B$.  We will show that the special Lagrangian torus fibers over $H$ are exactly those which bound holomorphic disks of Maslov index zero (Lemma 4.24).  $H$, which is called `the wall', separates $B_0 := B - \Gamma$ into two chambers:
$$B_0 - H =B_+\cup B_-.$$

As we have discussed above, fiberwise dualizing the torus bundle over $B_0$ gives a complex manifold $\check{X}_0$ which is called the semi-flat mirror.  Yet this procedure ignores the singular fibers of $\mu$, and `quantum corrections' are needed to construct the mirror $\check{X}$ out from $\check{X}_0$.  To do this, we consider virtual counting of Maslov index two holomorphic disks with boundary in special Lagrangian torus fibers.  The result for the counting is different for fibers over the chambers $B_+$ and $B_-$ (see Propositions \ref{disk counting + in X} and \ref{disk counting + in X}). This leads to a wall-crossing formula for disk counting invariants, which is exactly the correct formula we need to glue the torus bundles over $B_+$ and $B_-$.  This wall-crossing phenomenon has been studied by Auroux \cite{auroux07, auroux09} in various examples including $\proj^2$ and the Hirzebruch surfaces $\mathbb{F}_2, \mathbb{F}_3$.\footnote{We shall emphasize that the wall-crossing formulas (or gluing formulas) studied by Auroux and us here are special cases of those studied by Kontsevich-Soibelman \cite{kontsevich-soibelman04} and Gross-Siebert \cite{gross07}.}

Now, one of the main results of this paper is that by the SYZ construction (see Section \ref{mir_construct} for the details), the instanton-corrected mirror of $X$ is given by the following noncompact Calabi-Yau manifold (Theorem \ref{mir_thm}):
\begin{equation} \label{SYZ}
\begin{split}
\check{X} &= \left\{(u,v,z) \in \cpx^2 \times (\cpx^\times)^{n-1}: uv = G(z) \right\}, \textrm{ where}\\
G(z) &= (1+\delta_0) + \sum_{j=1}^{n-1} (1 + \delta_j) z_j + \sum_{i=n}^{m-1} (1 + \delta_i) q_{i-n+1} z^{v_i}
\end{split}
\end{equation}
where
$$\delta_i (q) =\sum_{\alpha\in H_2^{\textrm{eff}}(X,\integer)-\{0\}}n_{\beta_i+\alpha}q^\alpha$$
and $q_a$ are K\"ahler parameters of $X$.  Here, $H_2^{\textrm{eff}}(X,\integer)$ is the cone of effective classes, $q^\alpha$ denotes $\exp(-\int_\alpha\omega)$ which can be expressed in terms of the K\"ahler parameters $q_a$, $\beta_i\in\pi_2(X,\mathbf{T})$ are the basic disk classes (see Section \ref{gen_disks_X}), and the coefficients $n_{\beta_i+\alpha}$ are one-pointed genus zero open Gromov-Witten invariants defined by Fukaya-Oh-Ohta-Ono \cite{FOOO1} (see Definition \ref{open_GW}). Furthermore, we can show that the symplectic structure $\omega$ on $X$ is transformed to a holomorphic volume form on the semi-flat mirror $\check{X}_0$, which naturally extends to a holomorphic volume form $\check{\Omega}$ on $\check{X}$ (Proposition 4.38).

Note that the instanton-corrected mirror \eqref{SYZ} that we write down is of the form \eqref{HIV} suggested by Hori-Iqbal-Vafa \cite{HIV00}. Yet \eqref{SYZ} contains more information: it is explicitly expressed in terms of symplectic data, namely, the K\"ahler parameters and open Gromov-Witten invariants on $X$. Morally speaking, the semi-flat complex structure is the constant term in the fiberwise Fourier expansion of the corrected complex structure $J$. The higher Fourier modes correspond to genus-zero open Gromov-Witten invariants $n_{\beta_i+\alpha}$ in $X$, which are virtual counts of Maslov index two holomorphic disks with boundary in Lagrangian torus fibers.

Local mirror symmetry asserts that there is (at least locally near the large complex structure limits) a canonical isomorphism
$$\psi:\mathcal{M}_C(\check{X})\to\mathcal{M}_K(X),$$
called the \textit{mirror map}, from the complex moduli space $\mathcal{M}_C(\check{X})$ of $\check{X}$ to the (complexified) K\"ahler moduli space $\mathcal{M}_K(X)$ of $X$, which gives flat coordinates on $\mathcal{M}_C(\check{X})$. The mirror map is defined by periods as follows. For a point $\check{q}=(\check{q}_1,\ldots,\check{q}_l)\in\mathcal{M}_C(\check{X})$, where $l=m-n$, let
$$\check{X}_{\check{q}} = \left\{(u,v,z) \in \cpx^2 \times (\cpx^\times)^{n-1}: uv = 1 + \sum_{j=1}^{n-1} z_j + \sum_{i=n}^{m-1} \check{q}_{i-n+1} z^{v_i} \right\}$$
be the corresponding mirror Calabi-Yau manifold equipped with a holomorphic volume form $\check{\Omega}_{\check{q}}$. Then, for any n-cycle $\gamma\in H_n(\check{X},\integer)$, the period
$$\Pi_\gamma(\check{q}):=\int_\gamma\check{\Omega}_{\check{q}},$$
as a function of $\check{q}\in\mathcal{M}_C(\check{X})$, satisfies the $A$-hypergeometric system of linear differential equations associated to $X$ (see e.g. Hosono \cite{hosono06}). Let $\Phi_1(\check{q}),\ldots,\Phi_l(\check{q})$ be a basis of the solutions of this system with a single logarithm. Then there is a basis $\gamma_1,\ldots,\gamma_l$ of $H_n(\check{X},\integer)$ such that
$$\Phi_a(\check{q})=\int_{\gamma_a}\check{\Omega}_{\check{q}}$$
for $a=1,\ldots,l$, and the mirror map $\psi$ is given by
\begin{align*}
\psi(\check{q}) &= (q_1(\check{q}),\ldots,q_l(\check{q}))\in\mathcal{M}_K(X), \textrm{ where} \\
q_a(\check{q}) &= \exp(-\Phi_a(\check{q}))=\exp\left(-\int_{\gamma_a}\check{\Omega}_{\check{q}}\right)
\end{align*}
for $a=1,\ldots,l$.

A striking feature of our instanton-corrected mirror family (\ref{SYZ}) is that it is inherently written in flat coordinates.\footnote{This was first observed by Gross and Siebert \cite{gross07}; see Remark \ref{remark1.2} below.} We formulate this as a conjecture as follows. By considering Equation \eqref{SYZ}, one obtains a map $\phi:\mathcal{M}_K(X)\to\mathcal{M}_C(\check{X}), q=(q_1,\ldots,q_l)\mapsto\phi(q)=(\check{q}_1(q),\ldots,\check{q}_l(q))$ defined by
$$\check{q}_a(q)=q_a (1 + \delta_{a+n-1}) \prod_{j=0}^{n-1} (1+\delta_j)^{-\pairing{\nu_j}{v_{a+n-1}}},\ a=1,\ldots,l.$$
Then we claim that $\check{q}_1(q),\ldots,\check{q}_l(q)$ are flat coordinates on $\mathcal{M}_C(\check{X})$ (see Conjecture \ref{can_coords} for more details):
\begin{conjecture}\label{can_coords}
The map $\phi$ is an inverse of the mirror map $\psi$. In other words, there exists a basis $\gamma_1,\ldots,\gamma_l$ of $H_n(\check{X},\integer)$ such that
$$q_a=\exp\left(-\int_{\gamma_a}\check{\Omega}_{\check{q}}\right),$$
for $a=1,\ldots,l$, where $\check{q}=\phi(q)$ is defined as above.
\end{conjecture}
In particular, our construction of the instanton-corrected mirror via SYZ provides an enumerative meaning to the inverse mirror map. This also shows that the integrality of the coefficients of the Taylor series expansions of the functions $\check{q}_a$ in terms of $q_1,\ldots,q_l$ (see, e.g. \cite{Z10}) is closely related to enumerative meanings of the coefficients.

In Section 5, we shall provide evidences to Conjecture \ref{can_coords} in some 3-dimensional examples including $K_{\proj^2}$ and $K_{\proj^1\times\proj^1}$. This is done by computing the one-pointed genus zero open Gromov-Witten invariants of a toric Calabi-Yau 3-fold of the form $X=K_Z$, where $Z$ is a toric del Pezzo surface, in terms of local BPS invariants of the toric Calabi-Yau 3-fold $K_{\tilde{Z}}$, where $\tilde{Z}$ is a toric blow-up of $Z$ at a toric fixed point. The computation is an application of the results in \cite{Chan10} and \cite{LLW10}. In \cite{Chan10}, the first author of this paper shows that the open Gromov-Witten invariants of $X=K_Z$ are equal to certain closed Gromov-Witten invariants of a suitable compactification $\bar{X}$ (see Theorem \ref{cptification}). In the joint work \cite{LLW10} of the second and third authors with Wu, these closed Gromov-Witten invariants are shown to be equal to certain local BPS invariants of $K_{\tilde{Z}}$, where $\tilde{Z}$ denotes the blow-up of $Z$, by employing blow-up and flop arguments. Now, these latter invariants have already been computed by Chiang-Klemm-Yau-Zaslow in \cite{CKYZ}. Hence, by comparing with period computations such as those done by Graber-Zaslow in \cite{graber-zaslow01}, we can give evidences to the above conjecture for $K_{\proj^2}$ and $K_{\proj^1\times\proj^1}$.\\

\begin{remark}
Recently in a joint work of the first and second authors with
Hsian-Hua Tseng \cite{CLT}, Conjecture 1.1 was proved for $X = K_Y$, where
$Y$ is any toric Fano manifold. This paper takes a different approach and
the proof was by a computation of open Gromov-Witten invariants via
J-functions and a study of solutions of A-hypergeometric systems.
\end{remark}

\noindent\textbf{Example: $X=K_{\proj^2}$.} The primitive generators of the 1-dimensional cones of the fan $\Sigma$ defining $X=K_{\proj^2}$ are given by
$$v_0=(0,0,1), v_1=(1,0,1), v_2=(0,1,1), v_3=(-1,-1,1)\in N=\integer^3.$$
We equip $X$ with a toric K\"{a}hler structure $\omega$ associated to the moment polytope $P$ given as
\begin{align*}
P=\{&(x_1,x_2,x_3)\in\real^3:\\
&x_3\geq0,x_1+x_3\geq0,x_2+x_3\geq0,-x_1-x_2+x_3\geq-t_1\},
\end{align*}
where $t_1=\int_l\omega_1>0$ and $l\in H_2(X,\integer)=H_2(\proj^2,\integer)$ is the class of a line in $\proj^2$. To complexify the K\"ahler class, we set $\omega^\cpx=\omega+2\pi\sqrt{-1}B$, where $B$ is a real two-form (the $B$-field). We let $t=\int_l\omega^\cpx\in\cpx$.

Fix $K_2>0$ and let $D=\{x\in X:w(x)-K_2=0\}$. Then the base $B$ of the Gross fibration $\mu:X\to B$ is given by $B=\real^2\times\real_{\geq K_2}$. The wall is the real codimension one subspace $H=\real^2\times\{0\}\subset B$. The discriminant loci $\Gamma$ is a codimension two subset contained in $H$ as shown in Figure \ref{KP2_base}.

\begin{figure}[htp]
\caption{The base of the Gross fibration on $K_{\proj^2}$, which is an upper half space in $\real^3$.}
\label{KP2_base}
\begin{center}
\includegraphics{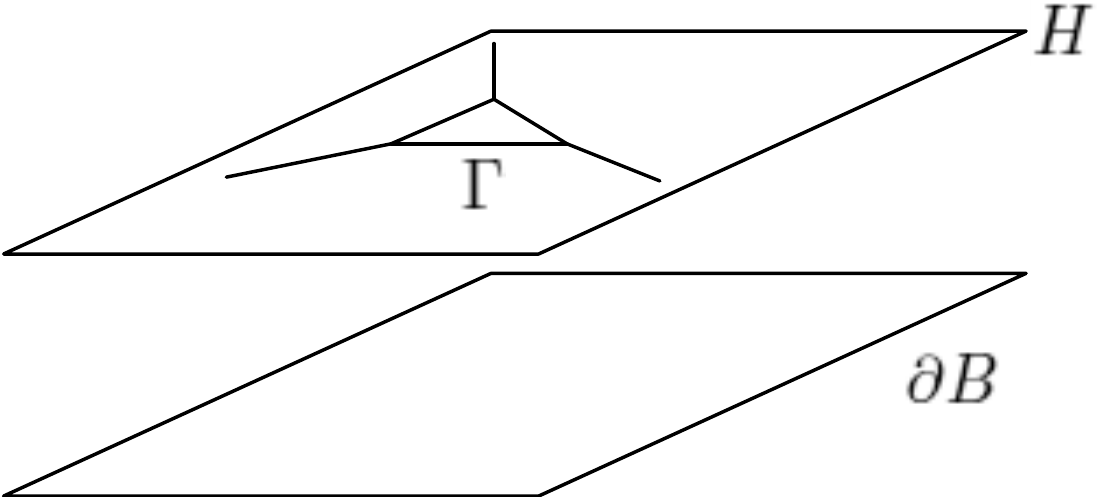}
\end{center}
\end{figure}

By Equation \eqref{SYZ}, the instanton-corrected mirror is given by
\begin{equation*}
\check{X}=\left\{
\begin{aligned}
&(u,v,z_1,z_2)\in \cpx^2\times(\cpx^\times)^2: \\
& uv= \left(1+\sum_{k=1}^\infty n_{\beta_0+kl}q^k\right)+z_1+z_2+\frac{q}{z_1z_2}
\end{aligned}
\right\}
\end{equation*}
where $q=\exp(-t)$ and $\beta_0$ is the basic disk class corresponding to the compact toric divisor $\proj^2 \subset X$.  By Corollary \ref{mirror_KS}, we can express the open Gromov-Witten invariants $n_{\beta_0+kl}$ in terms of the local BPS invariants of $K_{\mathbb{F}_1}$, where $\mathbb{F}_1$ is the blowup of $\proj^2$ at one point. More precisely, let $e,f\in H_2(\mathbb{F}_1,\integer)=H_2(K_{\mathbb{F}_1},\integer)$ be the classes represented by the exceptional divisor and fiber of the blowing up $\mathbb{F}_1\to\proj^2$ respectively. Then $n_{\beta_0+kl}$ is equal to the local BPS invariant $\textrm{GW}_{0,0}^{K_{\mathbb{F}_1},kf+(k-1)e}$ for the class $kf+(k-1)e\in H_2(K_{\mathbb{F}_1},\integer)$. These latter invariants have been computed by Chiang-Klemm-Yau-Zaslow and listed in the `sup-diagonal' of Table 10 on p. 56 in \cite{CKYZ}:
\begin{align*}
n_{\beta_0+l}  & = -2,\\
n_{\beta_0+2l} & = 5,\\
n_{\beta_0+3l} & = -32,\\
n_{\beta_0+4l} & = 286,\\
n_{\beta_0+5l} & = -3038,\\
n_{\beta_0+6l} & = 35870,\\
               &\vdots&
\end{align*}
Hence, the instanton-corrected mirror for $X=K_{\proj^2}$ is given by
$$\check{X}=\left\{(u,v,z_1,z_2)\in\cpx^2\times(\cpx^\times)^2:uv= c(q) + z_1+z_2+\frac{q}{z_1z_2}\right\},$$
where
$$c(q) = 1-2q+5q^2-32q^3+286q^4-3038q^5+\ldots.$$
By a change of coordinates the above defining equation can also be written as
$$ uv= 1 + z_1+z_2+\frac{q}{c(q)^3 z_1z_2}.$$
According to our conjecture above, the inverse mirror map $\phi:\mathcal{M}_K(X)\to\mathcal{M}_C(\check{X})$ is then given by
$$q\mapsto\check{q}:=q(1-2q+5q^2-32q^3+286q^4-3038q^5+\ldots)^{-3}.$$

On the other hand, the $A$-hypergeometric system of linear differential equations associated to $K_{\proj^2}$ is equivalent to the Picard-Fuchs equation
$$[\theta_{\check{q}}^3+3\check{q}\theta_{\check{q}}(3\theta_{\check{q}}+1)(3\theta_{\check{q}}+2)]\Phi(\check{q})=0,$$
where $\theta_{\check{q}}$ denotes $\check{q}\frac{\partial}{\partial\check{q}}$ (see, e.g. Graber-Zaslow \cite{graber-zaslow01}.) The solution of this equation with a single logarithm is given by
$$\Phi(\check{q})=-\log\check{q}-\sum_{k=1}^\infty\frac{(-1)^k}{k}\frac{(3k)!}{(k!)^3}\check{q}^k.$$
Then, $\Phi(\check{q})$ is the period of some 3-cycle $\gamma\in H_3(\check{X},\integer)$ and we have the mirror map
$$\psi:\mathcal{M}_C(\check{X})\to\mathcal{M}_K(X),\ q = \check{q}\exp\left(\sum_{k=1}^\infty\frac{(-1)^k}{k}\frac{(3k)!}{(k!)^3}\check{q}^k\right).$$
We can then invert the mirror map and express $\check{q}=\psi^{-1}(q)$ as a function of $q\in\mathcal{M}_K(X)$. One can check by direct computation that
$$(\exp(-\Phi(\check{q}))/\check{q})^{\frac{1}{3}}=1-2q+5q^2-32q^3+286q^4-3038q^5+\ldots.$$
This shows that $\phi$ agrees with the inverse mirror map at least up to the order $q^5$, thus providing ample evidence to Conjecture \ref{can_coords} for $X=K_{\proj^2}$. In fact, by using a computer, one can verify that $\phi$ agrees with the inverse mirror map up to a much larger order. \hfill $\square$

\begin{remark}\label{remark1.2}
The relevance of the work of Graber-Zaslow \cite{graber-zaslow01} to the relationship between the series $1-2q+5q^2-32q^3+286q^4-3038q^5+\ldots$ and canonical coordinates was first mentioned in Remark 5.1 of the paper \cite{gross07} by Gross and Siebert. They were also the first to observe that the coefficients of the above series have geometric meanings; namely, they show that these coefficients can be obtained by imposing the ``normalization" condition for slabs, which is a condition necessary to run their program and construct toric degenerations of Calabi-Yau manifolds. They predict that these coefficients are counting certain tropical disks. See Conjecture 0.2 in \cite{gross07} for more precise statements.

After reading a draft of our paper, Gross informed us that they have long been expecting that the slabs are closely related to 1-pointed open Gromov-Witten invariants, so they also expect that a version of our Conjecture \ref{can_coords} is true. While the Gross-Siebert program constructs the mirror B-model starting from tropical data (on which they impose the normalization condition) on the base of the Lagrangian torus fibration, we start from the A-model on a toric Calabi-Yau manifold and use symplectic enumerative data (holomorphic disk counting invariants) directly to construct the mirror B-model. Our approach is in a way complementary to that of Gross-Siebert.
\end{remark}

The organization of this paper is as follows. Section \ref{A-side} reviews the general concepts in the symplectic side needed in this paper  and gives the T-duality procedure with quantum corrections.  It involves a family version of the Fourier transform, which is defined in Section \ref{FT}. Then we carry out the SYZ construction of instanton-corrected mirrors for toric Calabi-Yau manifolds in details in Section \ref{torCY}. The (inverse) mirror maps and their enumerative meanings are discussed in Section \ref{period}.\\

\noindent\textbf{Acknowledgments.} We are heavily indebted to Baosen Wu for generously sharing his ideas and insight. In particular he was the first to observe that there is a relation between the closed Gromov-Witten invariants of $K_Z$ and the local BPS invariants of $K_{\tilde Z}$, which was studied in more details in the joint work Lau-Leung-Wu \cite{LLW10}. We are grateful to Mark Gross and Bernd Siebert for many useful comments and for informing us about their related work and thoughts. We would also like to thank Denis Auroux, Cheol-Hyun Cho, Kenji Fukaya, Yong-Geun Oh, Hiroshi Ohta, Kaoru Ono and Shing-Tung Yau for numerous helpful discussions and the referees for suggestions which greatly improve the exposition of this paper. The second author thanks Mark Gross for inviting him to visit UCSD in February 2010 and for many enlightening discussions on wall-crossing. He is also grateful to Cheol-Hyun Cho for the invitation to Seoul National University in June 2010 and for the joyful discussions on disk-counting invariants.

The research of the first author was partially supported by Harvard University and the Croucher Foundation Fellowship. The research of the second author was supported by Institute for the Physics and Mathematics of the Universe. The work of the third author described in this paper was substantially supported by a grant from the Research Grants Council of the Hong Kong Special Administrative Region, China (Project No. CUHK401809).

\section{The SYZ mirror construction} \label{A-side}
This section gives the procedure to construct the instanton-corrected mirror using SYZ. First, we give a review of the symplectic side of SYZ mirror symmetry. Section \ref{open GW section} introduces the open Gromov-Witten invariants defined by Fukaya-Oh-Ohta-Ono \cite{FOOO_I, FOOO_II}, which are essential to our mirror construction. The construction procedure of the instanton-corrected mirror is given in Section \ref{mir_construct}. We will apply this procedure to produce the instanton-corrected mirrors of toric Calabi-Yau manifolds in Section \ref{mirror}.

\subsection{Proper Lagrangian fibrations and semi-flat mirrors} \label{Lag fib}
This subsection is devoted to introduce the notion of a Lagrangian fibration, which is central to the SYZ program.  The setting introduced here includes moment maps on toric manifolds as examples, whose bases are polytopes which are manifolds with corners, but we also allow singular fibers.

Let $X^{2n}$ be a smooth connected manifold of dimension $2n$, and $\omega$ be a closed non-degenerate two-form on $X$. The pair $(X,\omega)$ is called a symplectic manifold. In our setup, $X$ is allowed to be \emph{non-compact}. This is important for us since a toric Calabi-Yau manifold can \emph{never} be compact.

We consider a fibration $\mu:X\to B$ (i.e. a smooth map such that $\mu(X)=B$), whose base $B$ is a smooth manifold with corners:
\begin{definition}\label{mfd_corner}
A Hausdorff topological space $B$ is a smooth $n$-manifold with corners if
\begin{enumerate}
\item For each $r\in B$, there exists an open set $U\subset B$ containing $r$, and a homeomorphism
$$\phi:U\to V\cap(\real_{\geq0}^k\times\real^{n-k})$$
for some $k=0,\ldots,n$, where $V$ is an open subset of $\real^n$ containing $0$, and $\phi(r)=0$. Such $r$ is called a $k$-corner point of $B$.
\item The coordinate changes are diffeomorphisms.
\end{enumerate}
\end{definition}

Basic examples of manifolds with corners are given by moment map polytopes of toric manifolds (see Figure \ref{KP1_poly} for an example). A manifold $B$ with corners is stratified by the subsets $B^{(k)}$ consisting of all $k$-corner points of $B$. $r\in B$ is called a boundary point of $B$ if it is a $k$-corner point for $k\geq1$.  Let
\begin{equation}
B^{\mathrm{int}}:=B^{(0)}=\{r\in B:r\textrm{ is not a boundary point}\}
\end{equation}
be the open stratum, and
\begin{equation}
\partial B:=B-B^{\mathrm{int}}.
\end{equation}

We will be dealing with those fibrations $\mu:X \to B$ which are \emph{proper} and \emph{Lagrangian}. Recall that $\mu$ is proper if $\mu^{-1}(K)$ is compact for every compact set $K \subset B$. And Lagrangian means the following:
\begin{definition}
Let $(X^{2n},\omega)$ be a symplectic manifold of dimension $2n$, and $B^n$ be a smooth $n$-fold with corners. A fibration $\mu:X\to B$ is said to be Lagrangian if at every regular point $x\in X$ with respect to the map $\mu$, the subspace $\Kernel(\der\mu(x))\subset T_x X$ is Lagrangian, that is,
$$\omega|_{\Kernel(\der\mu (x))}=0.$$
\end{definition}

>From now on we always assume that $\mu:X \to B$ is a Lagrangian fibration, whose fibers are denoted by
$$F_r:=\mu^{-1}(\{r\}),\ r\in B.$$
$F_r$ is called a regular fiber when $r$ is a regular value of $\mu$; otherwise it is called a singular fiber. It is the presence of singular fibers which makes the SYZ construction of instanton-corrected mirrors non-trivial.

The SYZ program asserts that the mirror of $X$ is given by the `dual torus fibration'. This dualizing procedure can be made precise if one restricts only to regular fibers. In view of this we introduce the following notations:
\begin{align}
\label{Gamma} \Gamma &:=\{r\in B:r\textrm{ is a critical value of }\mu\}; \\
B_0 &:= B-\Gamma; \\
\label{X_0} X_0 &:=\mu^{-1}(B_0).
\end{align}
$\Gamma$ is called the discriminant locus of $\mu$.

Now the restriction $\mu:X_0\to B_0$ is a proper Lagrangian submersion with connected fibers. Using the following theorem of Arnold-Liouville (see Section 50 of \cite{arnold_book}) on action-angle coordinates, it turns out that this can only be a torus bundle:

\begin{theorem}[Arnold-Liouville \cite{arnold_book}] \label{act_ang}
Let $\mu:X_0\to B_0$ be a proper Lagrangian submersion with connected fibers. Then $\mu$ is a torus bundle. Moreover, an integral affine structure is induced on $B_0$ in a canonical way.
\end{theorem}

An integral affine structure on $B_0$ is an atlas of coordinate charts such that the coordinate changes belong to $\GL(n,\integer)\ltimes\real^n$. The key to proving Theorem \ref{act_ang} is the observation that every cotangent vector at $r\in B_0$ induces a tangent vector field on $F_r$ by contracting with the symplectic two-form $\omega$ on $X$. Since $F_r$ is smooth and compact, a vector field integrates (for time $1$) to a diffeomorphism on $F_r$. In this way we get an action of $T^*_r B_0$ on $F_r$, and the isotropy subgroup of a point $x\in F_r$ can only be a lattice $\mathrm{L}$ in $T^*_r B_0$. Thus $T^*_r B_0/\mathrm{L}\cong F_r$.

Knowing that $\mu:X_0\to B_0$ is a torus bundle, we may then take its dual defined in the following way:
\begin{definition} \label{dual torus bundle}
Let $\mu:X_0\to B_0$ be a torus bundle. Its dual is the space $\check{X}_0$ of pairs $(F_r, \conn)$ where $r\in B_0$ and $\conn$ is a flat $U(1)$-connection on the trivial complex line bundle over $F_r$ up to gauge. There is a natural map $\check{\mu}:\check{X}_0\to B_0$ given by forgetting the second coordinate.
\end{definition}
The fiber of $\check{\mu}$ at $r$ is denoted as $\check{F}_r$.

\begin{prop}
In Definition \ref{dual torus bundle}, $\check{\mu}:\check{X}_0\to B_0$ is a torus bundle.
\end{prop}
\begin{proof}
For each $r\in B_0$, $\mathrm{Hom}(\pi_1(F_r),U(1))$ parameterizes all flat $U(1)$-connections on the trivial complex line bundle over $F_r$ by recording their holonomy. Thus
$$\check{F}_r\cong\mathrm{Hom}(\pi_1(F_r), U(1)).$$
Since $F_r$ is an $n$-torus, one has $\pi_1(F_r)\cong\integer^n$, and so
$$\mathrm{Hom}(\pi_1(F_r),U(1))\cong\real^n/\integer^n.$$
This shows that each fiber $\check{F}_r$ is a torus (which is dual to $F_r$).

To see that $\check{\mu}$ is locally trivial, take a local trivialization $\mu^{-1}(U)\cong U\times T$ of $\mu$, where $U$ is an open set containing $r$ and $T$ is a torus.  Then
$$\check{\mu}^{-1}(U)\cong U\times\mathrm{Hom}(\pi_1(T),U(1))\cong U\times T^*,$$
where $T^*$ denotes the dual torus to $T$.
\end{proof}

For a Lagrangian torus bundle $\mu:X_0\to B_0$, its dual $\check{X}_0$ has a canonical complex structure \cite{boss01} (we will write this down explicitly for toric Calabi-Yau manifolds in Section \ref{mirror}). Moreover, when the monodromies of the torus bundle $\mu:X_0\to B_0$ belong to $SL(n,\integer)$, there is a holomorphic volume form on $\check{X}_0$. Thus, by torus duality, a symplectic manifold with a Lagrangian bundle structure gives rise to a complex manifold. This exhibits the mirror phenomenon.

However, the above dualizing procedure takes place only away from the singular fibers (see Equation \eqref{X_0}) and hence it loses information. $\check{X}_0$ is called the \textit{semi-flat mirror}, which is only the `zeroth-order part' of the mirror of $X$ (see Remark \ref{approx_rem}). To remedy this, we need to `add back' the information coming from the singular fibers. This is precisely captured by the open Gromov-Witten invariants of $X$, which is discussed in the next subsection.

\subsection{Open Gromov-Witten invariants} \label{open GW section}
A crucial difference between $X_0$ and $X$ is that loops in the fibers of $X_0$ which represent non-trivial elements in $\pi_1$ never shrink, that is, $\pi_2(X_0,F_r)=0$ for every $r\in B_0$; while this is not the case for $X$ in general. To quantify this difference one needs to equip $X$ with an almost complex structure compatible with $\omega$ and count pseudo-holomorphic disks $(\Delta,\partial\Delta)\to(X,F_r)$. This gives the genus-zero open Gromov-Witten invariants defined by Fukaya-Oh-Ohta-Ono \cite{FOOO_I}.

In this section, we give a brief review on these invariants (Definition \ref{open_GW}). We also explain how to pack these invariants to form a generating function under the setting of Section \ref{Lag fib} (Definition \ref{Lambda*}).  We will see that via the fiberwise Fourier transform defined in Section \ref{FT}, these data serve as `quantum corrections' to the semi-flat complex structure of the mirror.

Let $(X,\omega)$ be a symplectic manifold equipped with an almost complex structure compatible with $\omega$. First of all, we need the following basic topological notions:
\begin{definition}
\begin{enumerate}
\item For a submanifold $L\subset X$, $\pi_2(X,L)$ is the group of homotopy classes of maps
$$u:(\Delta,\partial\Delta)\to(X,L),$$
where $\Delta:=\{z\in\cpx:|z|\leq1\}$ denotes the closed unit disk in $\cpx$. We have a natural homomorphism
$$\partial:\pi_2(X,L)\to\pi_1(L)$$
defined by $\partial[u]:=[u|_{\partial\Delta}]$.

\item For two submanifolds $L_0, L_1\subset X$, $\pi_2(X,L_0,L_1)$ is the set of homotopy classes of maps
$$u: ([0,1]\times\sphere{1},\{0\}\times\sphere{1},\{1\}\times\sphere{1})\to(X,L_0,L_1).$$
Similarly we have the natural boundary maps $\partial_+:\pi_2(X,L_0,L_1)\to\pi_1(L_1)$ and $\partial_-:\pi_2(X,L_0,L_1)\to\pi_1(L_0)$.
\end{enumerate}
\end{definition}

>From now on, we shall always assume that $L\subset X$ is a compact Lagrangian submanifold. Given a disk class $\beta\in\pi_2(X,L)$, an important topological invariant for $\beta$ is its Maslov index:
\begin{definition} \label{Maslov_index}
Let $L$ be a Lagrangian submanifold and $\beta\in\pi_2(X,L)$. Let $u:(\Delta,\partial\Delta)\to(X,L)$ be a representative of $\beta$. Then one may trivialize the symplectic vector bundle
$$u^*TX\cong\Delta\times V,$$
where $V$ is a symplectic vector space. Thus the subbundle $(\partial u)^* TL\subset(\partial u)^*TX$ induces the Gauss map
$$\partial\Delta\to U(n)/O(n)\to U(1)/O(1)\cong\sphere{1},$$
where $U(n)/O(n)$ parameterizes all Lagrangian subspaces in $V$. The degree of this map is called the Maslov index, which is independent of the choice of representative $u$.
\end{definition}

The Maslov index of $\beta$ is denoted by $\mu(\beta)\in\integer$.\footnote{It should be clear from the context whether $\mu$ refers to a Lagrangian fibration or the Maslov index.} $\mu(\beta)$ is important for open Gromov-Witten theory because it determines the expected dimension of the moduli space of holomorphic disks (see Equation \eqref{exp_dim}).

Now we are going to define the genus-zero open Gromov-Witten invariants. First of all, we have the notion of a pseudoholomorphic disk:
\begin{definition}
\begin{enumerate}
\item A pseudoholomorphic disk bounded by a Lagrangian $L\subset X$ is a smooth map $u:(\Delta,\partial\Delta)\to(X, L)$ such that $u$ is holomorphic with respect to the almost complex structure $J$, that is,
    $$(\partial u)\circ j=J\circ \partial u,$$
    where $j$ is the standard complex structure on the disk $\Delta\subset\cpx$.

\item The moduli space $\mathcal{M}^\circ_k(L,\beta)$ of pseudoholomorphic disks representing $\beta\in\pi_2(X,L)$ with $k$ ordered boundary marked points is defined as the quotient by $\Aut(\Delta)$ of the set of all pairs $(u,(p_i)_{i=0}^{k-1})$, where
    $$u:(\Delta,\partial\Delta)\to(X, L)$$
    is a pseudoholomorphic disk bounded by $L$ with homotopy class $[u]=\beta$, and $(p_i\in\partial\Delta:i=0,\ldots,k-1)$ is a sequence of boundary points respecting the cyclic order of $\partial\Delta$. For convenience the notation $(u,(p_i)_{i=0}^{k-1})$ is usually abbreviated as $u$.

\item The evaluation map $\mathrm{ev}_i:\mathcal{M}^\circ_k(L,\beta)\to L$ for $i=0,\ldots,k-1$ is defined as
    $$\mathrm{ev}_i([u,(p_i)_{i=0}^{k-1}]):=u(p_i).$$
\end{enumerate}
\end{definition}

$\mathcal{M}^\circ_k(L,\beta)$ has expected dimension
\begin{equation} \label{exp_dim}
\dim_{\textrm{virt}}(\mathcal{M}^\circ_k(L,\beta))=n+\mu(\beta)+k-3,
\end{equation}
where the shorthand `virt' stands for the word `virtual' (which refers to `virtual fundamental chain' discussed below).

To define open Gromov-Witten invariants, one requires an intersection theory on the moduli spaces. This involves various issues:\\

\noindent {\large 1. \textit{Compactification of moduli.}}

$\mathcal{M}^\circ_k(L,\beta)$ is non-compact in general, and one needs to compactify the moduli. Analogous to closed Gromov-Witten theory, this involves the concept of stable disks. A \textit{stable disk} bounded by a Lagrangian $L$ with $k$ ordered boundary marked points is a pair $(u,(p_i)_{i=0}^{k-1})$, where
$$u:(\Sigma,\partial\Sigma)\to(X,L)$$
is a pseudoholomorphic map whose domain $\Sigma$ is a `semi-stable' Riemann surface of genus-zero (which may have several disk and sphere components) with a non-empty connected boundary $\partial\Sigma$ and $(p_i\in\partial\Sigma)$ is a sequence of boundary points respecting the cyclic order of the boundary, which satisfies the stability condition: If a component $C$ of $\Sigma$ is contracted under $u$, then $C$ contains at least three marked or singular points of $\Sigma$.

A compactification of $\mathcal{M}^\circ_k(L,\beta)$ is then given by the moduli space of stable disks:
\begin{definition}[Definition 2.27 of \cite{FOOO_I}]
Let $L$ be a compact Lagrangian submanifold in $X$ and $\beta\in\pi_2(X,L)$. Then $\mathcal{M}_k(L,\beta)$ is defined to be the set of isomorphism classes of stable disks representing $\beta$ with $k$ ordered boundary marked points.  Two stable disks $(u,(p_i))$ and $(u',(p'_i))$ are isomorphic if the maps $u$ and $u'$ have the same domain $\Sigma$ and there exists $\phi\in\Aut(\Sigma)$ such that $u'=u\circ\phi$ and $\phi(p'_i)=p_i$.
\end{definition}

\begin{remark}
In the above definition we require that the ordering of marked points respects the cyclic order of $\partial\Sigma$.  In the terminologies and notations of \cite{FOOO_I}, the above moduli is called the main component and is denoted by $\mathcal{M}^{\textrm{main}}_k(\beta)$ instead.
\end{remark}

The moduli space $\mathcal{M}_k(L,\beta)$ has a Kuranishi structure (\cite{FOOO_I}). We briefly recall its construction in the following. First of all, let us recall the definition of a Kuranishi structure. See Appendix A1 of the book \cite{FOOO_II} for more details.

Let $\mathcal{M}$ be a compact metrizable space.
\begin{definition}[Definitions A1.1, A1.3, A1.5 in \cite{FOOO_II}]
A Kuranishi structure on $\mathcal{M}$ of (real) virtual dimension $d$ consists of the following data:
\begin{enumerate}
\item[(1)] For each point $\sigma\in\mathcal{M}$,
           \begin{enumerate}
           \item[(1.1)] A smooth manifold $V_\sigma$ (with boundary or corners) and a finite group $\Gamma_\sigma$ acting smoothly and effectively on $V_\sigma$.
           \item[(1.2)] A real vector space $E_\sigma$ on which $\Gamma_\sigma$ has a linear representation and such that $\textrm{dim }V_\sigma-\textrm{dim }E_\sigma=d$.
           \item[(1.3)] A $\Gamma_\sigma$-equivariant smooth map $s_\sigma:V_\sigma\to E_\sigma$.
           \item[(1.4)] A homeomorphism $\psi_\sigma$ from $s_\sigma^{-1}(0)/\Gamma_\sigma$ onto a neighborhood of $\sigma$ in $\mathcal{M}$.
           \end{enumerate}
\item[(2)] For each $\sigma\in\mathcal{M}$ and for each $\tau\in\textrm{Im }\psi_\sigma$,
           \begin{enumerate}
           \item[(2.1)] A $\Gamma_\tau$-invariant open subset $V_{\sigma\tau}\subset V_\tau$ containing $\psi_\tau^{-1}(\tau)$.\footnote{Here we regard $\psi_\tau$ as a map from $s_\tau^{-1}(0)$ to $\mathcal{M}$ by composing with the quotient map $V_\tau\to V_\tau/\Gamma_\tau$.}
           \item[(2.2)] A homomorphism $h_{\sigma\tau}:\Gamma_\tau\to\Gamma_\sigma$.
           \item[(2.3)] An $h_{\sigma\tau}$-equivariant embedding $\varphi_{\sigma\tau}:V_{\sigma\tau}\to V_\sigma$ and an injective $h_{\sigma\tau}$-equivariant bundle map $\hat\varphi_{\sigma\tau}:E_\tau\times V_{\sigma\tau}\to E_\sigma\times V_\sigma$ covering $\varphi_{\sigma\tau}$.
           \end{enumerate}
\end{enumerate}
Moreover, these data should satisfy the following conditions:
\begin{enumerate}
\item[(i)] $\hat\varphi_{\sigma\tau}\circ s_\tau=s_\sigma\circ\varphi_{\sigma\tau}$.\footnote{Here and after, we also regard $s_\sigma$ as a section $s_\sigma:V_\sigma\to E_\sigma\times V_\sigma$.}
\item[(ii)] $\psi_\tau=\psi_\sigma\circ\varphi_{\sigma\tau}$.
\item[(iii)] If $\xi\in\psi_\tau(s_\tau^{-1}(0)\cap V_{\sigma\tau}/\Gamma_\tau)$, then in a sufficiently small neighborhood of $\xi$,
    $$\varphi_{\sigma\tau}\circ\varphi_{\tau\xi}=\varphi_{\sigma\xi},\ \hat\varphi_{\sigma\tau}\circ\hat\varphi_{\tau\xi}=\hat\varphi_{\sigma\xi}.$$
\end{enumerate}
\end{definition}
The spaces $E_\sigma$ are called obstruction spaces (or obstruction bundles), the maps $\{s_\sigma:V_\sigma\to E_\sigma\}$ are called Kuranishi maps, and $(V_\sigma,E_\sigma,\Gamma_\sigma,s_\sigma,\psi_\sigma)$ is called a Kuranishi neighborhood of $\sigma\in\mathcal{M}$.

Now we come back to the setting of open Gromov-Witten invariants.  Let $(u,(p_i)_{i=0}^{k-1})$ represent a point $\sigma\in\mathcal{M}_k(L,\beta)$. Let $W^{1,p}(\Sigma;u^*(TX);L)$ be the space of sections $v$ of $u^*(TX)$ of $W^{1,p}$ class such that the restriction of $v$ to $\partial\Sigma$ lies in $u^*(TL)$, and $W^{0,p}(\Sigma;u^*(TX)\otimes\Lambda^{0,1})$ be the space of $u^*(TX)$-valued $(0,1)$-forms of $W^{0,p}$ class. Then consider the linearization of the Cauchy-Riemann operator $\bar\partial$
$$D_u\bar\partial:W^{1,p}(\Sigma;u^*(TX);L)\to W^{0,p}(\Sigma;u^*(TX)\otimes\Lambda^{0,1}).$$
This map is not always surjective (i.e. $u$ may not be regular), and this is why we need to introduce the notion of Kuranishi structures.  Nevertheless the cokernel of $D_u\bar\partial$ is finite-dimensional, and so we may choose a finite-dimensional subspace $E_\sigma$ of $W^{0,p}(\Sigma;u^*(TX)\otimes\Lambda^{0,1})$ such that
$$ W^{0,p}(\Sigma;u^*(TX)\otimes\Lambda^{0,1}) = E_\sigma \oplus D_u\bar\partial(W^{1,p}(\Sigma;u^*(TX);L)). $$
Define $\Gamma_\sigma$ to be the automorphism group of $(u,(p_i)_{i=0}^{k-1})$.

To construct $V_\sigma$, first let $V'_{\textrm{map},\sigma}$ be the space of solutions of the equation
$$D_u\bar\partial\, v=0\textrm{ mod }E_\sigma.$$
Now, the Lie algebra $\textrm{Lie(Aut}(\Sigma,(p_i)_{i=0}^{k-1}))$ of the automorphism group of $(\Sigma,(p_i)_{i=0}^{k-1})$ can naturally be embedded in $V'_{\textrm{map},\sigma}$. Take its complementary subspace and let $V_{\textrm{map},\sigma}$ be a neighborhood of its origin. On the other hand, let $V_{\textrm{domain},\sigma}$ be a neighborhood of the origin in the space of first order deformations of the domain curve $(\Sigma,(p_i)_{i=0}^{k-1})$. Now, $V_\sigma$ is given by $V_{\textrm{map},\sigma}\times V_{\textrm{domain},\sigma}$.

Next, one needs to prove that there exist a $\Gamma_\sigma$-equivariant smooth map $s_\sigma:V_\sigma\to E_\sigma$ and a family of smooth maps $u_{v,\zeta}:(\Sigma_\zeta,\partial\Sigma_\zeta)\to(X,L)$ for $(v,\zeta)\in V_\sigma$ such that $\bar\partial u_{v,\zeta}=s_\sigma(v,\zeta)$, and there is a map $\psi_\sigma$ mapping $s_\sigma^{-1}(0)/\Gamma_\sigma$ onto a neighborhood of $\sigma\in\mathcal{M}_k(L,\beta)$. The proofs of these are very technical and thus omitted.

This finishes the review of the construction of the Kuranishi structure on $\mathcal{M}_k(L,\beta)$.\\

\noindent {\large 2. \textit{Orientation.}}

According to Chapter 9 of \cite{FOOO_II}, $\mathcal{M}_k(L,\beta)$ is canonically oriented by fixing a relative spin structure on $L$. Thus the issue of orientation can be avoided by assuming that the Lagrangian $L$ is relatively spin, which we shall always do from now on. Indeed, in this paper, $L$ is always a torus, and so this assumption is satisfied.\\

\noindent {\large 3. \textit{Transversality.}}

An essential difficulty in Gromov-Witten theory is that in general, the moduli space $\mathcal{M}_k(L, \beta)$ is not of the expected dimension, which indicates the issue of non-transversality. To construct the virtual fundamental chains, a generic perturbation is needed to resolve this issue. This is done by Fukaya-Oh-Ohta-Ono \cite{FOOO_I, FOOO_II} using the so-called \textit{Kuranishi multi-sections}. We will not give the precise definition of multi-sections here. See Definitions A1.19, A1.21 in \cite{FOOO_II} for details. Roughly speaking, a multi-section $\mathfrak{s}$ is a system of multi-valued perturbations $\{s_\sigma':V_\sigma\to E_\sigma\}$ of the Kuranishi maps $\{s_\sigma:V_\sigma\to E_\sigma\}$ satisfying certain compatibility conditions. For a Kuranishi space with certain extra structures (this is the case for $\mathcal{M}_k(L,\beta)$), there exist multi-sections $\mathfrak{s}$ which are transversal to 0. Furthermore, suppose that $\mathcal{M}$ is oriented. Let $ev:\mathcal{M}\to Y$ be a strongly smooth map to a smooth manifold $Y$, i.e. a family of $\Gamma_\sigma$-invariant smooth maps $\{ev_\sigma:V_\sigma\to Y\}$ such that $ev_\sigma\circ\varphi_{\sigma\tau}=ev_\tau$ on $V_{\sigma\tau}$. Then, using these transversal multisections, one can define the virtual fundamental chain $ev_*([\mathcal{M}]^{\textrm{vir}})$ as a $\rat$-singular chain in $Y$ (Definition A1.28 in \cite{FOOO_II}).\\

\noindent {\large 4. \textit{Boundary strata of the moduli space.}}

Another difficulty in the theory is that in general $\mathcal{M}_k(L,\beta)$ has codimension-one boundary strata, which consist of stable disks whose domain $\Sigma$ has more than one disk components. Then intersection theory on $\mathcal{M}_k(L,\beta)$ is still not well-defined (which then depends on the choice of perturbation). Fortunately, for our purposes, it suffices to consider the case when $k=1$ and $\mu(\beta)=2$. In this case, the moduli space of stable disks has empty codimension-one boundary. Let us first introduce the concept of `minimal Maslov index':
\begin{definition}
The minimal Maslov index of a Lagrangian submanifold $L$ is defined as
$$\min\{\mu(\beta)\in\integer:\beta\neq0\textrm{ and }\mathcal{M}_0(L,\beta)\textrm{ is non-empty}\}.$$
\end{definition}
Then one has the following proposition.
\begin{prop} \label{no_boundary}
Let $L\subset X$ be a compact Lagrangian submanifold which has minimal Maslov index at least two, that is, $L$ does not bound any non-constant stable disks of Maslov index less than two. Also let $\beta\in\pi_2(X,L)$ be a class with $\mu(\beta)=2$. Then $\mathcal{M}_k(L,\beta)$ has no codimension-one boundary stratum.
\end{prop}
\begin{proof}
Let $u\in\beta$ be a stable disk belonging to a codimension-one boundary stratum of $\mathcal{M}_k(L,\beta)$. Then, by the results of \cite{FOOO_I, FOOO_II}, $u$ is a union of two stable disks $u_1$ and $u_2$. Since $L$ does not bound any non-constant stable disks of Maslov index less than two, $\mu([u_1]),\mu([u_2])\geq2$. But then $2=\mu([u])=\mu([u_1])+\mu([u_2])\geq4$ which is impossible.
\end{proof}

When $\mathcal{M}_k(L,\beta)$ is compact oriented without codimension-one boundary strata, the virtual fundamental chain is a \textit{cycle}. Hence, we have the virtual fundamental cycle $ev_*[\mathcal{M}_k(L,\beta)]\in H_d(L^k,\rat)$, where $d=\dim_{\textrm{virt}}\mathcal{M}_k(L,\beta)$. While one cannot do intersection theory on the moduli due to non-transversality, by introducing the virtual fundamental cycles, one may do intersection theory on $L^k$ instead. We can now define one-pointed genus-zero open Gromov-Witten invariants as follows.
\begin{definition} \label{open_GW}
Let $L\subset X$ be a compact relatively spin Lagrangian submanifold which has minimal Maslov index at least two.  For a class $\beta\in\pi_2(X,L)$ with $\mu(\beta)=2$, we define
$$n_\beta:=\mathrm{P.D.}(ev_*[\mathcal{M}_1(L,\beta)])\cup\mathrm{P.D.}([\mathrm{pt}])\in\rat,$$
where $[\mathrm{pt}]\in H_0(L,\rat)$ is the point class in $L$, $\mathrm{P.D.}$ denotes the Poincar\'e dual, and $\cup$ is the cup product on $H^{*}(L,\rat)$.
\end{definition}
The number $n_\beta$ is invariant under deformation of complex structure and under Lagrangian isotopy in which all Lagrangian submanifolds in the isotopy have minimal Maslov index at least two (see Remark 3.7 of \cite{auroux07}). Hence, $n_\beta$ is indeed a one-pointed genus-zero open Gromov-Witten invariant. Also, notice that the virtual dimension of $\mathcal{M}_1(L,\beta)$ equals $n+\mu(\beta)-2\geq n$ and it is equal to $n=\dim L$ only when $\mu(\beta)=2$. So we set $n_\beta=0$ if $\mu(\beta)\neq2$.

As in closed Gromov-Witten theory, a good way to pack the data of open Gromov-Witten invariants is to form a generating function. This idea has been used a lot in the physics literature.
\begin{definition} \label{gen_fcn}
Let $L\subset X$ be a compact relatively spin Lagrangian submanifold with minimal Maslov index at least two. For each $\lambda\in\pi_1(L)$, we have the generating function
\begin{equation} \label{F(L)}
\mathcal{F}(L,\lambda):=\sum_{\beta\in\pi_2(X,L)_\lambda} n_\beta\exp\left(-\int_\beta\omega\right)
\end{equation}
where
\begin{equation} \label{pi_lambda}
\pi_2(X,L)_\lambda:=\{\beta\in\pi_2(X,L):\partial\beta=\lambda\}.
\end{equation}
\end{definition}

Intuitively $\mathcal{F}(L,\lambda)$ is a weighted count of stable disks bounded by the loop $\lambda$ which pass through a generic point in $L$. In general, the above expression for $\mathcal{F}(L,\lambda)$ can be an infinite series, and one has to either take care of convergence issues or bypass the issues by considering the Novikov ring $\Lambda_0(\rat)$, as done by Fukaya-Oh-Ohta-Ono in their works.

\begin{definition} \label{Novikov}
The Novikov ring $\Lambda_0 (\rat)$ is the set of all formal series
$$\sum_{i=0}^\infty a_i T^{\lambda_i}$$
where $T$ is a formal variable, $a_i \in \rat$ and $\lambda_i \in \real_{\geq 0}$ such that $\lim_{i \to \infty} \lambda_i = \infty$.
\end{definition}

Then $\mathcal{F}(L,\lambda) \to \Lambda_0 (\rat)$ is defined by
$$\mathcal{F}(L,\lambda) = \sum_{\beta \in \pi_2(X,L)_\lambda} n_\beta T^{\int_\beta \omega}.$$
The evaluation $T = \conste^{-1}$ recovers Equation \eqref{F(L)}, if the corresponding series converges. In the rest of this paper, Equation \eqref{F(L)} will be used, while we keep in mind that we can bypass the convergence issues by invoking the Novikov ring $\Lambda_0 (\rat)$.

Now let's come back to the setting developed in Section \ref{Lag fib} and restrict to the situation that $L = F_r$ is a torus fiber of $\mu$ at $r \in B_0$.  To make sense of open Gromov-Witten invariants, we restrict our attention to those fibers with minimal Maslov index at least two:

\begin{definition} \label{wall}
Let $\mu: X \to B$ be a proper Lagrangian fibration under the setting of Section \ref{Lag fib}.
The subset $H \subset B_0$ which consists of all $r \in B_0$ such that $F_r$ has minimal Maslov index less than two is called the wall.
\end{definition}

We will see that when $\mu$ is the Gross fibration of a toric Calabi-Yau manifold $X$, the wall $H$ is indeed a hypersurface in $B_0$ (Proposition \ref{Maslov-zero disk}).  This explains why such a subset is called a `wall'.  Then for $r \in B_0 - H$ and $\beta \in \pi_2 (X, F_r)$, the open Gromov-Witten invariant $n_\beta$ is well-defined.

Under the setting of Lagrangian fibration, the generating functions given in Definition \ref{gen_fcn} pack together to give a function on the fiberwise homotopy loop space:

\begin{definition} \label{Lambda*}
Given a proper Lagrangian fibration $\mu: X \to B$ under the setting of Section \ref{Lag fib},
\begin{enumerate}
\item The fiberwise homotopy loop space $\Lambda^*$ is defined as the lattice bundle over $B_0$ whose fiber at $r$ is $\Lambda^*_r = \pi_1(F_r) \cong \integer^n$.

\item The generating function for $\mu$ is
$\mathcal{F}_X: \Lambda^*|_{B_0 - H} \to \real$ defined by
\begin{equation} \label{F_X}
\mathcal{F}_X(\lambda) := \mathcal{F}(F_r, \lambda) = \sum_{\beta \in \pi_2(X,F_r)_\lambda} n_\beta \exp\left(-\int_\beta \omega\right)
\end{equation}
where $r \in B_0 - H$ is the image of $\lambda$ under the bundle map $\Lambda^* \to B_0$.
\item Let $D \subset X$ be a codimension-two submanifold which has empty intersection with every fiber $F_r$ for $r \in B_0 - H$.  The corresponding generating function
$\mathcal{I}_D: \Lambda^*|_{B_0 - H} \to \real$ is defined by
\begin{equation} \label{I_D}
\mathcal{I}_D(\lambda) := \sum_{\beta \in \pi_2(X,F_r)_\lambda} (\beta \cdot D) n_\beta \exp\left(-\int_\beta \omega\right)
\end{equation}
where $r \in B_0 - H$ is the image of $\lambda$ under the bundle map $\Lambda^* \to B_0$; $\beta \cdot D$ is the intersection number between $\beta$ and $D$, which is well-defined because $D \cap F_r = \emptyset$.
\end{enumerate}
\end{definition}

Intuitively speaking, $\mathcal{I}_D$ is a weighted count of stable disks bounded by $\lambda$ emanating from $D$.  In the next section, we'll take Fourier transform of these generating functions to obtain the complex coordinates of the mirror.

\subsection{T-duality with corrections} \label{mir_construct}

Now we are ready for introducing a construction procedure which employs the SYZ program.  A family version of Fourier transform is needed, and it will be discussed in Section \ref{FT for family}.

>From now on, we will make the additional assumption that \emph{the base $B$ of the proper Lagrangian fibration $\mu$ is a polyhedral set in $\real^n$ with at least $n$ distinct codimension-one faces}, which are denoted as $\Psi_j$ for $j=0,\ldots,m-1$.  Moreover, \emph{the preimage
$$D_j := \mu^{-1}(\Psi_j)$$
of each $\Psi_j$ is assumed to be a codimension-two submanifold in $X$}.  An important example is given by a toric moment map $\mu$ on a toric manifold whose fan is strictly convex.  The Lagrangian fibrations constructed in Section \ref{add boundary} also satisfy these assumptions.

Our construction procedure is the following:
\begin{enumerate}
\item Take the dual torus bundle (see Definition \ref{dual torus bundle})
$$\check{\mu}: \check{X}_0 \to B_0$$
of $\mu:X_0 \to B_0$ .  $\check{X}_0$ has a canonical complex structure, and it is called the semi-flat mirror of $X$.

The semi-flat complex structure only captures the symplectic geometry of $X_0$, and it has to be corrected to capture additional information (which are the open Gromov-Witten invariants) carried by the symplectic geometry of $X$.

\item We have the generating functions $\mathcal{I}_{D_i}: \Lambda^*|_{B_0 - H} \to \real$ (Equation \eqref{I_D}) defined by
\begin{equation} \label{I_i}
\mathcal{I}_{D_i}(\lambda) := \sum_{\beta \in \pi_2(X,F_r)_\lambda} (\beta \cdot D_i) n_\beta \exp\left(-\int_\beta \omega\right).
\end{equation}
We'll abbreviate $\mathcal{I}_{D_i}$ as $\mathcal{I}_{i}$ for $i = 0, \ldots, m-1$.  Applying a family version of Fourier transform on each $\mathcal{I}_{i}$ (see Section \ref{FT for family}), one obtains $m$ holomorphic functions $\tilde{z}_i$ defined on $\check{\mu}^{-1}(B_0 - H) \subset \check{X}_0$.

These $\tilde{z}_i$ serve as the `corrected' holomorphic functions.  In general $\check{\mu}^{-1}(B_0 - H) \subset \check{X}_0$ consists of several connected components, and $\tilde{z}_i$ changes dramatically from one component to another component, and this is called the wall-crossing phenomenon.  This phenomenon will be studied in Section \ref{mirror} in the case of toric Calabi-Yau manifolds.

\item Let $R$ be the subring of holomorphic functions on $(\check{\mu})^{-1}(B_0 - H) \subset \check{X}_0$ generated by constant functions and $\{\tilde{z}^{\pm 1}_i\}_{i=0}^{m-1}$.  One defines $Y = \mathrm{Spec} R$.
\end{enumerate}

In Section \ref{mirror}, the above procedure will be carried out in details for toric Calabi-Yau manifolds.

\section{Fourier transform} \label{FT}
This is a short section on Fourier transform from the torus bundle aspect.  We start with the familiar Fourier transform for functions on tori.  Then we define fiberwise Fourier transform for functions on torus bundle.  Indeed Fourier transform discussed here fits into a more general framework for differential forms which gives the correspondence between Floer complex and the mirror Ext complex.  This will be discussed in a separate paper.

\subsection{Fourier transform on tori} \label{Mukai}
Let $\underline{\Lambda}$ be a lattice, and $V := \underline{\Lambda} \otimes \real$ be the corresponding real vector space.  Then $\mathbf{T} := V / \underline{\Lambda}$ is an $n$-dimensional torus.  We use $V^*$, $\underline{\Lambda}^*$ and $\mathbf{T}^*$ to denote the dual of $V$, $\underline{\Lambda}$ and $\mathbf{T}$ respectively.  There exists a unique $\mathbf{T}$-invariant volume form $\der \mathrm{Vol}$ on $\mathbf{T}$ such that $\int_{\mathbf{T}} \der \mathrm{Vol} = 1$.  One has the following well-known Fourier transform for complex-valued functions:
\begin{eqnarray*}
l^2(\underline{\Lambda}^*) & \cong & L^2(\mathbf{T})\\
f & \leftrightarrow & \check{f}
\end{eqnarray*}
where for each $\check{\theta} \in \mathbf{T}$,
\begin{equation} \label{FT1}
\check{f}(\check{\theta}) = \sum_{\lambda \in \underline{\Lambda}^*} f(\lambda) \conste^{2\pi \consti \pairing{\lambda}{\check{\theta}}}
\end{equation}
and for each $\lambda \in \underline{\Lambda}^*$,
\begin{equation} \label{FT2}
f(\lambda) = \int_{\mathbf{T}} \check{f}(\check{\theta}) \conste^{-2\pi \consti \pairing{\lambda}{\check{\theta}}} \der \mathrm{Vol}(\check{\theta}).
\end{equation}

The above familiar expression comes up naturally as follows.  $\underline{\Lambda}^* = \Hom(\mathbf{T},U(1))$ parametrizes all characters of the Abelian group $\mathbf{T}$, and conversely $\mathbf{T} = \Hom(\underline{\Lambda}^*, U(1))$ parametrizes all characters of $\underline{\Lambda}^*$.  Consider the following diagram:

$$
\begin{diagram} \label{FM diagram 1}
	\node[2]{\underline{\Lambda}^* \times \mathbf{T}} \arrow{sw,t}{\pi_1} \arrow{se,t}{\pi_2} \\
	\node{\underline{\Lambda}^*} \node[2]{\mathbf{T}}
\end{diagram}
$$

$\underline{\Lambda}^* \times \mathbf{T}$ admits the universal character function $\chi:\underline{\Lambda}^* \times \mathbf{T} \to U(1)$ defined by $$\chi(\lambda, \check{\theta}) := \conste ^{2\pi\consti\pairing{\lambda}{\check{\theta}}}$$
which has the property that $\chi|_{\{\lambda\} \times \mathbf{T}}$ is exactly the character function on $\mathbf{T}$ corresponding to $\lambda$, and $\chi|_{\underline{\Lambda}^* \times \{\check{\theta}\}}$ is the character function on $\underline{\Lambda}^*$ corresponding to $\check{\theta}$.  For a function $f: \underline{\Lambda}^* \to \cpx$, we have the following natural transformation
$$\check{f} := (\pi_2)_*\big((\pi_1^* f) \cdot \chi\big) $$
where $(\pi_2)_*$ denotes integration along fibers using the counting measure of $\underline{\Lambda}^*$.  This gives equation (\ref{FT1}).  Conversely, given a function $\check{f}:\mathbf{T} \to \cpx$, we have the inverse transform
$$f := (\pi_1)_*\big((\pi_2^* \check{f}) \cdot \chi^{-1} \big)$$
where $(\pi_1)_*$ denotes integration along fibers using the volume form $\der\mathrm{Vol}$ of $\mathbf{T}$.  This gives equation (\ref{FT2}).

We will mainly focus on the subspace $C^{\infty}(\mathbf{T})$ of smooth functions on $\mathbf{T}$.  Via Fourier transform, one has
$$C^{\textrm{r.d.}}(\underline{\Lambda}^*) \cong C^{\infty}(T)$$
where $C^{\textrm{r.d.}}(\underline{\Lambda}^*)$ consists of rapid-decay functions $f$ on $\underline{\Lambda}^*$.  $f$ decays rapidly means that for all $k \in \nat$,
$$||\lambda||^k f(\lambda) \to 0$$
as $\lambda \to \infty$.  Here we have chosen a linear metric on $V$ and
$$||\lambda|| := \sup_{|v| = 1} |\pairing{\lambda}{v}|.$$
The notion of rapid decay is independent of the choice of linear metric on $V$.

\subsection{Family version of Fourier transform} \label{FT for family}
Now let's consider Fourier transform for families of tori.  We turn back to the setting described in Section \ref{Lag fib}: $\mu: X_0 \to B_0$ is a Lagrangian torus bundle which is associated with the dual torus bundle $\check{\mu}: \check{X}_0 \to B_0$.  $\Lambda^*$ is the lattice bundle over $B_0$ defined in Definition \ref{Lambda*}.  Notice that $\check{\mu}$ always has the zero section (while $\mu$ may not have a Lagrangian section in general), which is essential in the definition of Fourier transform.

Analogous to Section \ref{Mukai}, we have the following commutative diagram

$$
\begin{diagram} \label{FM diagram 2}
	\node[2]{\Lambda^* \times_{B_0} \check{X}_0} \arrow{sw,t}{\pi_1} \arrow{se,t}{\pi_2} \\
	\node{\Lambda^*} \arrow{se} \node[2]{\check{X}_0} \arrow{sw} \\
	\node[2]{B_0}
\end{diagram}
$$

Each fiber $\check{F}_r$ parametrizes the characters of $\Lambda^*_r$, and vice versa.
$\Lambda^* \times_{B_0} \check{X}_0$ admits the universal character function $\chi:\Lambda^* \times_{B_0} \check{X}_0 \to U(1)$ defined as follows.  For each $r \in B_0$, $\lambda \in \Lambda^*_r$ and $\conn \in \check{F}_r$,
$$\chi(\lambda, \conn) := \mathrm{Hol}_{\conn}(\lambda)$$
which is the holonomy of the flat $U(1)$-connection $\conn$ over $F_r$ around the loop $\lambda$.  Thus we have the corresponding Fourier transform between functions on $\Lambda^*$ and $\check{X}_0$ similar to Section \ref{Mukai}:
$$C^{\textrm{r.d.}}(\Lambda^*) \cong C^{\infty}(\check{X}_0)$$
where $C^{\textrm{r.d.}}(\Lambda^*)$ consists of smooth functions $f$ on $\Lambda^*$ such that for each $r \in B_0$, $f|_{\Lambda^*_r}$ is a rapid-decay function.  Explicitly, $f \in C^{\textrm{r.d.}}(\Lambda^*)$ is transformed to
\begin{eqnarray*}
\check{f}: \check{X}_0 & \to & \cpx, \\
\check{f}(F_r, \conn) & = & \sum_{\lambda \in \Lambda^*_r} f(\lambda)  \mathrm{Hol}_\conn (\lambda).
\end{eqnarray*}

Again, even without the condition of rapid-decay, the above series is well-defined when it is considered to be valued in the Novikov ring $\Lambda_0 (\cpx)$ (Definition \ref{Novikov}).

In Section \ref{mir_construct}, this family version of Fourier transform is applied to the generating functions $\mathcal{I}_{i}: \Lambda^*|_{B_0 - H} \to \real$ (Equation \eqref{I_i}) to get the holomorphic functions
$\tilde{z}_i: \check{\mu}^{-1} (B_0 - H) \to \cpx$ given by
\begin{align*}
\tilde{z}_i &= \sum_{\lambda \in \pi_1(X',F_r)} \mathcal{I}_i(\lambda) \mathrm{Hol}_\conn (\lambda) \\
&= \sum_{\beta \in \pi_2(X',F_r)} (\beta \cdot D_i) n_\beta \exp\left(-\int_\beta \omega\right) \mathrm{Hol}_\conn (\partial \beta).
\end{align*}
In Section \ref{mirror} this is applied to toric Calabi-Yau manifolds to construct their mirrors.

\section{Mirror construction for toric Calabi-Yau manifolds} \label{torCY}
Throughout this section, we'll always take $X$ to be a toric Calabi-Yau manifold.  For such manifolds M. Gross \cite{gross_examples} and E. Goldstein \cite{goldstein} have independently written down a non-toric proper Lagrangian fibration $\mu:X \to B$ which falls in the setting of Section \ref{Lag fib}, and we'll give a brief review of them.  These Lagrangian fibrations have interior discriminant loci of codimension two, leading to the wall-crossing of genus-zero open Gromov-Witten invariants which will be discussed in Section \ref{counting disks}.  Section \ref{mirror} is the main subsection, in which we apply the procedure given in Section \ref{mir_construct} to construct the mirror $\check{X}$.

\subsection{Gross fibrations on toric Calabi-Yau manifolds} \label{review}

Let $N$ be a lattice of rank $n$ and $\Sigma$ be a simplicial fan supported in $N_\real := N \otimes \real$.  We'll always assume that $\Sigma$ is `strongly convex', which means that its support $|\Sigma|$ is convex and does not contain a whole line through $0 \in N_\real$.  The toric manifold associated to $\Sigma$ is denoted by $X = X_\Sigma$.  The primitive generators of rays of $\Sigma$ are denoted by $v_i$ for $i = 0, \ldots, m-1$, where $m \in \integer_{\geq n}$ is the number of these generators.  Each $v_i$ corresponds to an irreducible toric divisor which we'll denote by $\mathscr{D}_i$.  These notations are illustrated by the fan picture of $K_{\proj^1}$ shown in Figure \ref{KP1_fan}.

\begin{definition}
A toric manifold $X = X_\Sigma$ is Calabi-Yau if there exists a toric linear equivalence between its canonical divisor $K_X$ and the zero divisor.  In other words, there exists a dual lattice point $\underline{\nu} \in M$ such that
$$\pairing{\underline{\nu}}{v_i} = 1$$
for all $i = 0, \ldots, m-1$.
\end{definition}

>From now on we'll always assume $X$ is a toric Calabi-Yau manifold, whose holomorphic volume form can be explicitly written down (Proposition \ref{vol form}).  An important subclass of toric Calabi-Yau manifolds is given by total spaces of canonical line bundles of compact toric manifolds.

\begin{figure}[htp]
\caption{The fan picture of $K_{\proj^1}$.}
\label{KP1_fan}
\begin{center}
\includegraphics{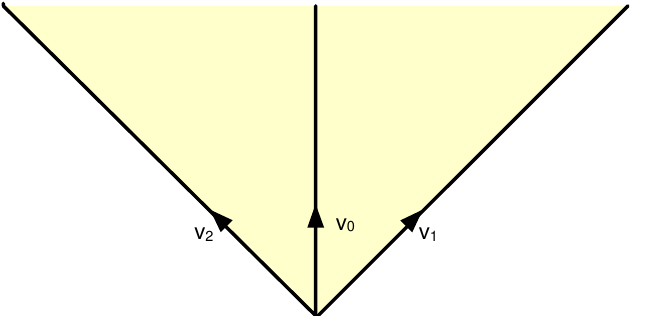}
\end{center}
\end{figure}

The following are some basic facts in toric geometry:
\begin{prop}[\cite{gross_examples}] \label{w}
The meromorphic function $w$ corresponding to $\underline{\nu} \in M$ is indeed holomorphic.  The corresponding divisor $(w)$ is $-K_X = \sum_{i=0}^{m-1} \mathscr{D}_i$.
\end{prop}
\begin{proof}
For each cone $C$ in $\Sigma$, let $v_{i_1}, \ldots, v_{i_n}$ be its primitive generators, which form a basis of $N$ because $C$ is simplicial by smoothness of $X_\Sigma$.  Let $\{\nu_j \in M\}_{j=1}^n$ be the dual basis, which corresponds to coordinate functions $\{\zeta_j\}_{j=1}^n$ on the affine piece $U_C$ corresponding to the cone $C$.  We have
$$ \underline{\nu} = \sum_{j=1}^n \nu_j $$
because $\pairing{\underline{\nu}}{v_{i_j}} = 1$ for all $j = 1, \ldots, n$.  Then
$$ w|_{U_C} = \prod_{j=1}^n \zeta_j $$
which is a holomorphic function whose zero divisor is exactly the sum of irreducible toric divisors of $U_C$.
\end{proof}

\begin{prop}[\cite{gross_examples}] \label{vol form}
Let $\{\nu_j\}_{j=0}^{n-1} \subset M$ be the dual basis of $\{v_0, \ldots, v_{n-1}\}$, and $\zeta_j$ be the meromorphic functions corresponding to $\nu_j$ for $j = 0, \ldots, n-1$.  Then
$$\der \zeta_0 \wedge \ldots \wedge \der \zeta_{n-1}$$
extends to a nowhere-zero holomorphic $n$-form $\Omega$ on $X$.
\end{prop}
\begin{proof}
$\der \zeta_0 \wedge \ldots \wedge \der \zeta_{n-1}$ defines a nowhere-zero holomorphic $n$-form on the affine piece corresponding to the cone $\real_{\geq 0}\langle v_0, \ldots, v_{n-1} \rangle$.  Let $C$ be an $n$-dimensional cone in $\Sigma$, $\{\nu'_j\}_{j=0}^{n-1} \subset M$ be a basis of $M$ which generates the dual cone of $C$, and let $\zeta'_0, \ldots, \zeta'_{n-1}$ be the corresponding coordinate functions on the affine piece $U_{C}$ corresponding to $C$.  Then
\begin{eqnarray*}
\der \zeta_0 \wedge \ldots \wedge \der \zeta_{n-1} &=& \zeta_0 \ldots \zeta_{n-1} \der \log\zeta_0 \wedge \ldots \der \log\zeta_{n-1}\\
&=& w \,\der \log\zeta_0 \wedge \ldots \der \log\zeta_{n-1} \\
&=& (\det A) w\, \der \log\zeta'_0 \wedge \ldots \der \log\zeta'_{n-1}\\
&=& (\det A) \der \zeta'_0 \wedge \ldots \wedge \der \zeta'_{n-1}
\end{eqnarray*}
where $A$ is the matrix such that $\nu_i = \sum_j A_{ij}\nu'_j$.
Since the fan $\Sigma$ is simplicial, $A \in \GL(n, \integer)$ and hence $\det A = \pm 1$.  Thus $\der \zeta_0 \wedge \ldots \wedge \der \zeta_{n-1}$ extends to a nowhere-zero holomorphic $n$-form on $U_{C}$.  This proves the proposition because $X$ is covered by affine pieces.
\end{proof}

\begin{remark}
In Proposition \ref{vol form} we have chosen the basis $\{v_i\}_{i=0}^{n-1} \subset N$.  If we take another basis $\{u_0, \ldots, u_{n-1}\} \subset N$ which spans some cone of $\Sigma$, then the same construction gives
$$\der \zeta'_0 \wedge \ldots \wedge \der \zeta'_{n-1} = \pm \der \zeta_0 \wedge \ldots \wedge \der \zeta_{n-1} $$
where $\zeta'_j$'s are coordinate functions corresponding to the dual basis of $\{u_i\}$.  The reason is that both $\{v_i\}$ and $\{u_i\}$ are basis of $N$, and thus the basis change belongs to $\GL (n, \integer)$, and its determinant is $\pm 1$.  Thus the holomorphic volume form, up to a sign, is independent of the choice of the cone and its basis.
\end{remark}

Let $\omega$ be a toric K\"ahler form on $\proj_{\Sigma}$ and $\mu_0: \proj_{\Sigma} \to P$ be the corresponding moment map, where $P$ is a polyhedral set defined by the system of inequalities
\begin{equation} \label{c_i}
\pairing{v_j}{\cdot} \geq c_j
\end{equation}
for $j = 1, \ldots, m$ and constants $c_j \in \real$ as shown in Figure \ref{KP1_poly}.
\begin{figure}[htp]
\caption{The toric moment map image of $K_{\proj^1}$.}
\begin{center}
\includegraphics{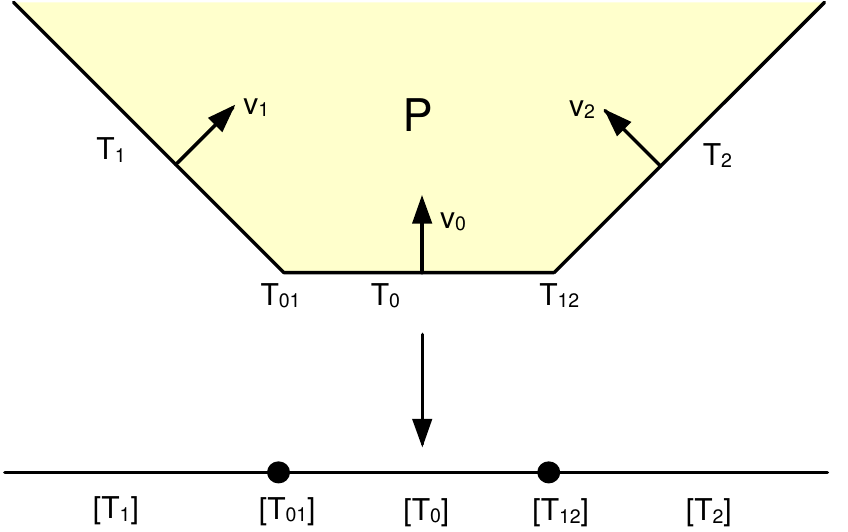}
\end{center}
\label{KP1_poly}
\end{figure}

The moment map corresponding to the action of the subtorus
$$ \torus{\perp \underline{\nu}} := N^{\perp \underline{\nu}}_\real / N^{\perp \underline{\nu}} \subset N_\real/N$$
on $X_\Sigma$ is
$$[\mu_0]: X_\Sigma \to M_\real / \real\langle \underline{\nu} \rangle$$
which is the composition of $\mu_0$ with the natural quotient map $M_\real \to M_\real / \real\langle \underline{\nu} \rangle$.

\begin{definition} \label{def_Gross_fib}
Fixing $K_2 > 0$, a Gross fibration is
$$
\begin{array}{rcccll}
\mu: X & \to & M_\real / \real \langle \underline{\nu} \rangle & \times & \real_{\geq -K_2^2} \\
x & \mapsto & \big([\mu_0(x)]& , & |w(x)- K_2|^2 - K_2^2 \big).&
\end{array}
$$
We'll always denote by $B$ the base $(M_\real / \real \langle \underline{\nu} \rangle) \times \real_{\geq -K_2^2}$.
\end{definition}

One has to justify the term `fibration' in the above definition, that is, $\mu: X \to B$ is surjective:
\begin{prop} \label{partialP}
Under the natural quotient $M_\real \to M_\real / \real\langle \underline{\nu} \rangle$, $\partial P$ is homeomorphic to $M_\real / \real\langle \underline{\nu} \rangle$.  Thus $\mu$ maps $X$ onto $B$.
\end{prop}
\begin{proof}
For any $\xi \in M_\real$, since $\pairing{v_j}{\underline{\nu}} = 1$ for all $j = 1, \ldots, m$, we may take $t \in \real$ sufficiently large such that $\xi + t \underline{\nu}$ satisfies the above system of inequalities
$$\pairing{v_j}{\xi + t \underline{\nu}} \geq c_j$$
and hence $\xi + t \underline{\nu} \in P$.  Let $t_0$ be the infimum among all such $t$.  Then $\xi + t_0 \underline{\nu}$ still satisfies all the above inequalities, and at least one of them becomes equality.  Hence $\xi + t_0 \underline{\nu} \in \partial P$, and such $t_0$ is unique.  Thus the quotient map gives a bijection between $\partial P$ and $M_\real / \real\langle \underline{\nu} \rangle$.  Moreover, the quotient map is continuous and maps open sets in $\partial P$ to open sets in $M_\real / \real\langle \underline{\nu} \rangle$, and hence it is indeed a homeomorphism.
\end{proof}

It is proved by Gross that the above fibration is special Lagrangian using techniques of symplectic reduction:
\begin{prop}[\cite{gross_examples}] \label{SLag fib}
With respect to the symplectic form $\omega$ and the holomorphic volume form $\Omega/(w-K_2)$ defined on $\mu^{-1}(B^{\mathrm{int}}) \subset X$,
$\mu$ is a special Lagrangian fibration, that is,  there exists $\theta_0 \in \real/2\pi\integer$ such that for every regular fiber $F$ of $\mu$, $\omega|_F = 0$ and $$\left.\mathrm{Re}\left( \frac{\conste^{2\pi\consti\theta_0}\,\Omega}{w-K_2}\right)\right|_{F} = 0.$$
\end{prop}

This gives a proper Lagrangian fibration $\mu: X \to B$ where the base $B$ is the upper half space, which is a manifold with corners in the sense of Definition \ref{mfd_corner}.

\subsection{Topological considerations for $X$} \label{top X}
In this section, we would like to write down the discriminant locus of $\mu$ and generators of $\pi_2(X,F)$, where $F \subset X$ is a regular fiber of $\mu$.

\subsubsection{The discriminant locus of $\mu$}
First, we give a notation for each face of the polyhedral set $P$:

\begin{definition}
For each index set $\emptyset \neq I \subset \{0,\ldots, m-1\}$ such that $\{v_i: i \in I\}$ generates some cone $C$ in $\Sigma$, let
\begin{equation} \label{T_I}
T_I := \big\{\xi \in P: \pairing{v_i}{\xi} = c_i \textrm{ for all } i \in I \big\}
\end{equation}
which is a codimension-$(|I|-1)$ face of $\partial P$.
\end{definition}
Via the homeomorphism described in Proposition \ref{partialP}, $[T_I]$ gives a stratification of $M_\real / \real\langle \underline{\nu} \rangle$.  This is demonstrated in Figure \ref{KP1_poly}.

We are now ready to describe the discriminant locus $\Gamma$ of $\mu$ (see Equation \eqref{Gamma} for the meaning of $\Gamma$):

\begin{prop}
Let $\mu$ be the Gross' fibration given in Definition \ref{def_Gross_fib}.
The discriminant locus of $\mu$ is
$$\Gamma = \partial B \cup \left(\left(\bigcup_{|I| = 2} [T_I]\right) \times \{0\}\right).$$
\end{prop}
\begin{proof}
The critical points of $\mu = ([\mu_0], |w-K_2|^2 - K_2^2)$ are where the differential of $[\mu_0]$ or that of $|w-K_2|^2 - K_2^2$ is not surjective.  The first case happens at the codimension-two toric strata of $X$, and the second case happens at the divisor defined by $w = K_2$.  The images under $\mu$ of these sets are $\left(\bigcup_{|I| = 2} [T_I]\right) \times \{0\}$ and $\partial B$ respectively.
\end{proof}

An illustration of the discriminant locus is given by Figure \ref{B}.
\begin{figure}[htp]
\caption{The base of the fibration $\mu:X \to B$ when $X = K_{\proj^1}$.  In this example, $\Gamma = \{r_1, r_2\} \cup \real \times \{-K_2\}$.}
\begin{center}
\includegraphics{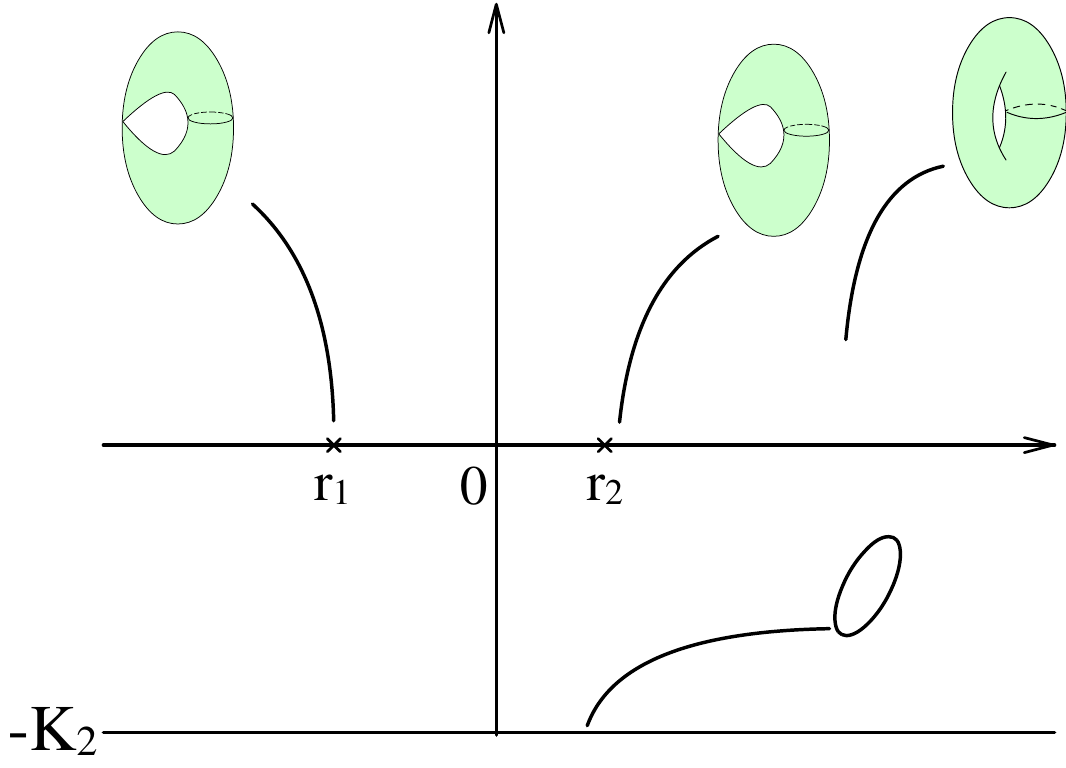}
\end{center}
\label{B}
\end{figure}

\subsubsection{Local trivialization}
As explained in Section \ref{Lag fib}, by removing the singular fibers, we obtain a torus bundle $\mu: X_0 \to B_0$ (see Equation \eqref{X_0} for the notations).  We now write down explicit local trivializations of this torus bundle, which will be used to make an explicit choice of generators of generators of $\pi_1(F)$ and $\pi_2(X,F)$.  Let
$$ U_i := B_0 - \bigcup_{k \neq i} \left([T_{k}] \times \{0\} \right) $$
for $i = 0, \ldots, m-1$, which are contractible open sets covering $B_0$, and hence $\mu^{-1} (U_i)$ can be trivialized.  Without loss of generality, we will always stick to the open set
$$U := U_0 = B_0 - \bigcup_{k \neq 0} \left([T_{k}] \times \{0\} \right) = \big\{(q_1, q_2) \in B_0: q_2 \not= 0 \textrm{ or } q_1 \in [T_{0}] \big\}.$$
\begin{prop} \label{T0}
$$[T_0] = \big\{q \in M_\real / \real\langle \underline{\nu} \rangle: \pairing{v'_j}{q} \geq c_j - c_0 \textrm{ for all } j = 1, \ldots, m-1\big\}$$
where
\begin{equation} \label{v'}
v'_j := v_j - v_0
\end{equation}
defines linear functions on $M_\real / \real\langle \underline{\nu} \rangle$ for $j = 1, \ldots, m-1$.
\end{prop}
\begin{proof}
$T_0$ consists of all $\xi \in M_\real$ satisfying
$$\left\{\begin{array}{rcl}
\pairing{v_j}{\xi} & \geq & c_j \textrm{ for all } j = 1, \ldots, m-1;\\ \pairing{v_0}{\xi} & = & c_0.
\end{array}\right.$$
which implies $\pairing{v'_j}{q} \geq c_j - c_0$ for all $j = 1, \ldots, m-1$.

Conversely, if $q = [\xi] \in M_\real / \real\langle \underline{\nu} \rangle$ satisfies $\pairing{v'_j}{q} \geq c_j - c_0$ for all $j = 1, \ldots, m-1$, then since $\pairing{\underline{\nu}}{v_0} = 1$, there exists $t \in \real$ such that $\pairing{v_0}{\xi + t \underline{\nu}} = c_0$.  And we still have $\pairing{v'_j}{\xi + t \underline{\nu}} \geq c_j - c_0$ for all $j = 1, \ldots, m-1$ because $\pairing{v'_j}{\underline{\nu}} = 0$.  Then $\pairing{v_j}{\xi} \geq c_j$ for all $j = 1, \ldots, m-1$.  Hence the preimage of $q$ contains $\xi + t \underline{\nu} \in T_0$.
\end{proof}

Using the above proposition, the open set $U = U_0$ can be written as
$$\big\{(q_1,q_2) \in B^{\mathrm{int}}: q_2 \neq 0 \textrm{ or } \pairing{v'_j}{q_1} > c_j - c_0 \textrm{ for all } j = 1, \ldots, m-1 \big\}$$
where $v_j'$ is defined by Equation \ref{v'}.  Now we are ready to write down an explicit coordinate system on $\mu^{-1} (U)$.

\begin{definition} \label{coordinates}
Let
$$\mathbf{T}_N /\mathbf{T}\langle v_0 \rangle := \frac{N_\real/\real\langle v_0 \rangle}{N/\integer\langle v_0 \rangle}.$$
We have the trivialization
$$
\mu^{-1} (U) \isomto U \times \left(\mathbf{T}_N /\mathbf{T}\langle v_0 \rangle\right) \times (\real/2\pi\integer)
$$
given as follows.  The first coordinate function is simply given by $\mu$.

To define the second coordinate function, let $\{\nu_0, \ldots, \nu_{n-1}\} \subset M$ be the dual basis
to $\{v_0, \ldots, v_{n-1}\} \subset N$.  Let $\zeta_j$ be the meromorphic functions corresponding to $\nu_j$ for $j = 1, \ldots, n-1$.  Then the second coordinate function is given by
$$\left(\frac{\arg \zeta_1}{2\pi}, \ldots, \frac{\arg \zeta_{n-1}}{2\pi}\right): \mu^{-1} (U) \to (\real/2\pi\integer)^{n-1} \cong \left(\mathbf{T}_N /\mathbf{T}\langle v_0 \rangle\right)$$
which is well-defined because for each $j = 1, \ldots, n-1$, $\nu_j \in M^{\perp v_0}$, implying $\zeta_j$ is a nowhere-zero holomorphic function on $\mu^{-1}(U)$.

The third coordinate is given by $\arg(w-K_2)$, which is well-defined because $w \neq K_2$ on $\mu^{-1}(U)$.
\end{definition}

\subsubsection{Explicit generators of $\pi_1(F_r)$ and $\pi_2(X,F_r)$} \label{gen_disks_X}
Now we define explicit generators of $\pi_1(F_r)$ and $\pi_2(X,F_r)$ for $r \in U$ in terms of the above coordinates.  For $r \in U$, one has
$$F_r \cong (\mathbf{T}_N /\mathbf{T}\langle v_0 \rangle) \times (\real/2\pi\integer)$$
and hence
$$\pi_1(F_r) \cong (N / \integer\langle v_0 \rangle) \times \integer$$
which has generators $\{\lambda_i\}_{i=0}^{n-1}$, where $\lambda_0 = (0,1)$ and $\lambda_i = ([v_i],0)$ for $i = 1, \ldots, n-1$.  This gives a basis of $\pi_1(F_r)$.


We take explicit generators of $\pi_2(X,F_{r})$ in the following way.  First we write down the generators for $\pi_2(X,\mathbf{T})$, which are well-known in toric geometry.  Then we fix $r_0 = (q_1, q_2) \in U$ with $q_2 > 0$, and identify $\pi_2(X,\mathbf{T})$ with $\pi_2(X,F_{r_0})$ by choosing a Lagrangian isotopy between $F_{r_0}$ and $\mathbf{T}$.  (The choice $q_2>0$ seems arbitrary at this moment, but it will be convenient for the purpose of describing holomorphic disks in Section \ref{counting disks}.) Finally $\pi_2(X,F_{r})$ for every $r \in B_0$ is identified with $\pi_2(X,F_{r_0})$ by using the trivialization of $\mu^{-1}(U) \cong U \times F_{r_0}$.  In this way we have fixed an identification $\pi_2(X,F_{r}) \cong \pi_2(X,\mathbf{T})$.  The details are given below.

\noindent {\large 1. \textit{Generators for $\pi_2(X,\mathbf{T})$.}}
Let $\mathbf{T} \subset X$ be a Lagrangian toric fiber, which can be identified with the torus $\mathbf{T}_N$.  By \cite{cho06}, $\pi_2(X,\mathbf{T})$ is generated by the basic disk classes $\beta^{\mathbf{T}}_j$ corresponding to primitive generators $v_j$ of a ray in $\Sigma$ for $j = 0, \ldots, m-1$.  One has
$$\partial \beta^{\mathbf{T}}_j = v_j \in N \cong \pi_1(\mathbf{T}_N).$$

These basic disk classes $\beta^{\mathbf{T}}_i$ can be expressed more explicitly in the following way.  We take the affine chart $U_C \cong \cpx^n$ corresponding to the cone $C = \langle v_0, \ldots, v_{n-1} \rangle$ in $\Sigma$.  Let
$$\mathbf{T}_{\rho} := \{(\zeta_0, \ldots, \zeta_{n-1}) \in \cpx^n: |\zeta_j| = \conste^{\rho_j} \textrm{ for } j=0,\ldots,n-1\} \subset X$$
be a toric fiber at $\rho = (\rho_0, \ldots, \rho_{n-1}) \in \real^n$.  For $i = 0, \ldots, n-1$, $\beta^{\mathbf{T}}_i$ is represented by the holomorphic disk $u: (\Delta, \partial\Delta) \to (U_C, \mathbf{T}_{\rho})$,
$$u(\zeta) = (\conste^{\rho_0}, \ldots, \conste^{\rho_{i-1}}, \conste^{\rho_{i}} \zeta, \conste^{\rho_{i+1}},  \ldots, \conste^{\rho_{n-1}}).$$
By taking other affine charts, other disk classes can be expressed in a similar way.  Figure \ref{disks in pi_2(X, T)} gives a drawing for $\beta^{\mathbf{T}}_i$ when $X = K_{\proj^1}$.  Since every disk class $\beta^{\mathbf{T}}_i$ intersects the anti-canonical divisor $\sum_{i=0}^{m-1} \mathscr{D}_i$ exactly once, it has Maslov index two (Maslov index is twice the intersection number \cite{cho06}).

\begin{figure}[htp]
\caption{The basic disk classes in $\pi_2(X,\mathbf{T})$ for a toric fiber $\mathbf{T}_{\rho}$ of $X = K_{\proj^1}$.}
\begin{center}
\includegraphics{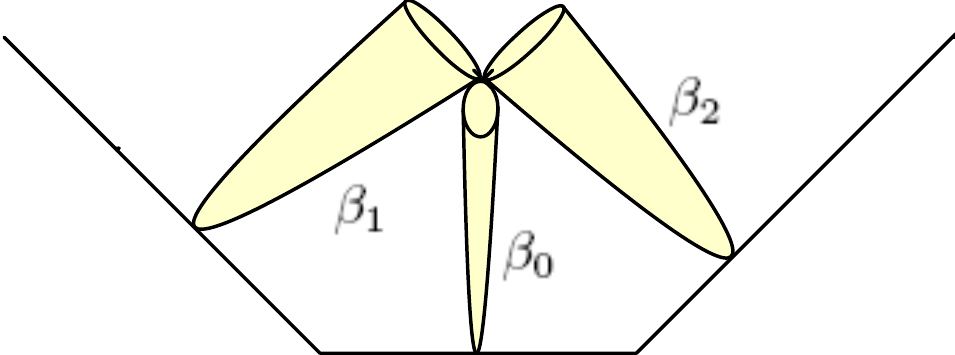}
\end{center}
\label{disks in pi_2(X, T)}
\end{figure}

\noindent {\large 2. \textit{Lagrangian isotopy between $F_{r_0}$ and $\mathbf{T}$.}}

Fix $r_0 = (q_1, q_2) \in B_0$ with $q_2 > 0$.  We have the following Lagrangian isotopy relating fibers of $\mu$ and Lagrangian toric fibers:
\begin{equation} \label{isotopy}
L_t := \{ x \in X: [\mu_0(x)] = q_1; |w(x) - t|^2 = K^2_2 + q_2 \}
\end{equation}
where $t \in [0, K_2]$.  $L_0$ is a Lagrangian toric fiber, and $L_{K_2} = F_{r_0}$.  (This is also true for $q_2 < 0$.  We fix $q_2 > 0$ for later purpose.)

The isotopy gives an identification between $\pi_2(X,F_{r_0})$ and $\pi_2(X,\mathbf{T})$.  Thus we may identify $\{\beta^{\mathbf{T}}_j\}_{j=0}^{m-1} \subset \pi_2(X,\mathbf{T})$ as a generating set of $\pi_2(X,F_{r_0})$, and we denote the corresponding disk classes by $\beta_j \in \pi_2(X,F_{r_0})$.  They are depicted in Figure \ref{disks in X}.

\begin{figure}[htp]
\caption{Disks generating $\pi_2(X,\mathbf{F_r})$ when $X = K_{\proj^1}$.}
\begin{center}
\includegraphics{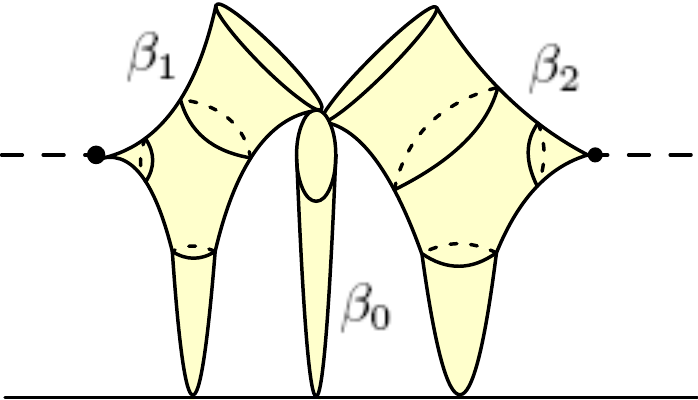}
\end{center}
\label{disks in X}
\end{figure}

Finally by the trivialization of $\mu^{-1}(U)$, every fiber $F_r$ at $r \in U$ is identified with $F_{r_0}$, and thus $\{\beta_j\}_{j=0}^{m-1}$ may be identified as a generating set of $\pi_2(X,F_r)$.

Notice that since Maslov index is invariant under Lagrangian isotopy, each $\beta_j \in \pi_2(X, F_r)$ remains to have Maslov index two.  We will need the following description for the boundary classes of $\beta_j$:

\begin{prop} \label{boundary of disks in X}
$$\partial \beta_j = \lambda_0 + \sum_{i=1}^{n-1} \pairing{\nu_i}{v_j}\lambda_i  \in (N / \integer\langle v_0 \rangle) \times \integer \cong \pi_1(F_r)$$
for all $j = 0, \ldots, m-1$, where $\{\nu_i\}_{i=0}^{n-1} \subset M$ is the dual basis of $\{v_i\}_{i=0}^{n-1} \subset N$.
\end{prop}

\begin{proof}
Under the identification
$$\mathbf{T}_N \stackrel{\cong}{\to} (\mathbf{T}_N /\mathbf{T}\langle v_0 \rangle) \times (\real/2\pi\integer)$$
where the last coordinate is given by $\pairing{\underline{\nu}}{\cdot}$, $\partial \beta^{\mathbf{T}}_j = v_j \in \pi_1(\mathbf{T}_N)$ is identified with
\begin{align*}
([v_j],1) = \left( \sum_{i=1}^{n-1} \pairing{\nu_i}{v_j} [v_i] , 1 \right)
&= \sum_{i=1}^{n-1} \pairing{\nu_i}{v_j}\lambda_i + \lambda_0 \\
&\in \pi_1\big((\mathbf{T}_N /\mathbf{T}\langle v_0 \rangle) \times (\real/2\pi\integer)\big)
\end{align*}
because $\pairing{\underline{\nu}}{v_j} = 1$ for all $j = 0, \ldots, m-1$.  Under the isotopy given in Equation \eqref{isotopy}, this relation is preserved.
\end{proof}

The following proposition gives the intersection numbers of the disk classes with various divisors:
\begin{prop} \label{intersection}
Let $r = (q_1, q_2) \in U$ with $q_2 \not= 0$, and $\beta_i \in \pi_2 (X, F_r)$ be the disk classes defined above.   Then
$$\beta_0 \cdot \mathscr{D}_j = 0$$
for all $j = 1, \ldots, m-1$;
$$\beta_i \cdot \mathscr{D}_j = \delta_{ij}$$
for all $i=1, \ldots, m-1$, $j = 1, \ldots, m-1$;
$$\beta_i \cdot D_0 = 1$$
for all $i=0, \ldots, m-1$, where
\begin{equation} \label{D_0}
D_0 := \{x \in X: w(x) = K_2 \}
\end{equation}
is the boundary divisor whose image under $\mu$ is $\partial B$,
and $\mathscr{D}_j$ are the irreducible toric divisors of $X$.
\end{prop}
\begin{proof}
We need to use the following topological fact: Let $\{L_t: t \in [0,1]\}$ be an isotopy between $L_0$ and $L_1$, and $\{S_t: t \in [0,1]\}$ be an isotopy between the cycles $S_0$ and $S_1$.  Suppose that for all $t \in [0,1]$, $L_t \cap S_t = \emptyset$.  Then for $\beta \in \pi_2 (X, L_0)$, one has the following equality of intersection numbers:
$$\beta \cdot S_0 = \beta' \cdot S_1$$
where $\beta' \in \pi_2 (X, L_1)$ corresponds to $\beta$ under the isotopy $L_t$.

First consider the case that $r = r_0$.  The first and second equalities follow by using the isotopy $L_t$ given by Equation \eqref{isotopy} and the equalities
$$\beta^{\mathbf{T}}_0 \cdot \mathscr{D}_j = 0$$
for $j = 1, \ldots, m-1$ and
$$\beta^{\mathbf{T}}_i \cdot \mathscr{D}_j = \delta_{ij}$$
for $i=1, \ldots, m-1$, $j = 1, \ldots, m-1$ respectively.

We also have the isotopy
$$ S_t = \{x \in X: w(x) = t \}$$
for $t = [0, K_2]$ between the anti-canonical divisor $-K_X = \sum_{l=0}^{m-1} \mathscr{D}_l$ and $D_0$.  One has $S_t \cap L_t = \emptyset$ for all $t$, and so
$$\beta_i \cdot D_0 = \beta^{\mathbf{T}}_i \cdot (-K_X) = 1$$
for all $i=0, \ldots, m-1$.

For general $r \in U$, since $U \cap \mathscr{D}_j = \emptyset$ for all $j = 1, \ldots, m-1$ and $U \cap D_0 = \emptyset$, the isotopy between $F_r$ and $F_{r_0}$ never intersect $D_0$ and $\mathscr{D}_j$ for all $j = 1, \ldots, m-1$.  Thus the above equalities of intersection numbers are preserved.
\end{proof}

\subsection{Toric modification} \label{add boundary}
Our idea of constructing the mirror $\check{X}$ is to construct coordinate functions of $\check{X}$ by counting holomorphic disks emanating from boundary divisors of $X$.  The problem is that in our situation, $B$ has only one codimension-one boundary, while we need $n$ coordinate functions!  To resolve this, one may consider counting of holomorphic cylinders, which requires the extra work of defining rigorously the corresponding Gromov-Witten invariants.  Another way is to consider a one-parameter family of toric K\"ahler manifolds with $n$ toric divisors such that $X$ appears as the limit of this family when $n-1$ of the toric divisors move to infinity.  We adopt the second approach in this paper.

Choose a basis of $N$ which generate a cone in $\Sigma$, say, the one given by $v_0, \ldots, v_{n-1}$.  Since this is simplicial, $\{v_j\}_{j=0}^{n-1}$ forms a basis of $N$.  We denote its dual basis by $\{\nu_j\}_{j=0}^{n-1} \subset M$ as before.

\begin{remark}
While all the constructions from now on depend on the choice of this basis, we will see in Proposition \ref{basis_change} that the mirrors resulted from different choices of basis differ simply by a coordinate change.
\end{remark}

We define the following modification to $X_{\Sigma}$:
\begin{definition} \label{modify_mu}
Fix $K_1 > 0$.
\begin{enumerate}
\item Let
$$P^{(K_1)} := \big\{\xi \in P: \pairing{v'_j}{\xi} \geq -K_1 \textrm{ for all } j = 1, \ldots, n-1 \big\} \subset P$$
where $v'_j := v_j - v_0$ for $j = 1, \ldots, n-1$.
$K_1$ is assumed to be sufficiently large such that none of the defining inequalities of $P^{(K_1)}$ is redundant.

\item Let $\Sigma^{(K_1)}$ be the inward normal fan to $P^{(K_1)}$, which consists of rays generated by $v_0, \ldots, v_{m-1}, v'_1, \ldots, v'_{n-1}$.

\item Let $X^{(K_1)}$ be the toric K\"ahler manifold corresponding to $P^{(K_1)}$ and
$$\mu^{(K_1)}_0: X^{(K_1)} \to P^{(K_1)}$$
be the moment map.
\end{enumerate}
\end{definition}

Notice that $X^{(K_1)}$ is no longer a Calabi-Yau manifold.  For notation simplicity, we always suppress the dependency on $K_1$ and write $\Sigma'$ in place of $\Sigma^{(K_1)}$ and $\mu'_0: X' \to P'$ in place of $\mu^{(K_1)}_0: X^{(K_1)} \to P^{(K_1)}$ in the rest of this paper.  The fan $\Sigma'$ and toric moment map image $P'$ of $X'$ are demonstrated in Figure \ref{fan for X'} and \ref{P'} respectively.

\begin{figure}[htp]
\caption{The fan $\Sigma'$ of $X'$ when $X = K_{\proj^1}$.}
\begin{center}
\includegraphics{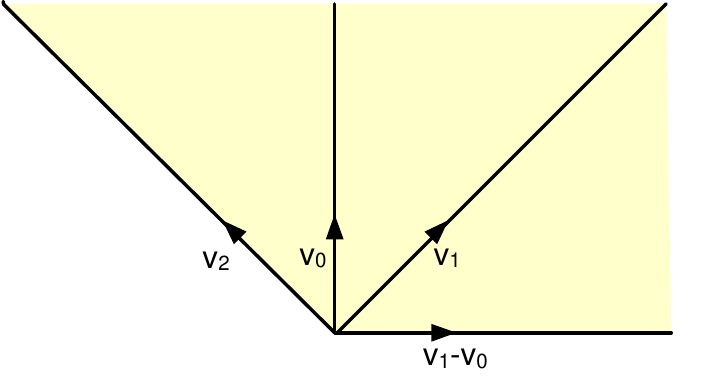}
\end{center}
\label{fan for X'}
\end{figure}

\begin{figure}[htp]
\caption{Toric moment map image $P'$ of $X'$ when $X = K_{\proj^1}$.}
\begin{center}
\includegraphics{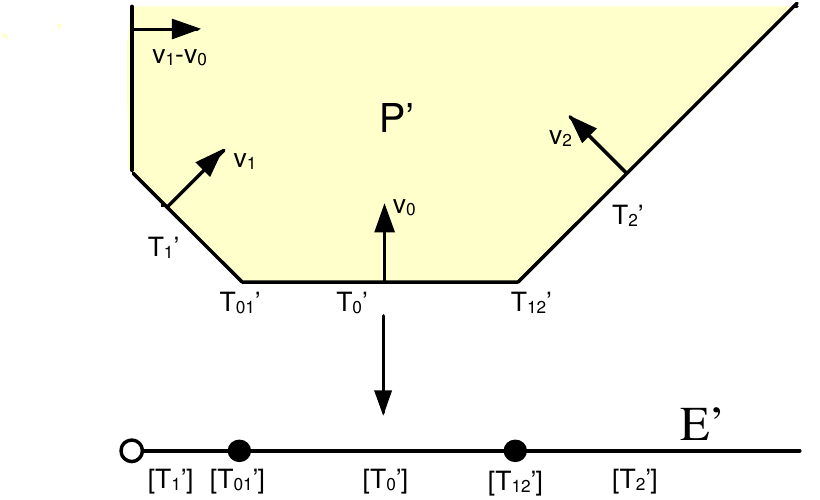}
\end{center}
\label{P'}
\end{figure}

Analogously, one has a special Lagrangian fibration on $X'$.  The definitions and propositions below are similar to that of Section \ref{review}, so we try to keep them brief.  The proofs are similar and thus omitted.

\begin{prop}
$\underline{\nu}$ corresponds to a holomorphic function $w'$ on $X'$, whose zero divisor is
$$(w') = \sum_{i=0}^{m-1} \mathscr{D}_i$$
where we denote each irreducible toric divisor corresponding to $v_i$ by $\mathscr{D}_i$, and that corresponding to $v'_j$ by $\mathscr{D}'_j$.  (Notice that $w'$ is non-zero on $\mathscr{D}'_j$ and so they do not appear in the above expression of $(w')$.)
\end{prop}
\begin{prop}
Let $\zeta_j$ be the meromorphic functions corresponding to $\nu_j$ for $j = 0, \ldots, n-1$.  Then
$$\Omega' := \der \zeta_0 \wedge \ldots \wedge \der \zeta_{n-1}$$
extends to a meromorphic $n$-form on $X'$ with
$$(\Omega') = - \sum_{j=1}^{n-1} \mathscr{D}'_j$$
where $\mathscr{D}'_j$ are the toric divisors corresponding to $v'_j$.
\end{prop}

\begin{definition}
Let
$$E^{(K_1)} := \big\{q \in (M_\real / \real \langle \underline{\nu} \rangle): \pairing{v'_j}{q} \geq -K_1 \textrm{ for all } j = 1, \ldots, n-1 \big\}$$
and
$$ B^{(K_1)} := E^{(K_1)} \times \real_{\geq -K_2}.$$
We have the fibration
$$
\begin{array}{rcc}
\mu^{(K_1)}: X^{(K_1)} & \to & B^{(K_1)} := E^{(K_1)} \times \real_{\geq -K_2} \\
x & \mapsto & \big([\mu^{(K_1)}_0(x)], |w'(x)- K_2|^2 - K^2_2\big).
\end{array}
$$
Again we'll suppress the dependency on $K_1$ for notation simplicity and use the notations $E$ and  $\mu': X' \to B'$ instead.
\end{definition}
Figure \ref{P'} gives an illustration to the notation $E'$, and Figure \ref{mu'} depicts the fibration $\mu'$ by an example.

\begin{prop} \label{partialP'}
Under the natural quotient $M_\real \to M_\real / \real\langle \underline{\nu} \rangle$, the image of $P'$ is $E$.
Indeed, this map give a homeomorphism between $$\big\{\xi \in \partial P': \pairing{v'_j}{\xi} > -K_1 \textrm{ for all } j = 1, \ldots, n-1 \big\}$$
and
$$E^{\mathrm{int}} = \{q \in M_\real / \real\langle \underline{\nu} \rangle: \pairing{v'_j}{q} > -K_1 \textrm{ for all } j = 1, \ldots, n-1\}.$$
As a consequence, $\mu':X' \to B'$ is onto.
\end{prop}

$B'$ is a manifold with corners with $n$ connected codimension-one boundary strata.  Using the notations given in the beginning of Section \ref{mir_construct}, the connected codimension-one boundary strata of $B'$ are
$$ \Psi_j := \{(q_1,q_2) \in B': \pairing{v'_j}{q_1} = -K_1 \} $$
for $j = 1, \ldots, n-1$ and
$$ \Psi_0 := \{(q_1,q_2) \in B': q_2 = -K_2 \}.$$
Moreover, the preimages
\begin{equation} \label{D_j}
D_j = (\mu')^{-1}(\Psi_j) \subset X'
\end{equation}
are divisors in $X'$.  Thus the assumptions needed in Section \ref{mir_construct} for the mirror construction are satisfied.  (Notice that these $D_j$ are NOT the toric divisors, which are denoted by $\mathscr{D}_i$ and $\mathscr{D}'_j$ instead.)

\begin{prop}
$\mu':X' \to  B'$ is a special Lagrangian fibration with respect to the toric K\"ahler form and the holomorphic volume form $\Omega' / (w' - K_2)$ defined on $X' - \bigcup_{j=0}^{n-1} D_j$.
\end{prop}

See Figure \ref{mu'} for an illustration of the above notations.  As $K_1 \to +\infty$, the divisors $D_j$ for $j=1, \ldots, n-1$ move to infinity and hence $\mu'$ tends to $\mu$ as Lagrangian fibrations.  We infer that the mirror of $\mu$ should appear as the limit of mirror of $\mu'$ as $K_1 \to +\infty$.  We will construct the mirror of $\mu'$ in the later sections.

\subsection{Topological considerations for $X'$} \label{top X'}
In this section we write down the discriminant locus of $\mu'$ and generators of $\pi_2(X',F)$, where $F$ is a fiber of $\mu'$.  This is similar to the discussion for $X$ in Section \ref{top X}, except that we have more disk classes due to the additional toric divisors.  The proofs to the propositions are similar to that in Section \ref{top X} and thus omitted.

\subsubsection{The discriminant locus of $\mu'$}
\begin{definition}
For each $\emptyset \neq I \subset \{0,\ldots, m-1\}$ such that $\{ v_i: i \in I \}$ generates some cone in $\Sigma'$, we define
$$T'_I := T_I \cap \big\{\xi \in P': \pairing{v'_j}{\xi} > -K_1 \textrm{ for all } j = 1, \ldots, n-1 \big\} $$
where $T_I$ is a face of $P$ given by Equation \eqref{T_I}.  $T_I'$ is a codimension-$(|I|-1)$ face of
$$\big\{\xi \in \partial P': \pairing{v'_j}{\xi} > -K_1 \textrm{ for all } j = 1, \ldots, n-1 \big\}.$$
\end{definition}

\begin{prop}
The discriminant locus of $\mu'$ is
$$\Gamma' := \left(\left(\bigcup_{|I| = 2} [T'_I]\right) \times \{0\}\right) \cup \partial B'.$$
\end{prop}

Figure \ref{mu'} gives an example for the base and discriminant locus of $\mu'$.

\begin{figure}[htp]
\caption{The base of $\mu'$ when $X = K_{\proj^1}$.  The discriminant locus is $\{r_1, r_2\} \cup \partial B'$.}
\begin{center}
\includegraphics{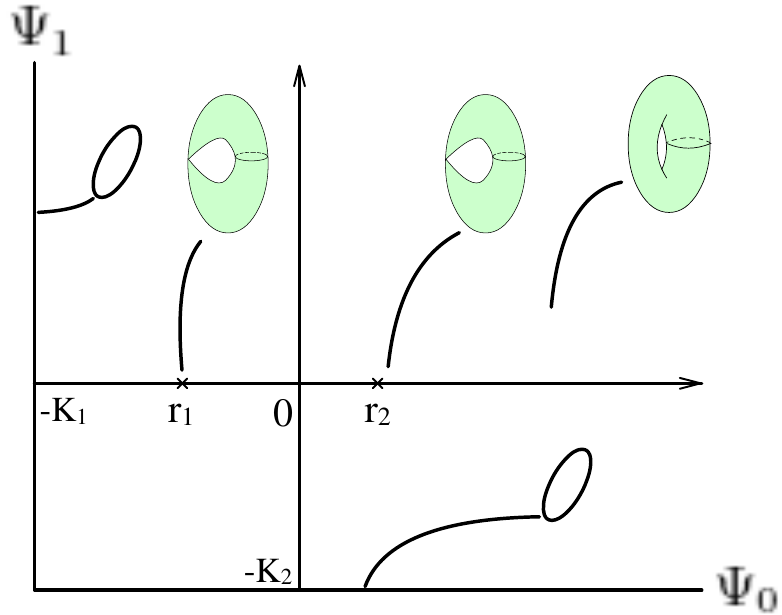}
\end{center}
\label{mu'}
\end{figure}

By removing the singular fibers of $\mu'$, we get a Lagrangian torus bundle $\mu': X'_0 \to B'_0$, where
\begin{eqnarray*}
B'_0 &:=& B' - \Gamma';\\
X'_0 &:=& (\mu')^{-1}(B'_0).
\end{eqnarray*}

\subsubsection{Local trivialization} \label{loc_trivial'}
We define
$$ U'_i := B'_0 - \bigcup_{k \neq i} \left([T'_{k}] \times \{0\} \right) $$
for $i = 0, \ldots, m-1$, so that $(\mu')^{-1} (U'_i)$ is trivialized.  Without loss of generality we stick to the trivialization over the open set
\begin{equation} \label{Uprime}
U' := U'_0 = B'_0 - \bigcup_{k \neq 0} \left([T'_{k}] \times \{0\} \right) = \big\{(q_1, q_2) \in B'_0: q_2 \not= 0 \textrm{ or } q_1 \in [T'_{0}] \big\}.
\end{equation}
Similar to Proposition \ref{T0}, one has
$$[T'_0] = \{q \in E^{\mathrm{int}}: \pairing{v'_j}{q} \geq c_j - c_0 \textrm{ for all } j = 1, \ldots, m-1\}.$$
Thus the open set $U' = U'_0$ can be written as
\begin{equation*}
\left\{
\begin{aligned}
&(q_1,q_2) \in E^{\mathrm{int}} \times \real_{> -K_2}:\\
&q_2 \neq 0 \textrm{ or } \pairing{v'_j}{q_1} > c_j - c_0 \textrm{ for all } j = 1, \ldots, m-1
\end{aligned}
\right\}.
\end{equation*}
Then the trivialization is explicitly written as
$$
(\mu')^{-1} (U') \stackrel{\cong}{\to} U' \times \left(\mathbf{T}_N /\mathbf{T}\langle v_0 \rangle\right) \times (\real/2\pi\integer)
$$
which is given in the same way as in Definition \ref{coordinates}.

\subsubsection{Explicit generators of $\pi_1(F_r)$ and $\pi_2(X,F_r)$} \label{gen_disks_X'}
For $r \in U'$, every $F_r$ is identified with the torus $\left(\mathbf{T}_N /\mathbf{T}\langle v_0 \rangle\right) \times (\real/2\pi\integer)$ via the above trivialization.  Then a basis of $\pi_1 (F_r)$ is given by $\{\lambda_i\}_{i=0}^{n-1}$, where $\lambda_0 = (0,1) \in N/\integer\langle v_0 \rangle \times \integer$ and $\lambda_i = ([v_i],0) \in N/\integer\langle v_0 \rangle \times \integer$ for $i = 1, \ldots, n-1$.

We use the same procedure as that given in Section \ref{gen_disks_X} to write down explicit generators of $\pi_2(X',F_r)$ for $r \in B'_0$.  First of all, $\pi_2(X', \mathbf{T})$ is generated by $\beta_i$ and $\beta'_j$ corresponding to $v_i$ and $v'_j$ respectively, where $i = 0, \ldots, m-1$ and $j = 1, \ldots, n-1$.  They are depicted in Figure \ref{disks in X'}.

\begin{figure}[htp]
\caption{Disks generating $\pi_2(X',T)$ for a regular moment-map fiber $T$ when $X = K_{\proj^1}$.}
\begin{center}
\includegraphics{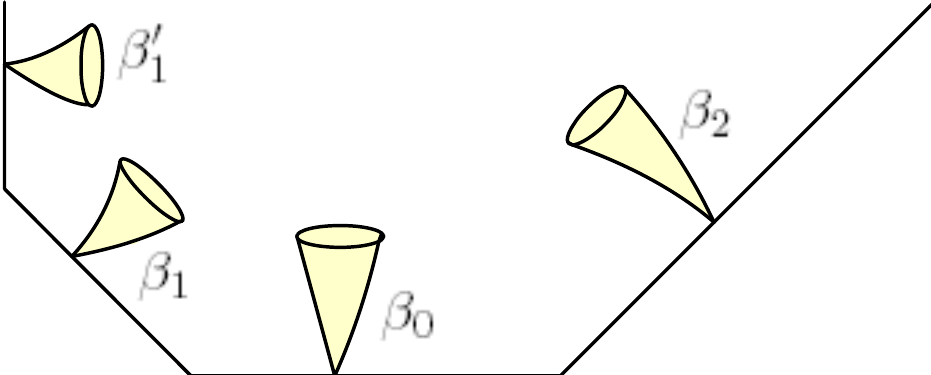}
\end{center}
\label{disks in X'}
\end{figure}

Then fixing a based point $r_0 = (q_1, q_2) \in U'$ with $q_2 > 0$, the isotopy
$$L_t := \{ x \in X: [\mu'_0(x)] = q_1; |w'(x) - t|^2 = K^2_2 + q_2 \} $$
between $F_{r_0}$ and a toric fiber $\mathbf{T}$ gives an identification $\pi_2(X', F_{r_0}) \cong \pi_2(X', \mathbf{T})$.  Finally the trivialization of $\mu^{-1}(U')$ gives an identification between $F_r$ and $F_{r_0}$ for any $r \in U'$.  Thus $\{\beta_i\}_{i=0}^{m-1} \cup \{\beta'_j\}_{j=1}^{n-1}$ can be regarded as a generating set of $\pi_2(X', F_{r})$.

\begin{prop} \label{boundary of disks in X'}
$$\partial \beta_j = \lambda_0 + \sum_{i=1}^{n-1} \pairing{\nu_i}{v_j}\lambda_i  \in (N / \integer\langle v_0 \rangle) \times \integer \cong \pi_1(F_r)$$
and
$$\partial \beta'_k = \lambda_k$$
for $j = 0, \ldots, m-1$ and $k = 1, \ldots, n-1$.
\end{prop}

\begin{prop} \label{intersection'}
Let $r = (q_1, q_2) \in U'$ with $q_2 \not= 0$, and $\beta_i \in \pi_2 (X', F_r)$ be the disk classes defined above. Then
\begin{align*}
\beta_0 \cdot \mathscr{D}_j &= 0 \textrm{ for all } j = 1, \ldots, m-1;\\
\beta_i \cdot \mathscr{D}_j &= \delta_{ij} \textrm{ for all } i,j=1, \ldots, m-1;\\
\beta_i \cdot \mathscr{D}'_k &= 0 \textrm{ for all } 0 = 1, \ldots, m-1 \textrm{ and } k = 1, \ldots, n-1;\\
\beta_i \cdot D_0 &= 1 \textrm{ for all } i=0, \ldots, m-1;\\
\beta'_l \cdot D_0 &= 0 \textrm{ for all } l=1,\ldots,n-1;\\
\beta'_l \cdot D_k &= \delta_{lk} \textrm{ for all } l,k=1,\ldots,n-1.
\end{align*}
\end{prop}

\subsection{Wall crossing phenomenon} \label{counting disks}
In Section \ref{open GW section} we give a review on the open Gromov-Witten invariants $n_\beta$ for a disk class $\beta \in \pi_2(X,L)$ bounded by a Lagrangian $L$ (Definition \ref{open_GW}).  We find that when $X$ is a toric Calabi-Yau manifold and $L = F_r$ is a Gross fiber, these invariants exhibit a wall-crossing phenomenon, which is the main topic of this section.  This is an application of the ideas and techniques introduced by Auroux \cite{auroux07, auroux09} to the case of toric Calabi-Yau manifolds.  The main results are Proposition \ref{disk counting + in X} and \ref{disk counting - in X}.  In Section \ref{blowup-flop} we will give methods to compute the open Gromov-Witten invariants.

Let's start with the Maslov index of disks (Definition \ref{Maslov_index}), which is important because it determines the expected dimension of the corresponding moduli (Equation \ref{exp_dim}).  The following lemma which appeared in \cite{auroux07} gives a formula for computing the Maslov index, which can be regarded as a generalization of the corresponding result by Cho-Oh \cite{cho06} for moment-map fibers in toric manifolds.

\begin{lemma}[Lemma 3.1 of \cite{auroux07}] \label{Maslov index}
Let $Y$ be a K\"ahler manifold of dimension $n$, $\sigma$ be a nowhere-zero meromorphic $n$-form on $Y$, and let $D$ denote its pole divisor.  If $L \subset Y - D$ is a compact oriented special Lagrangian submanifold with respect to $\sigma$, then for each $\beta \in \pi_2(Y, L)$,
$$\mu(\beta) =  2 \beta \cdot D.$$
\end{lemma}

Recall that the regular fibers $F_r$ of $\mu: X \to B$ are special Lagrangian with respect to $\Omega / (w - K_2)$ whose pole divisor is $D_0$ (see Equation \eqref{D_0} for the definition of $D_0$).  Using the above lemma, the Maslov index of $\beta \in \pi_2(X, F_r)$ is
$$\mu(\beta) = 2 \beta \cdot D_0.$$
Similarly $\mu': X' \to B'$ are special Lagrangian with respect to $\Omega' / (w' - K_2)$ whose pole divisor is $\sum_{j=0}^{n-1} D_j$.  Thus the Maslov index of $\beta \in \pi_2(X', F_r)$ is
\begin{equation} \label{Maslov_X'}
\mu(\beta) = 2 \beta \cdot \left(\sum_{j=0}^{n-1} D_j\right).
\end{equation}

>From this we deduce the following corollary:

\begin{corollary} \label{Maslov>=0}
For every $\beta \in \pi_2(X, F_r)$, if $\beta$ is represented by stable disks, then $\mu(\beta) \geq 0$.
\end{corollary}
\begin{proof}
>From the above formulae, it follows that the Maslov index of any holomorphic disks in $\beta \in \pi_2(X, F_r)$ or $\beta \in \pi_2(X', F_r)$ is non-negative.

Every stable disk consists of holomorphic disk components and holomorphic sphere components, and its Maslov index is the sum of Maslov indices of its disk components and two times Chern numbers of its sphere components.  The disk components have non-negative Maslov index as mentioned above.  Since $X$ is Calabi-Yau, every holomorphic sphere in $X$ has Chern number zero.  Thus the sum is non-negative.
\end{proof}

\subsubsection{Stable disks in $X$}
First we consider the toric Calabi-Yau $X$.  The lemma below gives an expression of the wall (see Definition \ref{wall}).

\begin{lemma} \label{Maslov-zero disk}
For $r = (q_1, q_2) \in B_0$, a Gross fiber $F_r$ bounds some non-constant stable disks of Maslov index zero in $X$ if and only if $q_2 = 0$.
\end{lemma}
\begin{proof}
Since $X$ is Calabi-Yau, sphere bubbles in a stable disk have Chern number zero and hence do not affect the Maslov index.  We can restrict our attention to a holomorphic disk $u: (\Delta,\partial\Delta) \to (X,F_r)$ whose Maslov index is zero.  By Lemma \ref{Maslov index}, $u$ has intersection number zero with the boundary divisor $D_0 = \{w = K_2\}$.  But since $u$ is holomorphic and $D_0$ is a complex submanifold, the multiplicity for each intersection point between them is positive.  This implies
$$\mathrm{Im}(u) \subset \mu^{-1}(B^{\mathrm{int}}).$$
Then $w \circ u - K_2$ is a nowhere-zero holomorphic function on the disk.  Moreover, $|w \circ u - K_2|$ is constant on $\partial\Delta$.  By applying maximum principle on $|w \circ u - K_2|$ and $|w \circ u - K_2|^{-1}$, $w \circ u$ must be constant with value $z_0$ in the circle
$$\big\{|z - K_2|^2 = K_2^2 + q_2\big\} \subset \cpx.$$
Unless $z_0 = 0$, $w^{-1}(z_0)$ is topologically $\real^{n-1} \times \torus{n-1}$, which contains no non-constant holomorphic disks whose boundary lies in $F_r \cap w^{-1}(z_0) \cong T^{n-1} \subset \real^{n-1} \times \torus{n-1}$.  Hence $z_0 = 0$, which implies $q_2 = 0$.
Conversely, if $q_2 = 0$, $F_r$ intersects a toric divisor along a (degenerate) moment map fiber, and hence bounds holomorphic disks which are part of the toric divisor.  They have Maslov index zero because they never intersect $D_0$.
\end{proof}

Combining the above lemma with Corollary \ref{Maslov>=0}, one has
\begin{corollary} \label{min_Maslov_2}
For $r = (q_1, q_2) \in B_0$ with $q_2 \not= 0$, $F_r$ has minimal Maslov index two.
\end{corollary}

Using the terminology introduced in Definition \ref{wall}, the wall is
$$H = M_\real / \real\langle \underline{\nu} \rangle \times \{0\}.$$
$B_0 - H$ consists of two connected components
\begin{equation} \label{B_+}
B_+ := M_\real / \real\langle \underline{\nu} \rangle \times (0, +\infty)
\end{equation}
and
\begin{equation} \label{B_-}
B_- := M_\real / \real\langle \underline{\nu} \rangle \times (-K_2, 0).
\end{equation}

For $r \in B_0 - H$, the fiber $F_r$ has minimal Maslov index two, and thus $n_\beta$ is well-defined for $\beta \in \pi_2(X,F_r)$ (see Section \ref{open GW section}).  There are two cases: $r \in B_+$ and $r \in B_-$.

\vspace{5pt}
\noindent 1. \textit{$r \in B_+$.}

One has the following lemma relating a Gross fiber $F_r$ to a Lagrangian toric fiber $\mathbf{T}$:

\begin{lemma} \label{Lag iso}
For $r \in B_+$, the Gross fiber $F_r$ is Lagrangian-isotopic to a Lagrangian toric fiber $\mathbf{T}$, and all the Lagrangians in this isotopy do not bound non-constant holomorphic disks of Maslov index zero.
\end{lemma}
\begin{proof}
Let $r = (q_1,q_2)$ with $q_2 > 0$.  The Lagrangian isotopy has already been given in Equation \eqref{isotopy}, which is
$$L_t := \{ x \in X: [\mu_0(x)] = q_1; |w(x) - t|^2 = K_2^2 + q_2 \} $$
where $t \in [0, K_2]$.  Since $q_2 > 0$, for each $t \in [0, K_2]$, $w$ is never zero on $L_t$.  By Lemma \ref{Maslov-zero disk}, $L_t$ does not bound non-constant holomorphic disks of Maslov index zero.
\end{proof}

Using the above lemma, one shows that the open Gromov-Witten invariants of $F_r$ when $r \in B_+$ are the same as that of $\mathbf{T}$:

\begin{prop} \label{disk counting + in X}
For $r \in B_+$ and $\beta \in \pi_2(X,F_r)$, let $\beta^{\mathbf{T}} \in \pi_2 (X,\mathbf{T}) \cong \pi_2(X,F_r)$ be the corresponding class under the isotopy given in Lemma \ref{Lag iso}.  Then
$$ n_\beta = n_{\beta^{\mathbf{T}}}.$$

$n_\beta \not= 0$ only when
$$\beta = \beta_j + \alpha$$
where $\alpha \in H_2(X)$ is represented by rational curves, and $\beta_j \in \pi_2(X,F_r)$ are the basic disk classes given in Section \ref{gen_disks_X}.  Moreover, $n_{\beta_j} = 1$ for all $j=0, \ldots, m-1$.
\end{prop}

\begin{proof}
It suffices to consider those $\beta \in \pi_2(X,F_r)$ with $\mu (\beta) = 2$, or otherwise $n_\beta = 0$ due to dimension reason.

The Lagrangian isotopy given in Lemma \ref{Lag iso} gives an identification between $\pi_2(X, F_r)$ and $\pi_2(X,\mathbf{T})$, where $\mathbf{T}$ is a regular fiber of $\mu_0$.  Moreover, since every Lagrangian in the isotopy has minimal Maslov index two, the isotopy gives a cobordism between $\mathcal{M}_1(F_r, \beta)$ and $\mathcal{M}_1(\mathbf{T}, \beta^{\mathbf{T}})$, where $\beta^{\mathbf{T}} \in \pi_2 (X,\mathbf{T})$ is the disk class corresponding to $\beta \in \pi_2(X,F_r)$ under the isotopy.  Hence $n_\beta$ keeps constant along this isotopy, which implies
$$ n_\beta = n_{\beta^{\mathbf{T}}}. $$

By dimension counting of the moduli space, $n_{\beta^{\mathbf{T}}}$ is non-zero only when $\beta^{\mathbf{T}}$ is of Maslov index two (see Equation \ref{exp_dim} and the explanation below Definition \ref{open_GW}).

Using Theorem 11.1 of \cite{FOOO1}, $\mathcal{M}_1(\mathbf{T}, \beta^{\mathbf{T}})$ is non-empty only when $\beta^{\mathbf{T}} = \beta_j + \alpha$, where $\alpha \in H_2(X)$ is represented by rational curves, and $\beta_j \in \pi_2 (X,\mathbf{T}) \cong \pi_2(X,F_r)$ are the basic disk classes given in Section \ref{gen_disks_X}.   For completeness we also give the reasoning here.  Let $u \in \mathcal{M}_1(\mathbf{T}, \beta^{\mathbf{T}})$ be a stable disk of Maslov index two.  $u$ is composed of holomorphic disk components and sphere components.  Since every holomorphic disk bounded by a toric fiber ${\mathbf{T}} \subset X$ must intersect some toric divisors, which implies that it has Maslov index at least two, $u$ can have only one disk component.  Moreover a holomorphic disk of Maslov index two must belong to a basic disk class $\beta_j$ \cite{cho06}.  Thus $\beta = [u]$ is of the form $\beta_j + \alpha$.

Moreover, by Cho-Oh's result \cite{cho06}, $n_{\beta_j} = 1$ for all $j=0, \ldots, m-1$.
\end{proof}

\vspace{5pt}
\noindent 2. \textit{$r \in B_-$.}

When $r \in B_-$, the open Gromov-Witten invariants behave differently compared to the case $r \in B_+$ (see Equation \ref{B_-} for the definition of $B_-$).  For $X = \cpx^n$, $n_\beta$ has been studied by Auroux \cite{auroux07,auroux09} (indeed he considered the cases $n = 2, 3$, but there is no essential difference for general $n$).  We give the detailed proof here for readers' convenience:

\begin{lemma}[\cite{auroux07}] \label{C^n}
When the toric Calabi-Yau manifold is $X = \cpx^n$ and $F_r \subset X$ is a Gross fiber at $r \in B_-$, we have
$$ n_\beta = \left \{
\begin{array}{ll}
1 & \textrm{ when } \beta = \beta_0;\\
0 & \textrm{ otherwise.}
\end{array} \right.
$$
\end{lemma}
\begin{proof}
Let $(\zeta_0, \ldots, \zeta_{n-1})$ be the standard complex coordinates of $\cpx^n$.  In these coordinates the Gross fibration is written as
$$ \mu = (|\zeta_0|^2 - |\zeta_1|^2, \ldots, |\zeta_{n-2}|^2 - |\zeta_{n-1}|^2, |\zeta_0 \ldots \zeta_{n-1} - K_2|^2 - K_2^2).$$

Due to dimension reason, $n_\beta = 0$ whenever $\mu(\beta) \not= 2$.  Thus it suffices to consider the case $\mu(\beta) = 2$.  Write $\beta = \sum_{i=0}^{n-1} k_i \beta_i$, where $\beta_i \in \pi_2 (X, F_r)$ are the basic disk classes defined in Section \ref{gen_disks_X}.  We claim that $k_0 = 1$ and $k_i = 0$ for all $i = 1, \ldots, n-1$ if the moduli space $\mathcal{M}_1(F_r, \beta)$ is non-empty.

Let $u$ be a stable disk in $\cpx^n$ representing $\beta$ with $\mu(\beta) = 2$.  Since $\cpx^n$ supports no non-constant holomorphic sphere, $u$ has no sphere component.  Also by Corollary \ref{min_Maslov_2}, $F_r$ has minimal Maslov index two, and so $u$ consists of only one disk component (see Proposition \ref{no_boundary}).  Thus $u$ is indeed a holomorphic map $\Delta \to \cpx^n$.

Since $q_2 < 0$, one has $|(\zeta_0\ldots\zeta_{n-1}) \circ u - K_2| < K_2$ on $\partial\Delta$.  By maximum principle this inequality holds on the whole disk $\Delta$.  In particular, $\zeta_0\ldots\zeta_{n-1}$ is never zero on $\Delta$, and so $u$ never hits the toric divisors $\mathscr{D}_i = \{\zeta_i = 0\}$ for $i = 0, \ldots, n-1$.  Thus $\beta \cdot \mathscr{D}_i = 0$ for all $i = 0, \ldots, n-1$.  By Proposition \ref{intersection}, $\pairing{\beta_0}{\mathscr{D}_j} = 0$ for all $j = 1, \ldots, n-1$, and $\pairing{\beta_i}{\mathscr{D}_j} = \delta_{ij}$ for $i = 1, \ldots, n-1$ and $j = 0, \ldots, n-1$.  Thus
$$\pairing{\beta}{\mathscr{D}_j} = k_j = 0$$
for $j = 1, \ldots, n-1.$
Thus $\beta = k_0 \beta_0$.  But $\mu(\beta) = k_0 \mu (\beta_0) = 2$ and $\mu (\beta_0) = 2$, and so $k_0 = 1$.

This proves that $n_\beta \not= 0$ only when $\beta = \beta_0$.  Now we prove that $n_{\beta_0} = 1$.  Since every fiber $F_r$ is Lagrangian isotopic to each other for $r \in B_-$ and the Lagrangian fibers have minimal Maslov index $2$, $n_{\beta_0}$ keeps constant as $r \in B_-$ varies.  Hence it suffices to consider $r = (0,q_2)$ for $q_2 < 0$, which means that $|\zeta_0| = |\zeta_1| = \ldots = |\zeta_{n-1}|$ for every $(\zeta_0, \ldots, \zeta_{n-1}) \in F_r$.

In the following we prove that for every $p \in F_r \subset (\cpx^\times)^n$, the preimage of $p$ under the evaluation map $\mathrm{ev}_0: \mathcal{M}_1(F_r, \beta_0) \to F_r$ is a singleton, and so $n_{\beta_0} = 1$.

Write $p = (p_0, \ldots, p_{n-1}) \in (\cpx^\times)^n$.  $p \in F_r$ implies that $|p_0| = |p_1| = \ldots = |p_{n-1}|$.  Consider the line
$$l := \{ (\zeta p_0, \zeta p_1, \ldots, \zeta p_{n-1}) \in (\cpx^\times)^n : \zeta \in \cpx^\times \}$$
spanned by $p$.  Then $w = \zeta_0 \ldots \zeta_{n-1}$ gives an $n$-to-one covering $l \to \cpx^\times$.  The disk
$$\Delta_{K_2} := \{\zeta \in \cpx: |\zeta - K_2| \leq (K_2^2 + q_2)^{1/2} \}$$
never intersects the negative real axis $\{\textrm{Re}(\zeta) \leq 0 \}$, and hence we may choose a branch to obtain a holomorphic map $\tilde{u}: \Delta_{K_2} \to l$ (There are $n$ such choices).  Moreover there is a unique choice such that $\tilde{u}\left(\prod_{j=0}^{n-1} p_j\right) = (p_0, \ldots, p_{n-1})$.  The image of $\partial \Delta_{K_2}$ under $\tilde{u}$ lies in $F_r$: Let $\zeta \in \partial \Delta_{K_2}$ and $z = \tilde{u}(\zeta)$.  Then $w(z) = \zeta$ satisfies $|w(z) - K_2|^2 = K_2^2 + q_2$.  Moreover $z \in l$, and so $|z_0| = |z_1| = \ldots = |z_{n-1}|$.  $\tilde{u}$ represents $\beta_0$ because it never intersects the toric divisors $\mathscr{D}_j$ for $j = 0, \ldots, n-1$ and it intersect with $D_0 = \{w = 0\}$ once.

The above proves that there exists a holomorphic disk representing $\beta_0$ such that its boundary passes through $p$.  In the following we prove that indeed this is unique.

Let $u \in \mathcal{M}_1(F_r, \beta_0)$ such that $\mathrm{ev}_0 (u) = p$.  By the above consideration $u$ is a holomorphic disk.  Since $\beta_0 \cdot \mathscr{D}_i = 0$, $u$ never hits the toric divisors $\{\zeta_i = 0\}$ for $i = 0, \ldots, n-1$, and hence $\zeta_i \circ u: \Delta \to \cpx$ are nowhere-zero holomorphic functions.  By applying maximum principle on $|\zeta_i/\zeta_1 \circ u|$ and $|\zeta_1/\zeta_i \circ u|^{-1}$ for each $i = 2, \ldots, n$, which has value $1$ on $\partial \Delta$, we infer that $u$ must lie on the complex line
$$\{ (\zeta, c_1\zeta, \ldots, c_{n-1} \zeta) \in (\cpx^\times)^n : \zeta \in \cpx^\times \}$$
where $|c_i| = 1$ are some constants for $i = 1, \ldots, n-1$.  Moreover, the line passes through $p$, and so this is the line $l$ defined above.

Consider the holomorphic map $w \circ u:\Delta \to \cpx^\times$.  Since $u$ has Maslov index two, it has intersection number one with the divisor $\{w - K_2 = 0\}$, implying that $w \circ u|_{\partial\Delta}$ winds around $K_2$ only once.  Hence $w \circ u$ gives a biholomorphism
$\Delta \stackrel{\cong}{\to} \Delta_{K_2}$ defined above.
One has $\tilde{u} \circ (w \circ u) = (\tilde{u} \circ w) \circ u = u$, where $\tilde{u}$ is the one-side inverse of $w$ defined above.  This means $u$ is the same as $\tilde{u}$ up to the biholomorphism $w \circ u$.  Thus $\tilde{u}$ is unique.
\end{proof}

Indeed the same statement holds for all toric Calabi-Yau manifolds:

\begin{prop} \label{disk counting - in X}
For $r \in B_-$ and $\beta \in \pi_2(X,F_r)$,
$$ n_\beta = \left \{
\begin{array}{ll}
1 & \textrm{ when } \beta = \beta_0;\\
0 & \textrm{ otherwise.}
\end{array} \right.
$$
\end{prop}

\begin{proof}
Due to dimension reason, $n_\beta = 0$ if $\mu(\beta) \not= 2$, and so it suffices to assume $\mu(\beta) = 2$.  Let $r = (q_1, q_2)$ with $q_2 < 0$.

First of all, one observes that when $r \in B_-$, every holomorphic disk $u:(\Delta,\partial\Delta) \to (X,F_r)$ has image
$$\Image (u) \subset S_- := \mu^{-1}(\{(q_1,q_2) \in B:q_2 < 0\}).$$
This is because $(w-K_2) \circ u$ defines a holomorphic function on $\Delta$.  Since $r \in B_-$, $|w-K_2|$ is constant with value less than $K_2$ on $F_r$.  By maximum principle, $|w-K_2| \circ u < K_2$.  This proves the observation.

Notice that $(S_-,F_r)$ is homeomorphic to $((\cpx^\times)^{n-1} \times \cpx, T)$, where
$$T = \{(\zeta_1, \ldots, \zeta_n) \in (\cpx^\times)^{n-1} \times \cpx: |\zeta_1| = \ldots = |\zeta_n| = c \} $$
for $c > 0$.  In particular, $\pi_2(S_-) = 0$ which implies that $S_-$ supports no non-constant holomorphic sphere.  Moreover, every non-constant holomorphic disk bounded by $F_r$ with image lying in $S_-$ must intersect $D_0$, and thus it has Maslov index at least two.

Now let $v \in \mathcal{M}_1(F_r, \beta)$ be a stable disk of Maslov index two, where $r \in B_-$.  By the above observation, each disk component of $v$ has Maslov index at least two, and so $v$ has only one disk component.

Moreover, the image of a non-constant holomorphic sphere $h: \cpx\proj^1 \to X$ does not intersect $S_-$: Consider $w \circ h$, which is a holomorphic function on $\cpx\proj^1$ and hence must be constant.  Thus image of $h$ lies in $w^{-1}(c)$ for some $c$.  But for $c \not= 0$, $w^{-1}(c)$ is $(\cpx^\times)^{n-1}$ which supports no non-constant holomorphic sphere.  Thus $c = 0$.  But $w$ is never zero on $S_-$, implying that $w^{-1}(0) \cap S_- = \emptyset$.

Thus $v$ does not have any sphere component, because any non-constant holomorphic sphere in $X$ never intersect its disk component.  This proves for all $\beta \in \pi_2(X, F_r)$, $\mathcal{M}_1(\beta, F_r)$ consists of holomorphic maps $u: (\Delta,\partial\Delta) \to (X,F_r)$, that is, neither disk nor sphere bubbling never occurs.

In particular, all elements in $\mathcal{M}_1(\beta, F_r)$ have images in $S_-$ and never intersect the toric divisors.  Writing $\beta = \sum_{i=0}^{m-1} k_i \beta_i$, one has
$$\pairing{\beta}{\mathscr{D}_j} = k_j = 0$$
(see Proposition \ref{intersection}).  Moreover, $\mu(\beta) = 2$ forces $k_0 = 1$.  Thus $\mathcal{M}_1(\beta, F_r)$, where $\beta$ has Maslov index two, is non-empty only when $\beta = \beta_0$.  Thus $n_\beta = 0$ whenever $\beta \not= \beta_0$.

Let $V = \cpx^n \hookrightarrow X$ be the complex coordinate chart corresponding to the cone $\langle v_0, \ldots, v_{n-1} \rangle$.  We have $F_r \subset S_0 \subset V$, and since $\beta_0 \cdot \mathscr{D} = 0$ for every toric divisor $\mathscr{D} \subset X$, any holomorphic disk representing $\beta_0$ in $X$ is indeed contained in $V$.  Thus
$$\mathcal{M}^X_1(\beta_0, F_r) \cong \mathcal{M}^V_1(\beta_0, F_r).$$
Then $n^X_{\beta_0} = n^V_{\beta_0}$, where the later has been proven to be $1$ in Lemma \ref{C^n}.
\end{proof}

>From the above propositions, one sees that $n_\beta$ for $\beta \in \pi_2(X, F_r)$ changes dramatically as $r$ crosses the wall $H$, and this is the so-called wall-crossing phenomenon.

\subsubsection{Stable disks in $X'$}
Now we consider open Gromov-Witten invariants of $X'$.  The statements are very similar, except that there are more disk classes due to the additional toric divisors.  The proofs are also very similar and thus omitted.

\begin{lemma}
For $r = (q_1, q_2) \in B'_0$, a fiber $F_r$ of $\mu'$ bounds some non-constant stable disks of Maslov index zero in $X'$ if and only if $q_2 = 0$.
\end{lemma}

Thus for every $r = (q_1, q_2) \in B_0$ with $q_2 \not= 0$, $F_r$ has minimal Maslov index two.  The wall (see Definition \ref{wall}) is
$$H' = E^{\mathrm{int}} \times \{0\}.$$
The two connected components of $B'_0 - H'$ are denoted by
$$B'_+ := E^{\mathrm{int}} \times (0, +\infty)$$
and
$$B'_- := E^{\mathrm{int}} \times (-K_2, 0)$$
respectively.  Again we have two cases to consider:

\vspace{5pt}
\noindent 1. \textit{$r \in B'_+$.}

\begin{lemma}
For $r \in B'_+$, the fiber $F_r$ is Lagrangian-isotopic to a Lagrangian toric fiber, and all the Lagrangians in this isotopy do not bound non-constant holomorphic disks of Maslov index zero.
\end{lemma}

\begin{prop} \label{disk counting +}
For $r \in B'_+$ and $\beta \in \pi_2(X',F_r)$, $n_\beta \not= 0$ only when
$$\beta = \beta'_k \textrm{ for } k=1, \ldots, n-1$$
or
$$\beta = \beta_j + \alpha \textrm{ for } j=0, \ldots, m-1$$
where $\alpha \in H_2 (X)$ is represented by rational curves of Chern number zero.  Moreover, $n_{\beta} = 1$ when $\beta = \beta_0, \ldots, \beta_{m-1}$ or $\beta'_1, \ldots, \beta'_{n-1}$.
\end{prop}

\vspace{5pt}
\noindent 2. \textit{$r \in B'_-$.}

\begin{prop} \label{disk counting -}
For $r \in B'_-$ and $\beta \in \pi_2(X',F_r)$,
$$ n_\beta = \left \{
\begin{array}{ll}
1 & \textrm{ when } \beta = \beta_0 \textrm{ or } \beta'_1, \ldots, \beta'_{n-1};\\
0 & \textrm{ otherwise.}
\end{array} \right.
$$
\end{prop}

These invariants contribute to the `quantum correction terms' of the complex structure of the mirror, as we will discuss in the next section.  In Section \ref{period}, we'll give two ways to compute these invariants: one is by relating to closed Gromov-Witten invariants, and one is by predictions from complex geometry of the mirror.

\subsection{Mirror construction} \label{mirror}
In this section, we use the procedure given in Section \ref{mir_construct} to construct the mirror $\check{X}$ of a Calabi-Yau $n$-fold $X$ with the Gross fibration $\mu: X \to B$.  The following is the main theorem:

\begin{theorem}
\label{mir_thm}
Let $\mu: X \to B$ be the Gross fibration over a toric Calabi-Yau $n$-fold $X$, and $\mu': X' \to B'$ be the modified fibration given by Definition \ref{modify_mu}.

\begin{enumerate}
\item Applying the construction procedure given in Section \ref{mir_construct} on the Lagrangian fibration $\mu': X' \to B'$, one obtains a complex manifold
\begin{equation}
Y = \left\{(u,v,z_1, \ldots, z_{n-1}) \in (\cpx^\times)^2 \times (\cpx^\times)^{n-1}: uv = G(z_1, \ldots, z_{n-1}) \right\}
\end{equation}
which admits a partial compactification
\begin{equation} \label{eq_mir}
\check{X} = \left\{(u,v,z_1, \ldots, z_{n-1}) \in \cpx^2 \times (\cpx^\times)^{n-1}: uv = G(z_1, \ldots, z_{n-1}) \right\}.
\end{equation}
Here $G$ is a polynomial given by
\begin{equation} \label{G}
G(z_1, \ldots, z_{n-1}) = (1+\delta_0) + \sum_{j=1}^{n-1} (1 + \delta_j) z_j + \sum_{i=n}^{m-1} (1 + \delta_i) q_{i-n+1} z^{v_i}
\end{equation}
The notations $\delta_j, q_{a}$ and $z^{v_i}$ appeared above are explained in the end of this theorem.

\item Let $H$ be the wall given in Definition \ref{wall}.  There exists a canonical map
$$\rho: \check{\mu}^{-1}(B_0 - H) \to \check{X}$$
such that the holomorphic volume form
\begin{equation}
\check{\Omega} := \mathrm{Res} \left( \frac{1}{uv - G(z_1, \ldots, z_{n-1})} \der\log z_1 \wedge \ldots \wedge \der\log z_{n-1} \wedge \der u \wedge \der v \right)
\end{equation}
defined on $\check{X} \subset \cpx^2 \times (\cpx^\times)^{n-1}$ is pulled back to the semi-flat holomorphic volume form (see Section \ref{semi-flat} below) on $\check{\mu}^{-1}(B_0 - H)$ under $\rho$.  In this sense the semi-flat holomorphic volume form extends to $\check{X}$.

\item Let $\mathcal{F}_X$ be the generating function given in Definition \ref{Lambda*}.  The Fourier transform of $\mathcal{F}_X$ (see Definition \ref{Lambda*}) is given by $\rho^* (C_0 u)$, where $C_0$ is some constant (defined by Equation \eqref{C_j}).  In this sense the Fourier transform of $\mathcal{F}_X$ extends to a function on $\check{X}$, which is called the superpotential.

\end{enumerate}

Explanation of the new notations $\delta_i$, $q_a$ and $z^{v_i}$ are as follows:

\begin{itemize}
\item $\delta_i$'s are constants defined by
\begin{equation} \label{corr_term}
\delta_i := \sum_{\alpha \not= 0} n_{\beta_i + \alpha} \exp\left(- \int_\alpha \omega \right)
\end{equation}
for $i = 0, \ldots, m-1$, in which the summation is over all $\alpha \in H_2 (X, \integer) - \{0\}$ represented by rational curves.  (The basic disk classes $\beta_i \in \pi_2 (X, F_r)$ are defined previously in Section \ref{gen_disks_X}.)

\item $z^{v_i}$ denotes the monomial
$$\prod_{j=1}^{n-1} z_j^{\pairing{\nu_j}{v_i}}$$
where $\{\nu_j\}_{j=0}^{n-1} \subset M$ is the dual basis of $\{v_j\}_{j=0}^{n-1} \subset N$.

\item For $a = 1, \ldots, m-n$, $q_a$ are K\"ahler parameters defined as follows.  Let $S_a \in H_2(X,\integer)$ be the classes defined by
\begin{equation} \label{theta}
S_a := \beta_{a+n-1} - \sum_{j=0}^{n-1} \pairing{\nu_j}{v_{a+n-1}} \beta_j
\end{equation}
Then $q_a := \exp (-\int_{S_a} \omega)$.
\end{itemize}
\end{theorem}

$\check{X}$ is the complex manifold mirror to $X$.  We need to check that the above expression \eqref{theta} of $S_a$ does define classes in $H_2(X,\integer)$:

\begin{prop} \label{gen_H_2}
$\{S_a\}_{a=1}^{m-n}$ is a generating subset of $H_2(X,\integer)$.
\end{prop}
\begin{proof}
One has the short exact sequence
$$ 0 \to H_2 (X) \to \pi_2 (X,\mathbf{T}) \to \pi_1 (\mathbf{T}) \to 0$$
where $\mathbf{T}$ is a Lagrangian toric fiber, and the second to last arrow is given by the boundary map $\partial$.  For $i = n, \ldots, m-1$,
\begin{align*}
\partial \left(\beta^{\mathbf{T}}_i - \sum_{j=0}^{n-1} \pairing{\nu_j}{v_i} \beta^{\mathbf{T}}_j\right)
&= \partial \beta^{\mathbf{T}}_i - \sum_{j=0}^{n-1} \pairing{\nu_j}{v_i} \partial \beta^{\mathbf{T}}_j \\
&= v_i - \sum_{j=0}^{n-1} \pairing{\nu_j}{v_i} v_j \\
&= v_i - v_i = 0
\end{align*}
where $\beta^{\mathbf{T}}_i$'s are the basic disk classes given in Section \ref{gen_disks_X}.  Thus
$$\beta^{\mathbf{T}}_i - \sum_{j=0}^{n-1} \pairing{\nu_j}{v_i} \beta^{\mathbf{T}}_j \in H_2 (X,\integer).$$
Moreover, they are linearly independent for $i = n, \ldots, m-1$, because $\beta^{\mathbf{T}}_i$'s are linearly independent.  But $H_2 (X,\integer) \cong \integer^{m-n}$, and so they form a basis of $H_2 (X,\integer)$.

$\beta_i$'s are identified with $\beta^{\mathbf{T}}_i$'s under the Lagrangian isotopy between $F_r$ and $\mathbf{T}$ given in Section \ref{gen_disks_X}.  Thus $\{S_a\}_{a=1}^{m-n}$ is a generating subset of $H_2(X,\integer)$.
\end{proof}

By the above proposition, $\delta_i$, and so $\check{X}$, can be expressed in terms of K\"ahler parameters $q_a$ and open GW invariants $n_\beta$.

While throughout the construction we have fixed a choice of ordered basis $\{v_i\}_{i=0}^{n-1}$ of $N$ which generates a cone of $\Sigma$,  in Proposition \ref{basis_change} we will see that another choice of the basis amounts to a coordinate change of the mirror.  In this sense the mirror $\check{X}$ is independent of choice of this ordered basis.

We now apply the construction procedure given in Section \ref{mir_construct} on the Lagrangian fibration $\mu':X' \to B'$ and prove Theorem \ref{mir_thm}.

\subsubsection{Semi-flat complex coordinates and semi-flat holomorphic volume form} \label{semi-flat}
First let's write down the semi-flat complex coordinates on the chart $(\check{\mu}')^{-1}(U') \subset \check{X}_0$, where $U' \subset B'$ is given in Equation \eqref{Uprime}, and $\check{\mu}': \check{X}'_0 \to B'_0$ is the dual torus bundle to $\mu': X'_0 \to B'_0$ (see Definition \ref{dual torus bundle}).

Fix a base point $r_0 \in U'$.  For each $r \in U'$, let $\lambda_i \subset \pi_1 (F_r)$ be the loop classes given in Section \ref{gen_disks_X'}.  Moreover define the cylinder classes $[h_i(r)] \in \pi_2((\mu')^{-1}(U'),F_{r_0},F_r)$ as follows.  Recall that we have the trivialization
$$(\mu')^{-1}(U') \cong U' \times \left(\mathbf{T}_N /\mathbf{T}\langle v_0 \rangle\right) \times (\real/2\pi\integer)$$
given in Section \ref{loc_trivial'}.  Let $\gamma:[0,1] \to U'$ be a path with $\gamma(0) = r_0$ and $\gamma(1) = r$.  For $j = 1, \ldots, n-1$,
$$h_j: [0,1] \times \real/\integer \to U' \times \left(\mathbf{T}_N /\mathbf{T}\langle v_0 \rangle\right) \times (\real/2\pi\integer)$$
is defined by
$$h_j(R,\Theta) := \left(\gamma(R), \frac{\Theta}{2\pi} [v_k], 0 \right)$$
and
$$h_0(R, \Theta) := (\gamma(R), 0, 2\pi\Theta).$$
The classes $[h_i(r)]$ is independent of the choice of $\gamma$.

Then the semi-flat complex coordinates $z_i$ on $(\check{\mu}')^{-1}(U')$ for $i = 0, \ldots, n-1$ are defined as
\begin{equation} \label{sf_coord}
z_i(F_r, \conn) := \exp(\rho_i + 2\pi \consti \check{\theta}_i)
\end{equation}
where $\conste^{2\pi\consti \check{\theta}_i} := \mathrm{Hol}_\conn (\lambda_i (r))$ and
$\rho_i := - \int_{[h_i(r)]} \omega$.

$\der z_1 \wedge \ldots \wedge \der z_{n-1} \wedge \der z_0$ defines a nowhere-zero holomorphic $n$-form on $(\check{\mu}')^{-1}(U')$, which is called the semi-flat holomorphic volume form.  It was shown in \cite{chan08} that this holomorphic volume form can be obtained by taking Fourier transform of $\exp (-\omega)$.  In this sense it encodes some symplectic information of $X$.

\subsubsection{Fourier transform of generating functions}
Next we correct the semi-flat complex structure by open Gromov-Witten invariants.  The corrected complex coordinate functions $\tilde{z}_i$ are expressed in terms of Fourier series whose coefficients are FOOO's disk-counting invariants of $X$.  The leading terms of these Fourier series give the original semi-flat complex coordinates.  In this sense the semi-flat complex structure is an approximation to the corrected complex structure.  The corrected coordinates have the following expressions:

\begin{prop} \label{FT_I}
Let $\mathcal{I}_i$ be the generating functions defined by Equation \eqref{I_i}.  The Fourier transforms of $\mathcal{I}_i$'s are holomorphic functions $\tilde{z}_i$ on $(\check{\mu}')^{-1}(B'_0 - H')$.  For $i = 1, \ldots, n-1$,
$$\tilde{z}_i = C'_i z_i$$
where $C'_i$ are constants defined by
\begin{equation} \label{C'_i}
C'_i = \exp\left(-\int_{\beta'_i (r_0)} \omega\right)>0.
\end{equation}
For $i = 0$,
$$
\tilde{z}_0 := \left\{
\begin{array}{ll}
C_0 z_0 & \textrm{ on } (\check{\mu}')^{-1}(B'_-) \\
z_0 g(z_1, \ldots, z_{n-1}) & \textrm{ on } (\check{\mu}')^{-1}(B'_+)
\end{array}
\right.
$$
where
$g(z_1, \ldots, z_{n-1})$ is the Laurent polynomial
\begin{equation} \label{g}
g(z_1, \ldots, z_{n-1}) := \sum_{i=0}^{m-1} C_i (1 + \delta_i) \prod_{j=1}^{n-1} z_j^{\pairing{\nu_j}{v_i}},
\end{equation}
$C_i$ are constants defined by
\begin{equation} \label{C_j}
C_i := \exp\left(-\int_{\beta_i (r_0)} \omega\right) > 0
\end{equation}
for $i = 0, \ldots, m-1$, and $\delta_i$ are constants previously defined by Equation \eqref{corr_term}.  Recall that $r_0$ is the based point chosen to define the semi-flat complex coordinates $z_0, \ldots, z_{n-1}$ in Section \ref{semi-flat}.
\end{prop}

\begin{proof}
The Fourier transform of each $\mathcal{I}_i$ is a complex-valued function $\tilde{z}_i$ on $(\check{\mu}')^{-1}(B'_0 - H')$ given by
\begin{align*}
\tilde{z}_i &= \sum_{\lambda \in \pi_1(X',F_r)} \mathcal{I}_i(\lambda) \mathrm{Hol}_\conn (\lambda) \\
&= \sum_{\beta \in \pi_2(X',F_r)} (\beta \cdot D_i) n_\beta \exp\left(-\int_\beta \omega\right) \mathrm{Hol}_\conn (\partial \beta).
\end{align*}

By Proposition \ref{disk counting +} and \ref{disk counting -}, $n_{\beta} = 0$ unless $\beta = \beta'_j$ for $j = 1, \ldots, n-1$ or $\beta = \beta_k + \alpha$ for $k = 0, \ldots, m-1$ and $\alpha \in H_2(X')$ represented by rational curves with Chern number zero, which implies that $\alpha \in H_2(X) \subset H_2(X')$.

Now consider $\tilde{z}_i$ for $i = 1, \ldots, n-1$.  Using Proposition \ref{intersection'}, $(\beta_k + \alpha) \cdot D_i = 0$ for all $k=0,\ldots,m-1$, $i=1,\ldots,n-1$ and $\alpha \in H_2(X)$.  Also $\beta'_j \cdot D_i = \delta_{ji}$.  Thus $\tilde{z}_i$ consists of only one term:
\begin{align*}
\tilde{z}_i &= \exp\left(-\int_{\beta'_i(r)} \omega\right) \mathrm{Hol}_\conn (\partial \beta'_i)
= \exp\left(-\int_{\beta'_i(r_0)} \omega - \int_{[h_i(r)]} \omega \right) \mathrm{Hol}_\conn (\lambda_i) \\
&= C'_i z_i.
\end{align*}

Now consider $\tilde{z}_0$.  One has $\beta'_j \cdot D_0 = 0$ and $(\beta_k + \alpha) \cdot D_0 = 1$.  There are two cases: When $r \in B'_-$,
$$n_\beta =
\left\{
\begin{array}{ll}
1 & \textrm{ for } \beta = \beta_0;\\
0 & \textrm{ otherwise.}
\end{array}
\right.
$$
In this case
\begin{align*}
\tilde{z}_0 &= \exp\left(-\int_{\beta_0 (r)} \omega\right)\mathrm{Hol}_\conn (\partial \beta_0)
= \exp\left(-\int_{\beta_0(r_0)} \omega - \int_{[h_0(r)]} \omega \right) \mathrm{Hol}_\conn (\lambda_0) \\
&= C_0 z_0.
\end{align*}

When $r \in B'_+$,
\begin{align*}
\tilde{z}_0
&= \sum_{j=0}^{m-1} \sum_{\alpha} n_{\beta_j(r) + \alpha} \exp\left(-\int_{\beta_j(r) + \alpha} \omega\right) \mathrm{Hol}_\conn (\partial \beta_j(r)) \\
&= \sum_{j=0}^{m-1} \left[ \left(\sum_{\alpha} n_{\beta_j(r) + \alpha} \exp\left( -\int_{\alpha} \omega \right) \right) \right. \\
&\hspace{20pt} \cdot \exp\left(-\int_{\beta_j (r_0)} - \int_{[h_0(r)]} - \sum_{i=1}^{n-1} \pairing{\nu_i}{v_j} \int_{[h_i(r)]} \right)\omega \\
&\hspace{20pt} \cdot \left. \mathrm{Hol}_\conn \left(\lambda_0 + \sum_{i=1}^{n-1} \pairing{\nu_i}{v_j}\lambda_i\right) \right]\\
&= \sum_{j=0}^{m-1} C_j \left(\sum_{\alpha} n_{\beta_j(r) + \alpha} \exp\left( -\int_{\alpha} \omega \right) \right) z_0 \prod_{i=1}^{n-1} z_i^{\pairing{\nu_i}{v_j}} \\
&= z_0 \sum_{j=0}^{m-1} C_j(1 + \delta_j) \prod_{i=1}^{n-1} z_i^{\pairing{\nu_i}{v_j}}.
\end{align*}
\end{proof}

\begin{remark} \label{approx_rem}
Let $r_0 \in U'$ be chosen such that $C_0$ equals to a specific constant, say, $2$.  One may also choose the toric K\"ahler form such that the symplectic sizes of the disks $\beta_i$ are very large for $i = 1, \ldots, m-1$, and so $C_i \ll 1$ (under this choice every non-zero two-cycle in $X'$ has large symplectic area, so this K\"ahler structure is said to be near the large K\"ahler limit).  According to the above expression of $\tilde{z}_0$, $C_0 z_0 = 2 z_0$ gives an approximation to $\tilde{z}_0$.  Thus the semi-flat complex coordinates of $\check{X}_0$ are approximations to the corrected complex coordinates.  The correction terms encode the enumerative data of $X'$.
\end{remark}

\subsubsection{The mirror $\check{X}$} \label{mirror equation}
Now we use $\{\tilde{z_i}^{\pm 1}\}_{i=0}^{n-1}$ derived from the previous subsection to generate a subring of functions on $\mu^{-1} (X' - H')$ and obtain
$$R = R_- \times_{R_0} R_+$$
where $R_- = R_+ := \cpx[z_0^{\pm 1},\ldots, z_{n-1}^{\pm 1}]$ and $R_0$ is the localization of $\cpx[z_0^{\pm 1},\ldots, z_{n-1}^{\pm 1}]$ at $g = \sum_{i=0}^{m-1} C_i (1 + \delta_i) z^{v_i}$ (see Equation \eqref{g}).  The gluing homomorphisms are given by $[Id]:R_- \to R_0$ and
\begin{eqnarray*}
R_+ &\to& R_0, \\
z_k &\mapsto& [z_k] \textrm{ for } k = 1, \ldots, n-1 \\
z_0 &\mapsto& \left[g^{-1} z_0\right].
\end{eqnarray*}
$\tilde{z_0}$ is identified with $u = (C_0 z_0, z_0 g) \in R$, and $\tilde{z_j}$ is identified with $(C'_j z_j, C'_j z_j) \in R$.
Setting
$$v := \left( C_0^{-1} z_0^{-1} g, z_0^{-1}\right) \in R$$
one has
$$R \cong \frac{\cpx[u^{\pm 1},v^{\pm 1},z_1^{\pm 1}, \ldots, z_{n-1}^{\pm 1}]}{\left\langle uv - g \right\rangle}.$$
Thus $\mathrm{Spec}(R)$ is geometrically realized as
$$Y = \left\{(u,v,z_1, \ldots, z_{n-1}) \in (\cpx^\times)^2 \times (\cpx^\times)^{n-1}: uv = g(z_1, \ldots, z_{n-1}) \right\}$$
which admits an obvious partial compactification
$$\check{X} = \left\{(u,v,z_1, \ldots, z_{n-1}) \in \cpx^2 \times (\cpx^\times)^{n-1}: uv = g(z_1, \ldots, z_{n-1}) \right\}.$$

One has the canonical map
\begin{equation} \label{rho_0}
\rho_0: \check{\mu}^{-1}(B_0 - H) \to \check{X}
\end{equation}
by setting
$$
u := \left\{
\begin{array}{ll}
C_0 z_0  & \textrm{ on } (\check{\mu}')^{-1}(B_-); \\
z_0 g & \textrm{ on } (\check{\mu}')^{-1}(B_+)
.
\end{array}
\right.
$$
and
$$
v := \left\{
\begin{array}{ll}
C_0^{-1} z_0^{-1} g & \textrm{ on } (\check{\mu}')^{-1}(B_-); \\
z_0^{-1} & \textrm{ on } (\check{\mu}')^{-1}(B_+)
.
\end{array}
\right.
$$

By a change of coordinates, the defining equation of $\check{X}$ can be transformed to the form appeared in Theorem \ref{mir_thm}:
\begin{prop}
By a coordinate change on $\cpx^2 \times (\cpx^\times)^{n-1}$, the defining equation
$$uv = \sum_{i=0}^{m-1} C_i (1 + \delta_i) z^{v_i}$$
can be transformed to
$$uv = (1 + \delta_0) + \sum_{j=1}^{n-1} (1 + \delta_j) z_j + \sum_{i=n}^{m-1} (1 + \delta_i) q_{i-n+1} z^{v_i}$$
where $C_i$'s are the constants defined by Equation \eqref{C_j}.
\end{prop}

\begin{proof}
Consider the coordinate change
$$ \hat{z}_j = \frac{C_j}{C_0} z_j $$
for $j = 0, \ldots, n-1$ on $(\cpx^\times)^{n-1}$.  Recall that $z^{v_i}$ denotes the monomial $\prod_{j=1}^{n-1} z_j^{\pairing{\nu_j}{v_i}}$, where $\{\nu_j\}_{j=0}^{n-1}$ is the dual basis to $\{v_j\}_{j=0}^{n-1}$.  Thus the $i=0$ term in the original equation is simply $C_0 (1 + \delta_0) z^{v_0} = C_0 (1 + \delta_0)$.

For $i = 1,\ldots,m-1$,
\begin{align*}
C_i z^{v_i} &= C_i \hat{z}^{v_i} \prod_{j=0}^{n-1} \left( \frac{C_0}{C_j} \right)^{\pairing{\nu_j}{v_i}} \\
&= C_0 C_i \hat{z}^{v_i} \left( \prod_{j=0}^{n-1} C_j^{\pairing{\nu_j}{v_i}} \right)^{-1}.
\end{align*}
The last equality in the above follows from the equality
$$ \sum_{j=0}^{n-1} \pairing{\nu_j}{v_i} = \pairing{\underline{\nu}}{v_i} = 1. $$

Thus for $i = 1, \ldots, n-1$,
$$C_i z^{v_i} = C_0 \hat{z}^{v_i}.$$

For i = $n, \ldots, m-1$,
$$C_i \left( \prod_{j=0}^{n-1} C_j^{\pairing{\nu_j}{v_i}} \right)^{-1}$$
is $\exp (-A_{i-n+1})$, where $A_{i-n+1}$ is the symplectic area of
$$S_{i-n+1} = \beta_{i} - \sum_{j=0}^{n-1} \pairing{\nu_j}{v_{i}} \beta_j.$$
Thus it equals to $q_{i-n+1}$.

Now set $\hat{u} = u/C_0$, the equation
$$uv = \sum_{i=0}^{m-1} C_i (1 + \delta_i) z^{v_i}$$
is transformed to
$$\hat{u}v = (1 + \delta_0) + \sum_{j=1}^{n-1} (1 + \delta_j) \hat{z}_j +  \sum_{i=n}^{m-1} (1 + \delta_i) q_{i-n+1} \hat{z}^{v_i}.$$
\end{proof}

This proves part (1) of Theorem \ref{mir_thm} that the construction procedure given in Section \ref{mir_construct} produces the mirror as stated.

\begin{remark}
In our case we have an obvious candidate serving as the partial compactification.  In general the technique of toric degenerations developed by Gross-Siebert \cite{gross07, gross08} is needed to ensure the existence of compactification.
\end{remark}

Notice that the defining equation of $\check{X}$ is independent of the parameter $K_1$ used to define the modification $X'$ in Section \ref{add boundary}, while the toric Calabi-Yau manifold $X$ appears as the limit of $X'$ as $K_1 \to \infty$.  Thus the mirror manifold of $X$ is also taken to be $\check{X}$.

\begin{remark}
Hori-Iqbal-Vafa \cite{HIV00} has written down the mirror of a toric Calabi-Yau manifold $X$ as
$$uv = 1 + \sum_{j=1}^{n-1} z_j + \sum_{i=n}^{m-1} q_{i-n+1} z^{v_i}$$
by physical considerations.  They realize that the above equation needs to be `quantum corrected', but they did not write down the correction in terms of the symplectic geometry of $X$.  From the SYZ consideration, now we see that the corrections can be expressed in terms of open Gromov-Witten invariants of $X$ (which are the factors $(1 + \delta_i)$).
\end{remark}

Composing the canonical map $\rho_0$ with the coordinate changes given above, one obtains a map
\begin{equation} \label{rho}
\rho: \check{\mu}^{-1}(B_0 - H) \to \check{X}
\end{equation}
where
$$
u := \left\{
\begin{array}{ll}
z_0  & \textrm{ on } (\check{\mu}')^{-1}(B_-); \\
z_0 G(z_1, \ldots, z_{n-1}) & \textrm{ on } (\check{\mu}')^{-1}(B_+)
.
\end{array}
\right.
$$
and
$$
v := \left\{
\begin{array}{ll}
z_0^{-1} G(z_1, \ldots, z_{n-1}) & \textrm{ on } (\check{\mu}')^{-1}(B_-); \\
z_0^{-1} & \textrm{ on } (\check{\mu}')^{-1}(B_+)
.
\end{array}
\right.
$$
Recall that $G$ is the Laurent polynomial defined by Equation \eqref{G}.

In the following we consider part (2) of Theorem \ref{mir_thm}.

\subsubsection{Holomorphic volume form.}
Recall that one has the semi-flat holomorphic volume form on $\check{X}_0$, which is written as $\der\log z_1 \wedge \ldots \wedge \der\log z_{n-1} \wedge \der\log z_0$ in Section \ref{semi-flat}.  Under the natural map $\rho$ (see Equation \eqref{rho}) this semi-flat holomorphic volume form extends to a holomorphic volume form $\check{\Omega}$ on $\check{X}$ which is exactly the one appearing in previous literatures (for example, see P.3 of \cite{konishi09}):

\begin{prop} \label{hol_vol}
There exists a holomorphic volume form $\check{\Omega}$ on $\check{X}$ which has the property that $\rho^*\check{\Omega} = \der\log z_0 \wedge \ldots \wedge \der\log z_{n-1}$.  Indeed in terms of the coordinates of $\cpx^2 \times (\cpx^\times)^{n-1}$,
$$\check{\Omega} = \mathrm{Res} \left( \frac{1}{uv - G(z_1, \ldots, z_{n-1})} \der\log z_1 \wedge \ldots \wedge \der\log z_{n-1} \wedge \der u \wedge \der v \right)$$
where $G$ is the polynomial defined by Equation \eqref{G}.
\end{prop}
\begin{proof}
Let $F = uv - G(z_1, \ldots, z_{n-1})$ be the defining function of $\check{X}$.  On $\check{X} \cap (\cpx^\times)^{n+1}$, we have the nowhere-zero holomorphic $n$-form
$$\der \log z_1 \wedge \ldots \wedge \der \log z_{n-1} \wedge \der \log u$$
whose pull-back by $\rho$ is $\der\log z_1 \wedge \ldots \wedge \der\log z_{n-1} \wedge \der\log z_0$.  It suffices to prove that this form extends to
$$\check{X} = \left\{(u,v,z_1, \ldots, z_{n-1}) \in \cpx^2 \times (\cpx^\times)^{n-1}: F = 0 \right\}.$$
It is clear that the form extends to the open subset of $\check{X}$ where $u \neq 0$.  By writing the form as
$$- \der \log z_1 \wedge \ldots \wedge \der \log z_{n-1} \wedge \der \log v$$
we see that it also extends to the open subset where $v \neq 0$.  Since
$$
u \der v + v \der u = \sum_{i=1}^{m-1} Q_i (1 + \delta_i) \prod_{j=1}^{n-1} z_j^{\pairing{\nu_j}{v_i}} \left(\sum_{k=1}^{n-1} \pairing{\nu_k}{v_i}\der\log z_k \right)
$$
where $Q_i := 1$ for $i = 0, \dots, n-1$ and $Q_i = q_i$ for $i = n, \ldots, m-1$, the above $n$-form can also be written as
\begin{align*}
&\frac{u\der v + v \der u}{\sum_{i=1}^{m-1} Q_i (1 + \delta_i) \prod_{j=1}^{n-1} z_j^{\pairing{\nu_j}{v_i}} \pairing{\nu_1}{v_i}} \wedge \der \log z_2 \wedge \ldots \wedge \der \log z_{n-1} \wedge \der \log u \\
=& \left(\frac{\partial F}{\partial z_1} \right)^{-1} \der v \wedge \der \log z_2 \wedge \ldots \wedge \der \log z_{n-1} \wedge \der u
\end{align*}
which is holomorphic when $\frac{\partial F}{\partial z_1} \neq 0$.  By similar change of variables, we see that the form is holomorphic whenever $\der F \neq 0$, which is always the case because $\check{X}$ is smooth.

For $u \not= 0$,
\begin{align*}
&\frac{1}{F} \der\log z_1 \wedge \ldots \wedge \der\log z_{n-1} \wedge \der u \wedge \der v \\
=& \der\log z_1 \wedge \ldots \wedge \der\log z_{n-1} \wedge \der \log u \wedge \frac{u \der v}{F} \\
=& \der\log z_1 \wedge \ldots \wedge \der\log z_{n-1} \wedge \der \log u \wedge \frac{\der F}{F} \\
\end{align*}
whose residue is $\der\log z_1 \wedge \ldots \wedge \der\log z_{n-1} \wedge \der \log u$.
\end{proof}

This proves part (2) of Theorem \ref{mir_thm}.

\subsubsection{Independence of choices of cones in $\Sigma$}

If in the beginning we have chosen another ordered basis which generates a cone of $\Sigma$ to construct the mirror, the complex manifold given in Theorem \ref{mir_thm} differs the original one by a biholomorphism which preserves the holomorphic volume form:

\begin{prop} \label{basis_change}
Let $\{u_0, \ldots, u_{n-1}\} \subset N$ and $\{v_0, \ldots, v_{n-1}\} \subset N$ be two ordered basis, each generates a cone of $\Sigma$.
Let $(\widetilde{X},\tilde{\Omega})$ and $(\check{X},\Omega)$ be the two mirror complex manifolds constructed from these two choices respectively.  Then there exists a biholomorphism $\phi:\widetilde{X} \to \check{X}$ with the property that
$ \phi^* \Omega = \pm \tilde{\Omega}.$
\end{prop}

\begin{proof}
Consider the mirror manifold given in Section \ref{mirror equation},
\begin{equation*}
\left\{
\begin{aligned}
&(u,v,z_1,\ldots,z_{n-1}) \in (\cpx^\times)^2 \times \cpx^{n-1}: \\
&uv = g(z) := \sum_{i=0}^{m-1} C_i (1 + \delta_i) \prod_{j=1}^{n-1} z_j^{\pairing{\nu_j}{v_i}}
\end{aligned}
\right\}
\end{equation*}
where $C_i := \exp\left(-\int_{\beta_i (r_0)} \omega\right) > 0$ and $\delta_i := \sum_{\alpha \not= 0} n_{\beta_i + \alpha} \exp\left(- \int_\alpha \omega \right)$ in which the summation is over all $\alpha \in H_2 (X, \integer) - \{0\}$ represented by rational curves.  If we choose another basis $\{u_0, \ldots, u_{n-1}\} \subset N$ whose dual basis is denoted by $\{\mu_0, \ldots, \mu_{n-1}\}$, then our mirror construction gives another equation
$$  \left\{ (\tilde{u},\tilde{v},\zeta_1,\ldots,\zeta_{n-1}) \in (\cpx^\times)^2 \times \cpx^{n-1}: \tilde{u} \tilde{v} = \sum_{i=0}^{m-1} C_i (1 + \delta_i) \prod_{j=1}^{n-1} \zeta_j^{\pairing{\mu_j}{v_i}} \right\} $$

Recall that $\underline{\nu} = \sum_{i=0}^{n-1} \nu_i = \sum_{i=0}^{n-1} \mu_i$.  Both $\{\underline{\nu}, \nu_1, \ldots, \nu_{n-1} \}$ and $\{\underline{\nu}, \mu_1, \ldots, \mu_{n-1} \}$ are basis of $M$.  Let $a \in \GL (n, \integer)$ be the change of basis, and so $\mu_j = a_{j,0} \underline{\nu} + \sum_{k=1}^{n-1} a_{jk} \nu_k$ for $j=1,\ldots,n-1$.  Then since $\pairing{\underline{\nu}}{v_i} = 1$ for all $i=0,\ldots,m-1$,
\begin{align*}
&\sum_{i=0}^{m-1} C_i (1 + \delta_i) \prod_{j=1}^{n-1} \zeta_j^{\pairing{\mu_j}{v_i}}\\
=& \sum_{i=0}^{m-1} C_i (1 + \delta_i) \prod_{j=1}^{n-1} \zeta_j^{a_{j,0} \pairing{\underline{\nu}}{v_i}} \prod_{j=1}^{n-1} \prod_{k=1}^{n-1} \zeta_j^{a_{jk} \pairing{\nu_k}{v_i}}\\
=& \left(\prod_{p=1}^{n-1} \zeta_p^{a_{p,0}}\right) \sum_{i=0}^{m-1} C_i (1 + \delta_i)  \prod_{k=1}^{n-1} \left(\prod_{j=1}^{n-1}\zeta_j^{a_{jk}}\right)^{\pairing{\nu_k}{v_i}}.
\end{align*}
Thus the coordinate change
$$ z_k = \prod_{j=1}^{n-1}\zeta_j^{a_{jk}}; u = \tilde{u}\left(\prod_{p=1}^{n-1} \zeta_p^{a_{p,0}}\right)^{-1}; v = \tilde{v} $$
gives the desired biholomorphism.  Moreover under this coordinate change
\begin{align*}
\Omega &= \der \log z_1 \wedge \ldots \wedge \der \log z_{n-1} \wedge \der \log u \\
&= \left( \sum_{j=1}^{n-1} a_{j,1} \log \zeta_j \right) \wedge \ldots \wedge \left( \sum_{j=1}^{n-1} a_{j,n-1} \log \zeta_j \right) \wedge \left( \log \tilde{u} - \sum_{j=1}^{n-1} a_{j,0} \log \zeta_j \right) \\
&= (\det A) \, \der \log \zeta_1 \wedge \ldots \wedge \der \log \zeta_{n-1} \wedge \der \log u \\
&= \pm \tilde{\Omega}.
\end{align*}

\end{proof}

\subsubsection{The superpotential.}
Recall that we have defined the generating function $\mathcal{F}_X$ of open Gromov-Witten invariants (Definition \ref{Lambda*}).  By taking Fourier transform, we obtain the superpotential, which is a holomorphic function on $(\check{\mu})^{-1}(B_0 - H)$, and it extends to be a holomorphic function on $\check{X}$:

\begin{prop}
Let $\tilde{z}_i$ be the holomorphic functions on $(\check{\mu}')^{-1}(B'_0 - H)$ given in Proposition \ref{FT_I}.
\begin{enumerate}
\item The Fourier transform of $\mathcal{F}_{X'}$ is the function
$$W' = \sum_{i=0}^{n-1} \tilde{z}_i$$
on $(\check{\mu}')^{-1}(B'_0 - H)$.

\item The Fourier transform of $\mathcal{F}_{X}$ is the function
$$W = \tilde{z}_0$$
on $(\check{\mu}')^{-1}(B'_0 - H)$, which equals to $\rho^* (C_0 u)$.  ($C_0$ is a constant defined by Equation \eqref{C_j}.)
\end{enumerate}
\end{prop}

\begin{proof}
Recall that (in Definition \ref{Lambda*})
$$\mathcal{F}_{X'}(\lambda) = \sum_{\beta \in \pi_2(X',\lambda)} n_\beta \exp\left(-\int_\beta \omega\right).$$
The sum is over all $\beta$ with $\mu(\beta) = 2$, which implies that $\beta$ intersect exactly one of the boundary divisors $D_i$ once (see Equation \ref{Maslov_X'}).  Thus
$$\mathcal{F}_{X'}(\lambda) = \sum_{i=0}^{n-1} \mathcal{I}_i (\lambda)$$
and so its Fourier transform $W'$ is $\sum_{i=0}^{n-1} \tilde{z}_i$.  This proves (1).

The Fourier transform of $\mathcal{F}_{X}$ is
$$W = \sum_{\lambda \in \pi_1(X,F_r)} \mathcal{F}_{X}(\lambda) \mathrm{Hol}_\conn (\lambda) = \sum_{\beta \in \pi_2(X,F_r)}  n_\beta \exp\left(-\int_\beta \omega\right) \mathrm{Hol}_\conn (\partial \beta).$$

For $r \in B_+$, by Proposition \ref{disk counting + in X}, $n_{\beta} = 0$ unless $\beta = \beta_k + \alpha$ for $k = 0, \ldots, m-1$ and $\alpha \in H_2(X)$ represented by rational curves.  Moreover, $n_{\beta_k} = 1$.  Thus
\begin{align*}
W
&= \sum_{j=0}^{m-1} \sum_{\alpha} n_{\beta_j(r) + \alpha} \exp\left(-\int_{\beta_j(r) + \alpha} \omega\right) \mathrm{Hol}_\conn (\partial \beta_j(r)) \\
&= \tilde{z}_0.
\end{align*}

For $r \in B_-$, by Proposition \ref{disk counting - in X}, $n_{\beta} = 0$ unless $\beta = \beta_0$, and $n_{\beta_0} = 1$.  Thus
\begin{align*}
W
&= \exp\left(-\int_{\beta_0 (r)} \omega\right)\mathrm{Hol}_\conn (\partial \beta_0)\\
&= \tilde{z}_0.
\end{align*}
By Equation \eqref{rho}, $\tilde{z}_0 = \rho^* (C_0 u)$.
\end{proof}

This ends the proof of Theorem \ref{mir_thm}.

\section{Enumerative meanings of (inverse) mirror maps} \label{period}

For a pair $(X$, $\check{X})$ of mirror Calabi-Yau manifolds, mirror symmetry asserts that there is a local isomorphism between the moduli space $\mathcal{M}_C(\check{X})$ of complex structures of $\check{X}$ and the complexified K\"ahler moduli space $\mathcal{M}_K(X)$ of $X$ near the large complex structure limit and large volume limit respectively, such that the Frobenius structures over the two moduli spaces get identified.  This is called the \textit{mirror map}. It gives canonical flat coordinates on $\mathcal{M}_C(\check{X})$ by transporting the natural flat structure on $\mathcal{M}_K(X)$. A remarkable feature of the instanton-corrected mirror family for a toric Calabi-Yau manifold we construct via SYZ is that it is inherently written in these canonical flat coordinates. In this section, we shall formulate this feature as a conjecture, and then give evidence for it for some 2- and 3-dimensional examples by applying the results in \cite{Chan10} and \cite{LLW10}.

\subsection{The conjecture}
Let $X = X_\Sigma$ be a toric Calabi-Yau $n$-fold.  We adopt the notation used in Section \ref{torCY}: $\{v_i\}_{i=0}^{m-1} \subset N$ are primitive generators of rays in the fan $\Sigma$, and $\{\nu_j\}_{i=0}^{n-1} \subset M$ is the dual basis of $\{v_j\}_{j=0}^{n-1} \subset N$.  Moreover, $H_2(X, \integer)$ is of rank $l = m-n$ generated by $\{S_a\}_{i=1}^{m-n}$ (see Equation \eqref{theta} and Proposition \ref{gen_H_2}).

\subsubsection{The complexified K\"ahler moduli}
Let $\mathcal{K}(X)$ be the K\"ahler cone of $X$, i.e. $\mathcal{K}(X)\subset H^2(X,\real)$ is the space of K\"ahler classes on $X$. Then let
$$\mathcal{M}_K(X)=\mathcal{K}(X)+2\pi\sqrt{-1}H^2(X,\real)/H^2(X,\integer).$$
This is the complexified K\"ahler moduli space of $X$. An element in $\mathcal{M}_K(X)$ is represented by a complexified K\"ahler class $\omega^\cpx=\omega+2\pi\sqrt{-1}B$, where $\omega\in\mathcal{K}(X)$ and $B\in H^2(X,\real)$. $B$ is usually called the B-field.  We have the map $\mathcal{M}_K(X) \to (\Delta^*)^l$ defined by
$$ q_i=\exp\left(-\int_{S_{n+i-1}} \omega^\cpx \right)$$
for $i = 1, \ldots, l$.
This map is a local biholomorphism from an open subset $U\subset\mathcal{M}_K(X)$ to $(\Delta^*)^l$, where $\Delta^*=\{z\in\cpx:0<|z|<1\}$ is the punctured unit disk.  The inclusion $(\Delta^*)^l \hookrightarrow \Delta^l$, where $\Delta=\{z\in\cpx:|z|<1\}$ is the unit disk, gives an obvious partial compactification, and the origin $0\in\Delta^l$ is called a \textit{large radius limit} point.  From now on, by abuse of notation we will take $\mathcal{M}_K(X)$ to be this open neighborhood of large radius limit.

\subsubsection{The mirror complex moduli}
On the other hand, let $\mathcal{M}_C(\check{X})=(\Delta^*)^l$.  We have a family of noncompact Calabi-Yau manifolds $\{\check{X}_{\check{q}}\}$ parameterized by $\check{q}\in\mathcal{M}_C(\check{X})$ defined as follows.  For $\check{q}=(\check{q}_1,\ldots,\check{q}_l)\in\mathcal{M}_C(\check{X})$,
\begin{equation}\label{HIV_mirror}
\check{X}_{\check{q}} := \left\{(u,v,z_1,\ldots,z_{n-1})\in\cpx^2\times(\cpx^\times)^{n-1}:
uv=\sum_{i=0}^{m-1}C_iz^{v_i}\right\},
\end{equation}
where $C_i \in \cpx$ are subject to the constraints
\begin{equation}\label{constraints}
C_{n+a-1} \prod_{i=0}^{n-1} C_i^{-\pairing{\nu_i}{v_{a+n-1}}} = \check{q}_a,\ a=1,\ldots,l.
\end{equation}

The origin $0\in\Delta^l$ in the partial compactification $\mathcal{M}_C(\check{X})\hookrightarrow\Delta^l$ is called a \textit{large complex structure limit} point.  Each $\{\check{X}_{\check{q}}\}$ is equipped with a holomorphic volume form $\check{\Omega}_{\check{q}}$ (see Proposition \ref{hol_vol}).

\subsubsection{The mirror map}
The mirror map $\psi:\mathcal{M}_C(\check{X}) \to \mathcal{M}_K(X)$ is defined by periods:
$$\psi(\check{q}):=\left(\int_{\gamma_1}\check{\Omega}_{\check{q}},
\ldots,\int_{\gamma_l}\check{\Omega}_{\check{q}}\right),$$
where $\{\gamma_1,\ldots,\gamma_l\}$ is a suitable basis of $H_n(\check{X},\integer)$.

Local mirror symmetry asserts that $\psi:\mathcal{M}_C(\check{X})\to\mathcal{M}_K(X)$
is an isomorphism onto $U\subset\mathcal{M}_K(X)$ (if $U$ is small enough), and this gives canonical flat coordinates on $\mathcal{M}_C(\check{X})$.

On the other hand, our construction of the instanton-corrected mirror gives a natural map $\phi:\mathcal{M}_K(X)\to\mathcal{M}_C(\check{X})$ as follows: Recall from Theorem \ref{mir_thm} that $\check{X}$ is defined by the equation
\begin{equation*}
uv = (1+\delta_0) + \sum_{j=1}^{n-1} (1 + \delta_j) z_j + \sum_{i=n}^{m-1} (1 + \delta_i) q_{i-n+1} z^{v_i}.
\end{equation*}
Comparing this with Equation \eqref{HIV_mirror} and \eqref{constraints}, one defines a map
$$\phi:\mathcal{M}_K(X)\to\mathcal{M}_C(\check{X}),\
(\check{q}_1,\ldots,\check{q}_l) = \phi(q_1,\ldots,q_l)$$
by
\begin{equation}\label{phi}
\check{q}_a = q_a (1 + \delta_{a+n-1}) \prod_{j=0}^{n-1} (1+\delta_j)^{-\pairing{\nu_j}{v_{a+n-1}}},\ a=1,\ldots,l.
\end{equation}

We claim that this gives the inverse of the mirror map:

\begin{conjecture} \label{can_coords2}
The map $\phi$ is an isomorphism locally near the large radius limit and gives the inverse of the mirror map $\psi$. In other words, there exists a basis $\gamma_1,\ldots,\gamma_l$ of $H_n(\check{X},\integer)$ such that
$$q_a=\exp\left(-\int_{\gamma_a}\check{\Omega}_{\check{q}}\right),$$
for $a=1,\ldots,l$, where $\check{q}=\phi(q)$ is defined as above. Hence, $\check{q}_1(q),\ldots,\check{q}_l(q)$ are flat coordinates on $\mathcal{M}_C(\check{X})$.
\end{conjecture}

In the literature, various integrality properties of mirror maps and their inverses (see e.g. \cite{Z10}) have been established. This suggests that the coefficients in the Taylor expansions of these maps have enumerative meanings. This is exactly what the above conjecture says for the inverse mirror map, namely, it can be expressed in terms of the open Gromov-Witten invariants $n_{\beta_i+\alpha}$ for $X$.\footnote{As we mentioned in the introduction, Gross and Siebert were the first to conjecture such a relation between canonical flat coordinates and disk counting invariants. More precisely, they found that canonical coordinates can be obtained by imposing a normalization condition on slabs, which are in-turn believed to be related to the counting of tropical disks. See Conjecture 0.2 in \cite{gross07}.} See Remark \ref{Int} below for a geometric reason why we have integrality for the inverse mirror map in case $X$ is a toric Calabi-Yau $3$-fold of the form $K_S$, where $S$ is a toric Fano surface.

In practice, one computes the mirror map by solving a system of linear differential equations associated to the toric Calabi-Yau manifold $X$. For $i=0,1,\ldots,m-1$, denote by $\theta_i$ the differential operator $C_i\frac{\partial}{\partial C_i}$. For $j=1,\ldots,n$, let
$$\mathcal{T}_j=\sum_{i=0}^{m-1}v_i^j\theta_i,$$
where $v_i^j = \pairing{\nu_j}{v_i}$.  For $a=1,\ldots,l$, let
$$\Box_a=\prod_{i:Q^a_i>0}\left(\frac{\partial}{\partial C_i}\right)^{Q^a_i}-\prod_{i:Q^a_i<0}\left(\frac{\partial}{\partial C_i}\right)^{-Q^a_i}$$
where $Q^a_j = - \pairing{\nu_j}{v_{a+n-1}}$ for $j = 0, \ldots, n-1$, and $Q^a_i = \delta_{i,a+n-1}$ for $i = n, \ldots, m-1$.
Then, the $A$-hypergeometric system (also called GKZ system) of linear differential equations associated to $X$ is given by
\begin{eqnarray*}
\mathcal{T}_j\Phi(C)=0\ (j=1,\ldots,n),\ \ \Box_a\Phi(C)=0\ (a=1,\ldots,l).
\end{eqnarray*}
If we denote by $\check{X}_C$ the noncompact Calabi-Yau manifold (\ref{HIV_mirror}) parameterized by $C=(C_0,C_1,\ldots,C_{m-1})\in\cpx^m$ and $\check{\Omega}_C$ a holomorphic volume form on it, then, for any n-cycle $\gamma\in H_3(\check{X},\integer)$, the period
$$\Pi_\gamma(C)=\int_\gamma\check{\Omega}_C,$$
as a function of $C=(C_0,C_1,\ldots,C_{m-1})$, satisfies the above $A$-hypergeometric system (see e.g. \cite{hosono06} and \cite{konishi09}).

By imposing the constraints (\ref{constraints}), the $A$-hypergeometric system is reduced to a set of Picard-Fuchs equations (see the examples in Subsection 5.3), which are satisfied by the periods
$$\Pi_\gamma(\check{q})=\int_\gamma\check{\Omega}_{\check{q}},\ \gamma\in H_n(\check{X},\integer),$$
as functions of $\check{q}\in\mathcal{M}_C(\check{X})$. Now, let $\Phi_1(\check{q}),\ldots,\Phi_l(\check{q})$ be a basis of the solutions of this set of Picard-Fuchs equations with a single logarithm. Then there is a basis $\gamma_1,\ldots,\gamma_l$ of $H_n(\check{X},\integer)$ such that
$$\Phi_a(\check{q})=\int_{\gamma_a}\check{\Omega}_{\check{q}}$$
for $a=1,\ldots,l$, and the mirror map $\psi:\mathcal{M}_C(\check{X})\to\mathcal{M}_K(X)$ is given by
$$\psi(\check{q})=(\exp(-\Phi_1(\check{q})),\ldots,\exp(-\Phi_l(\check{q}))).$$

In the literature, the mirror map is computed by solving the Picard-Fuchs equations (\cite{aganagic-klemm-vafa}, \cite{graber-zaslow01}). One can then also compute (the Taylor series expansion of) its inverse. To give evidences for Conjecture \ref{can_coords2}, we need to compute the open Gromov-Witten invariants $n_{\beta_i+\alpha}$ and then compare the map $\phi$ define by (\ref{phi}) with the inverse mirror map.

\subsection{Computation of open Gromov-Witten invariants}\label{blowup-flop}

In this subsection, we compute the open Gromov-Witten invariants $n_{\beta_i+\alpha}$ for a class of examples using the results in \cite{Chan10} and \cite{LLW10}. We first establish some general basic properties for these invariants.

\begin{lemma}\label{hol_sphere}
Let $X_\Sigma$ be a toric manifold defined by a fan $\Sigma$. Suppose that there exists $\nu\in M$ such that $\nu$ defines a holomorphic function on $X_\Sigma$ whose zero set contains all the toric prime divisors $\mathscr{D}_i\subset X_\Sigma$. Then the image of any non-constant holomorphic map $u:\proj^1\to X_\Sigma$ lies entirely inside the union $\bigcup_i\mathscr{D}_i$ of the toric prime divisors in $X_\Sigma$. In particular this holds for a toric Calabi-Yau manifold $X$.
\end{lemma}
\begin{proof}
Denote the holomorphic function corresponding to $\nu\in M$ by $f$. Then $f\circ u$ gives a holomorphic function on $\proj^1$, which must be a constant by the maximal principle. $f\circ u$ cannot be constantly non-zero since otherwise, the image of $u$ lies entirely inside the open orbit $(\cpx^\times)^n\subset X_\Sigma$, which forces $u$ to be a constant map. Thus $f\circ u\equiv 0$, which implies that the image of $u$ lies in the union of the toric prime divisors in $X_\Sigma$.

For a toric Calabi-Yau manifold $X$, we have $\pairing{\underline{\nu}}{v_i}=1>0$ for $i=0,1,\ldots,m-1$. This implies that the meromorphic function corresponding to $\underline{\nu}$ has no poles and its zero set is exactly $\bigcup_i\mathscr{D}_i$.
\end{proof}

It is known that $n_{\beta_i}=1$ by the results of Cho-Oh \cite{cho06}. In addition, we have
\begin{prop}\label{no_cpt_div}
Suppose that $\alpha\in H_2^\textrm{eff}(X,\integer)-\{0\}$, where $H_2^\textrm{eff}(X,\integer)$ is the semi-group of all classes of holomorphic curves in $X$.  Then $n_{\beta_i+\alpha}=0$ unless the toric prime divisor $\mathscr{D}_i\subset X$ corresponding to $v_i$ is compact.
\end{prop}
\begin{proof}
Suppose that $\mathcal{M}_1(F,\beta_i+\alpha)$ is non-empty. Then $\alpha\neq0$ is realized by a non-constant genus zero stable map to $X$, whose image $Q$ must lie inside $\bigcup_{i=0}^{m-1}\mathscr{D}_i$ by Lemma \ref{hol_sphere}. $Q$ has non-empty intersection with the holomorphic disk representing $\beta_i\in\pi_2(X,\mathbf{T})$ for generic toric fiber $\mathbf{T}$. This implies that there must be some components of $Q$ which lie inside $\mathscr{D}_i$ and have non-empty intersection with the open orbit $(\cpx^\times)^{n-1}\subset\mathscr{D}_i$. But if $\mathscr{D}_i$ is non-compact, then the fan of $\mathscr{D}_i$ is simplicial convex incomplete, and so $\mathscr{D}_i$ itself is a toric manifold satisfying the condition of Lemma \ref{hol_sphere}. This forces $Q$ to have empty intersection with $(\cpx^\times)^{n-1}\subset\mathscr{D}_i$, which is a contradiction.
\end{proof}

>From now on, we shall restrict ourselves to the case where $X$ is the total space of the canonical line bundle of a toric Fano manifold, i.e. $X=K_Z$, where $Z$ is a toric Fano manifold. Note that there is only one compact toric prime divisor $\mathscr{D}_0\subset X$ (the zero section $Z\hookrightarrow K_Z$) which corresponds to the primitive generator $v_0$. By the above proposition, it suffices to compute the numbers $n_{\beta_0+\alpha}$ for $\alpha\in H_2^\textrm{eff}(X,\integer)-\{0\}$.

Let $\bar{\Sigma}$ be the refinement of $\Sigma$ by adding the ray generated by $v_\infty:=-v_0$ (and then completing it into a convex fan), and let $\bar{X}=\proj_{\bar{\Sigma}}$. This gives a toric compactification of $X$, and $v_\infty$ corresponds to the toric prime divisor $D_\infty=\bar{X}-X$.

\begin{theorem}[Theorem 1.1 in Chan \cite{Chan10}]\label{cptification}
Let $X=K_Z$, where $Z$ is a toric Fano manifold. Fix a toric fiber $\mathbf{T}\subset X$. For $\alpha\in H_2^\textrm{eff}(X,\integer)-\{0\}\subset H_2^\textrm{eff}(\bar{X},\integer)-\{0\}$, we have the following equality between open and closed Gromov-Witten invariants
$$n_{\beta_0+\alpha}=\mathrm{GW}_{0,1}^{\bar{X},h+\alpha}(\mathrm{P.D.}[\mathrm{pt}]).$$
Here, $h\in H_2(\bar{X},\integer)$ is the fiber class of the $\proj^1$-bundle $\bar{X}\to Z$, $\mathrm{P.D.}[\mathrm{pt}]\in H^{2n}(\bar{X},\cpx)$ is the Poincar\'e dual of a point in $\bar{X}$, and the 1-point genus zero Gromov-Witten invariant $\mathrm{GW}_{0,1}^{\bar{X},h+\alpha}(\mathrm{P.D.}[\mathrm{pt}])$ is defined by
$$\mathrm{GW}_{0,1}^{\bar{X},h+\alpha}(\mathrm{P.D.}[\mathrm{pt}])
=\int_{[\overline{\mathcal{M}}_{0,1}(\bar{X},h+\alpha)]}\textrm{ev}^*(\mathrm{P.D.}[\mathrm{pt}]),$$
where $[\overline{\mathcal{M}}_{0,1}(\bar{X},h+\alpha)]$ is the virtual fundamental cycle of the moduli space $\overline{\mathcal{M}}_{0,1}(\bar{X},h+\alpha)$ of genus zero stable maps to $\bar{X}$ with 1 marked point in the class $h+\alpha$ and $\textrm{ev}:\overline{\mathcal{M}}_{0,1}(\bar{X},h+\alpha)\to\bar{X}$ is the evaluation map.
\end{theorem}
\begin{proof}[Sketch of proof]
Denote by $M_\textrm{op}$ and $M_\textrm{cl}$ the moduli spaces $\mathcal{M}_1(\mathbf{T},\beta_0+\alpha)$ and $\overline{\mathcal{M}}_{0,1}(\bar{X},h+\alpha)$ respectively. Fix a point $p\in\mathbf{T}\subset\bar{X}$. Then let
$$M_\textrm{op}^{\textrm{ev}=p}:=\textrm{ev}^{-1}(p),\ M_\textrm{cl}^{\textrm{ev}=p}:=\textrm{ev}^{-1}(p)$$
be the fibers of the evaluation maps $\textrm{ev}:M_\textrm{op}\to\mathbf{T}$ and $\textrm{ev}:M_\textrm{cl}\to\bar{X}$ respectively.

$M_\textrm{op}^{\textrm{ev}=p}, M_\textrm{cl}^{\textrm{ev}=p}$ have Kuranishi structures induced naturally from those on $M_\textrm{op}, M_\textrm{cl}$ respectively. We have (trivial) evaluation maps $\textrm{ev}:M_\textrm{op}^{\textrm{ev}=p}\to\{p\}$, $\textrm{ev}:M_\textrm{cl}^{\textrm{ev}=p}\to\{p\}$ and virtual fundamental cycles
$$[M_\textrm{op}^{\textrm{ev}=p}], [M_\textrm{cl}^{\textrm{ev}=p}]\in H_0(\{p\},\rat)=\rat.$$
Moreover, by Lemma A1.43 in \cite{FOOO_I}, we have
$$n_{\beta_0+\alpha}=[M_\textrm{op}^{\textrm{ev}=p}]\textrm{ and }
\mathrm{GW}_{0,1}^{\bar{X},h+\alpha}(\mathrm{P.D.}[\mathrm{pt}])=[M_\textrm{cl}^{\textrm{ev}=p}].$$
Hence, to prove the desired equality, it suffices to show that $M_\textrm{op}^{\textrm{ev}=p}, M_\textrm{cl}^{\textrm{ev}=p}$ are isomorphic as Kuranishi spaces.

Let $\sigma^\textrm{op}=((\Sigma^\textrm{op},z),u)$ be a point in $M_\textrm{op}^{\textrm{ev}=p}$. This consists of a genus 0 nodal Riemann surface $\Sigma^\textrm{op}$ with nonempty connected boundary and a boundary marked point $z\in\partial\Sigma^\textrm{op}$ and a stable holomorphic map $u:(\Sigma^\textrm{op},\partial\Sigma^\textrm{op})\to(\bar{X},\mathbf{T})$ with $u(z)=p$ representing the class $\beta_0+\alpha$. By applying the results of Cho-Oh \cite{cho06}, we see that $\Sigma^\textrm{op}$ must be singular and can be decomposed as $\Sigma^\textrm{op}=\Sigma^\textrm{op}_0\cup\Sigma_1$, where $\Sigma^\textrm{op}_0=\Delta$ is the unit disk and $\Sigma_1$ is a genus zero nodal curve, such that the restrictions of $u$ to $\Sigma^\textrm{op}_0$ and $\Sigma_1$ represent the classes $\beta_0$ and $\alpha$ respectively (Proposition 4.2 in \cite{Chan10}).

Now, there exists a unique holomorphic disk $u_\infty:(\Delta,\partial\Delta)\to(\bar{X},\mathbf{T})$ with class $\beta_\infty$ (which corresponds to $v_\infty=-v_0$ and intersects $\mathscr{D}_\infty$ at one point) such that its boundary $\partial u_\infty$ coincides with $\partial u$ but with opposite orientations (Proposition 4.3 in \cite{Chan10}). We can then glue the maps $u:(\Sigma^\textrm{op},\partial\Sigma^\textrm{op})\to(\bar{X},\mathbf{T})$, $u_\infty:(\Delta,\partial\Delta)\to(\bar{X},\mathbf{T})$ along the boundary to give a map $u':\Sigma\to\bar{X}$, where $\Sigma$ is the union of $\Sigma^\textrm{op}$ and $\Delta$ by identifying their boundaries. The map $u'$ has class $\beta_0+\beta_\infty+\alpha=h+\alpha\in H_2(\bar{X},\integer)$.

This defines a map $j:M_\textrm{op}^{\textrm{ev}=p}\to M_\textrm{cl}^{\textrm{ev}=p}$. To see that $j$ is a bijective map, let $\sigma^\textrm{cl}=((\Sigma,z),u)$ be representing a point in $M_\textrm{cl}^{\textrm{ev}=p}$, which consists of a genus 0 nodal curve $\Sigma$ with a marked point $z\in\Sigma$ and a stable holomorphic map $u:\Sigma\to\bar{X}$ such that $u(z)=p$. One can show that $\Sigma$ must be singular and decomposes as $\Sigma=\Sigma_0\cup\Sigma_1$, where $\Sigma_0=\proj^1$ and $\Sigma_1$ is genus 0 nodal curve such that the restrictions of $u$ to $\Sigma_0$ and $\Sigma_1$ represent the classes $h$ and $\alpha$ respectively (Proposition 4.4 in \cite{Chan10}). Now, the Lagrangian torus $\mathbf{T}$ cuts the image of $u|_{\Sigma_0}$ into two halves, one representing $\beta_0$ and the other representing $\beta_\infty$. We can then reverse the above construction and defines the inverse of $j$.

Furthermore, from these descriptions of the structures of the maps in $M_\textrm{op}^{\textrm{ev}=p}, M_\textrm{cl}^{\textrm{ev}=p}$, it is evident that they have the same Kuranishi structures. We refer the reader to Proposition 4.5 in \cite{Chan10} for a rigorous proof of this assertion.
\end{proof}

By the above theorem, we can use techniques for computing closed Gromov-Witten invariants (e.g. localization) to compute the open Gromov-Witten invariants $n_{\beta_0+\alpha}$. When $\dim X=3$, blow-up and flop arguments can be employed to relate the 1-point invariant $\mathrm{GW}_{0,1}^{\bar{X},h+\alpha}(\mathrm{P.D.}[\mathrm{pt}])$ to certain local BPS invariants of another toric Calabi-Yau manifold. This idea is developed in more details in \cite{LLW10}. As a special case of the results in \cite{LLW10}, we have the following
\begin{theorem}[Theorem 1.2 in Lau-Leung-Wu \cite{LLW10}]\label{thm_flop}
Let $X=K_S$, where $S$ is a toric Fano surface, and let $Y=K_{\tilde{S}}$ where $\tilde{S}$ is the toric blow up of $S$ at a toric fixed point $q$. Let $\bar{X}$ and $\bar{Y}$ be the toric compactifications of $X$ and $Y$ respectively as before. Then we have
$$\mathrm{GW}_{0,1}^{\bar{X},h+\alpha}(\mathrm{P.D.}[\mathrm{pt}])=\mathrm{GW}_{0,0}^{\bar{Y},\tilde{\alpha}-e},$$
where $h\in H_2(\bar{X},\integer)$ is the fiber class as before, $e\in H_2(\tilde{S},\integer)\subset H_2(\bar{Y},\integer)$ is the class of the exceptional divisor, and $\tilde{\alpha}\in H_2(\tilde{S},\integer)$ is the total transform of $\alpha\in H_2(S,\integer)$ under the blowing up $\tilde{S}\to S$. If $\tilde{S}$ is also Fano, then we further have
$$\mathrm{GW}_{0,1}^{\bar{X},h+\alpha}(\mathrm{P.D.}[\mathrm{pt}])=\mathrm{GW}_{0,0}^{Y,\tilde{\alpha}-e}.$$
Here, the invariant on the right-hand-side is the local BPS invariant of the toric Calabi-Yau 3-fold $Y$ defined by
$$\mathrm{GW}_{0,0}^{Y,\tilde{\alpha}-e}=\int_{[\overline{\mathcal{M}}_{0,0}(\tilde{S},\tilde{\alpha}-e)]}
c_{\textrm{top}}(R^1\textrm{forget}_*\textrm{ev}^*K_{\tilde{S}}),$$
where $[\overline{\mathcal{M}}_{0,0}(\tilde{S},\tilde{\alpha}-e)]$ is the virtual fundamental cycle of the moduli space $\overline{\mathcal{M}}_{0,0}(\tilde{S},\tilde{\alpha}-e)$ of genus zero stable maps to $\tilde{S}$ in the class $\tilde{\alpha}-e$, $\textrm{forget}:\overline{\mathcal{M}}_{0,1}(\tilde{S},\tilde{\alpha}-e)\to
\overline{\mathcal{M}}_{0,0}(\tilde{S},\tilde{\alpha}-e)$ is the map forgetting the marked point, $\textrm{ev}:\overline{\mathcal{M}}_{0,1}(\tilde{S},\tilde{\alpha}-e)\to\tilde{S}$ is the evaluation map and $c_{\textrm{top}}$ denotes top Chern class.
\end{theorem}
\begin{proof}[Sketch of proof]
A toric fixed point $q \in S$ corresponds to a toric fixed point $p \in D_\infty \subset \bar{X}$. First we blow up $p$ to get $X_1$, whose defining fan $\Sigma_1$ is obtained by adding the ray generated by $w=v_\infty+u_1+u_2$ to $\bar{\Sigma}$, where $v_\infty$, $u_1$ and $u_2$ are the normal vectors to the three facets adjacent to $p$. Now $\langle u_1,u_2,w\rangle_\real$ and $\langle u_1,u_2,v_0\rangle_\real$ form two adjacent simplicial cones in $\Sigma_1$, and we may employ a flop to obtain a new toric variety $\bar{Y}$, whose fan contains the adjacent cones $\langle w,v_0,u_1\rangle_\real$ and $\langle w,v_0,u_2\rangle_\real$ (see Figure \ref{fig_flop}). In fact $\bar{Y}$ is the toric compactification of $Y=K_{\tilde{S}}$, where $\tilde{S}$ is the toric blow up of $S$ at the torus-fixed point $q$. By using the equalities of Gromov-Witten invariants for blowing up \cite{Hu00, gathmann} and flop \cite{li-ruan01}, one has
$$\mathrm{GW}_{0,1}^{\bar{X},h+\alpha}([\mathrm{pt}])=\mathrm{GW}^{X_1,h+\alpha}_{0,0}
=\mathrm{GW}_{0,0}^{\bar{Y},\tilde{\alpha}-e}.$$
If we further assume that $\tilde{S}$ is Fano, then any rational curve representing $\tilde{\alpha}-e\in H_2(\tilde{S},\integer)\subset H_2(\bar{Y},\integer)$ never intersects $\mathscr{D}_\infty$. Thus
$$\mathrm{GW}_{0,0}^{\bar{Y},\tilde{\alpha}-e}=\mathrm{GW}_{0,0}^{Y,\tilde{\alpha}-e}.$$

\begin{figure}[htp]
\begin{center}
\includegraphics[scale=0.7]{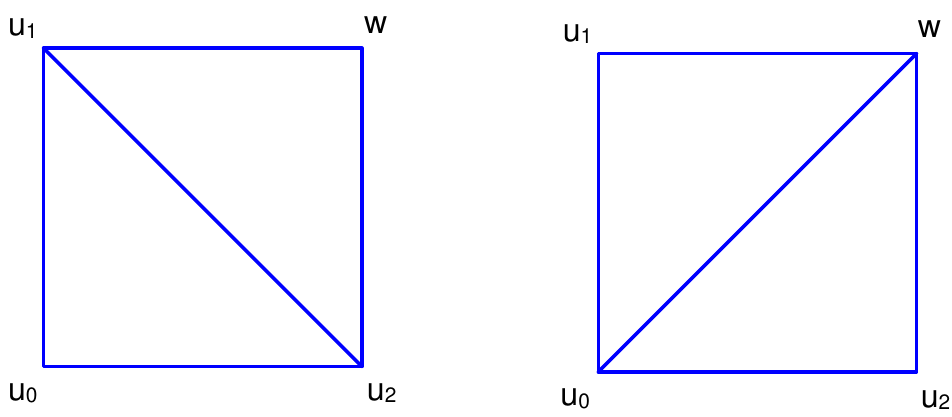}
\end{center}
\caption{A flop.} \label{fig_flop}
\end{figure}
\end{proof}

Combining the above two theorems, we get
\begin{corollary} \label{mirror_KS}
Let $S$ be a smooth toric Fano surface and $X=K_S$. Fix $\alpha\in H_2^\textrm{eff}(X,\integer)-\{0\}=H_2^\textrm{eff}(S,\integer)-\{0\}$. Suppose the toric blow-up $\tilde{S}$ of $S$ at a toric fixed point is still a toric Fano surface. Then we have
\begin{equation} \label{open-closed}
n_{\beta_0+\alpha}=\mathrm{GW}_{0,0}^{Y,\tilde{\alpha}-e}
\end{equation}
where $Y=K_{\tilde{S}}$, $e\in H_2(\tilde{S},\integer)$ is the class of the exceptional divisor, and $\tilde{\alpha} \in H_2(\tilde{S},\integer)$ is the pull-back (via Poincar\'e duality) of $\alpha\in H_2(S,\integer)$ under the blowing up $\tilde{S}\to S$.
\end{corollary}

We conclude that the instanton-corrected mirror $\check{X}$ of $X=K_S$ is given by
$$\check{X}=\left\{(u,v,z_1,z_2)\in\cpx^2\times(\cpx^\times)^2:uv=1+\delta_0(q)+ \sum_{j=1}^{m-1}e^{c_i}z^{w_i}\right\}$$
where
$$\delta_0(q)=\sum_{\alpha\in H_2^\textrm{eff}(X,\integer)-\{0\}}\mathrm{GW}_{0,0}^{Y,\tilde{\alpha}-e}q^\alpha.$$

\begin{remark} \label{Int}
Since the class $\tilde{\alpha}-e\in H_2(\tilde{S},\integer)=H_2(Y,\integer)$ is primitive, there is no multiple-cover contribution and hence $\mathrm{GW}_{0,0}^{Y,\tilde{\alpha}-e}$ is indeed an integer. Hence, the coefficients of the Taylor series expansions of $\delta_0$ and hence the map $\psi$ we define are all integers. This explains why we have integrality properties for the inverse mirror maps.
\end{remark}

The invariants on the right hand side of the formula (\ref{open-closed}) have been computed by Chiang-Klemm-Yau-Zaslow \cite{CKYZ}. Making use of their results, we can now give supportive evidences for Conjecture \ref{can_coords2} in various examples.

\subsection{Examples}\label{verification}

In this subsection, we shall use the results in the previous section to give evidences for Conjecture \ref{can_coords2} in various examples.

\subsubsection{$K_{\proj^1}$}
Consider our familiar example $X=K_{\proj^1}$. The generators of the 1-dimensional cones of the defining fan $\Sigma$ are $v_0=(0,1), v_1=(1,1)$ and $v_2=(-1,1)$, as shown in Figure \ref{KP1_fan}. We equip $X$ with a toric K\"{a}hler structure $\omega$ so that the associated moment polytope $P$ is given by
$$P=\{(x_1,x_2)\in\real^2:x_2\geq0,x_1+x_2\geq0,-x_1+x_2\geq-t_1\},$$
where $t_1=\int_l\omega>0$ and $l\in H_2(X,\integer)$ is the class of the zero section in $K_{\proj^1}$. To complexify the K\"ahler class, we set $\omega^\cpx=\omega+2\pi\sqrt{-1}B$, for some real two-form $B$ (the $B$-field). We let $t=\int_l\omega^\cpx\in\cpx$.

Since $\mathscr{D}_0$, the zero section of $K_{\proj^1}\to\proj^1$, is the only compact toric prime divisor, by Proposition \ref{no_cpt_div} and Theorem \ref{mir_thm}, the instanton-corrected mirror is given by
$$\check{X}=\{(u,v,z)\in\cpx^2\times\cpx^\times: uv=1+\sum_{k=1}^\infty n_{\beta_0+kl}q^k+z+\frac{q}{z}\},$$
where $q=\exp(-t)$.

Now, the toric compactification of $X$ is $\bar{X}=\proj(K_{\proj^1}\oplus\mathcal{O}_{\proj^1})=\mathbb{F}_2$ (a Hirzebruch surface). Using Theorem \ref{cptification}, the open Gromov-Witten invariants $n_{\beta_0+kl}$ can easily be computed as $\mathbb{F}_2$ is symplectomorphic to $\mathbb{F}_0=\proj^1\times\proj^1$ (see \cite{Chan10} and also \cite{auroux09} and \cite{FOOO10}). The result is
$$n_{\beta_0+kl}=\begin{cases}
				  1 & \textrm{ if }k=0,1;\\
				  0 & \textrm{ otherwise.}
				 \end{cases}$$
Hence, the corrected mirror $\check{X}$ can be written as
$$\check{X}=\{(u,v,z)\in\cpx^2\times\cpx^\times: uv=(1+\frac{q}{z})(1+z)\}.$$
We remark that this agrees with the formula written down by Hosono (See Proposition 3.1 and the following remark in \cite{hosono06}). We have $Q=(-2,1,1)$, and both $\mathcal{M}_C(\check{X})$ and $\mathcal{M}_K(X)$ can be identified with the punctured unit disk $\Delta^*$. The map $\phi:\Delta^*\to\Delta^*$ we define is thus given by $q\mapsto\check{q}(q)=q(1+q)^{-2}$.

In this example, the period of $\check{X}$ can be computed directly. Recall that the holomorphic volume form on $\check{X}$ is given by $\check{\Omega}=d\log u\wedge d\log z$. There is an embedded $S^2\subset\check{X}$ given by $\{(u,v,-1+(1-q)t)\in\check{X}:|u|=|v|,0\leq t\leq1\}$. Let $\gamma\in H_2(\check{X},\integer)$ be its class. Then
$$\int_\gamma\check{\Omega}=-\log q=t.$$
This verifies Conjecture \ref{can_coords2} for $K_{\proj^1}$.\hfill $\square$

\subsubsection{$\mathcal{O}_{\proj^1}(-1)\oplus\mathcal{O}_{\proj^1}(-1)$}
For $X =\mathcal{O}_{\proj^1}(-1)\oplus\mathcal{O}_{\proj^1}(-1)$, the generators of the 1-dimensional cones of the defining fan $\Sigma$ are $v_0=(0,0,1), v_1=(1,0,1), v_2=(0,1,1)$ and $v_3=(1,-1,1)$. We equip $X$ with a toric K\"{a}hler structure $\omega$ so that the associated moment polytope $P$ is given by
$$P=\{(x_1,x_2,x_3)\in\real^3:x_3\geq0,x_1+x_3\geq0,x_2+x_3\geq0,x_1-x_2+x_3\geq-t_1\},$$
where $t_1=\int_l\omega>0$ and $l\in H_2(X,\integer)$ is the class of the embedded $\proj^1\subset X$. To complexify the K\"ahler class, we set $\omega^\cpx=\omega+2\pi\sqrt{-1}B$, for some real two-form $B$ (the $B$-field). We let $t=\int_l\omega^\cpx\in\cpx$.

Since there is no compact toric prime divisors in $X$ (see Figure \ref{O(-1)+O(-1)} below), by Proposition \ref{no_cpt_div} and Theorem \ref{mir_thm}, the instanton-corrected mirror is given by
$$\check{X}=\{(u,v,z_1,z_2)\in\cpx^2\times(\cpx^\times)^2:uv=1+z_1+z_2+qz_1z_2^{-1}\},$$
where $q=\exp(-t)$.

\begin{figure}[htp]
\begin{center}
\includegraphics{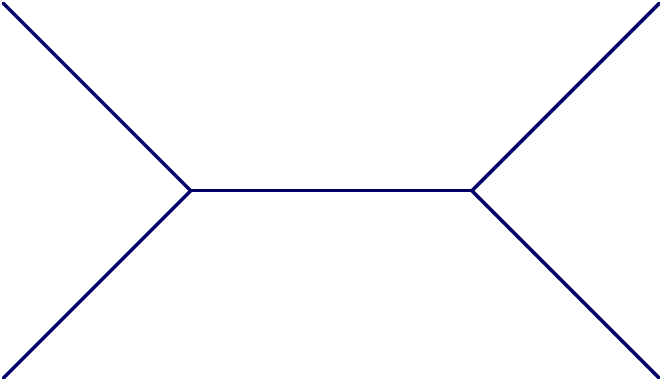}
\end{center}
\caption{$\mathcal{O}_{\proj^1}(-1)\oplus\mathcal{O}_{\proj^1}(-1)$.} \label{O(-1)+O(-1)}
\end{figure}

Both $\mathcal{M}_C(\check{X})$ and $\mathcal{M}_K(X)$ can be identified with the punctured unit disk $\Delta^*$ and the map $\phi:\Delta^*\to\Delta^*$ we define in (\ref{phi}) is the identity map. This agrees with the fact $\Phi(\check{q})=-\log\check{q}$ is the unique (up to addition and multiplication by constants) solution with a single logarithm of the Picard-Fuchs equation
$$((1-\check{q})\theta_{\check{q}}^2)\Phi(\check{q})=0,$$
where $\theta_{\check{q}}$ denotes $\check{q}\frac{\partial}{\partial\check{q}}$, which implies that the mirror map $\psi$ is the identity. Hence, Conjecture \ref{can_coords2} also holds for this example. \hfill $\square$

\subsubsection{$K_{\proj^2}$} \label{K_P2}
The primitive generators of the 1-dimensional cones of the fan $\Sigma$ defining $X=K_{\proj^2}$ can be chosen to be
$v_0=(0,0,1), v_1=(1,0,1), v_2=(0,1,1)$ and $v_3=(-1,-1,1)$. We equip $X$ with a toric K\"{a}hler structure $\omega$ associated to the moment polytope
\begin{align*}
P=\{&(x_1,x_2,x_3)\in\real^3:\\
&x_3\geq0,x_1+x_3\geq0,x_2+x_3\geq0,-x_1-x_2+x_3\geq-t_1\},
\end{align*}
where $t_1=\int_l\omega>0$ and $l\in H_2(X,\integer)=H_2(\proj^2,\integer)$ is the class of a line in $\proj^2\subset X$. To complexify the K\"ahler class, we set $\omega^\cpx=\omega+2\pi\sqrt{-1}B$, where $B$ is a real two-form (the $B$-field). We let $t=\int_l\omega^\cpx\in\cpx$.

There is only one compact toric prime divisor $\mathscr{D}_0$ which is the zero section $\proj^2\hookrightarrow K_{\proj^2}$ and it corresponds to $v_0$. By Proposition \ref{no_cpt_div} and Theorem \ref{mir_thm}, the instanton-corrected mirror $\check{X}$ is given by
$$\check{X}=\left\{
\begin{aligned}
&(u,v,z_1,z_2)\in\cpx^2\times(\cpx^\times)^2:\\
&uv=\left(1+\sum_{k=1}^\infty n_{\beta_0+kl}q^k\right)+z_1+z_2+\frac{q}{z_1z_2}
\end{aligned}
\right\},$$
where $q=\exp(-t)$.

By Corollary \ref{mirror_KS}, we have
$$n_{\beta_0+kl}=\mathrm{GW}_{0,0}^{Y,kf+(k-1)e},$$
where $Y=K_{\mathbb{F}_1}$, $\mathbb{F}_1$ is the blowup of $\proj^2$ at a point and $e,f\in H_2(\mathbb{F}_1,\integer)$ are the classes of the exceptional divisor and the fiber of the blowup $\mathbb{F}_1\to\proj^2$. See Figure \ref{KP2_KF1} below.

\begin{figure}[htp]
\begin{center}
\includegraphics[scale=0.75]{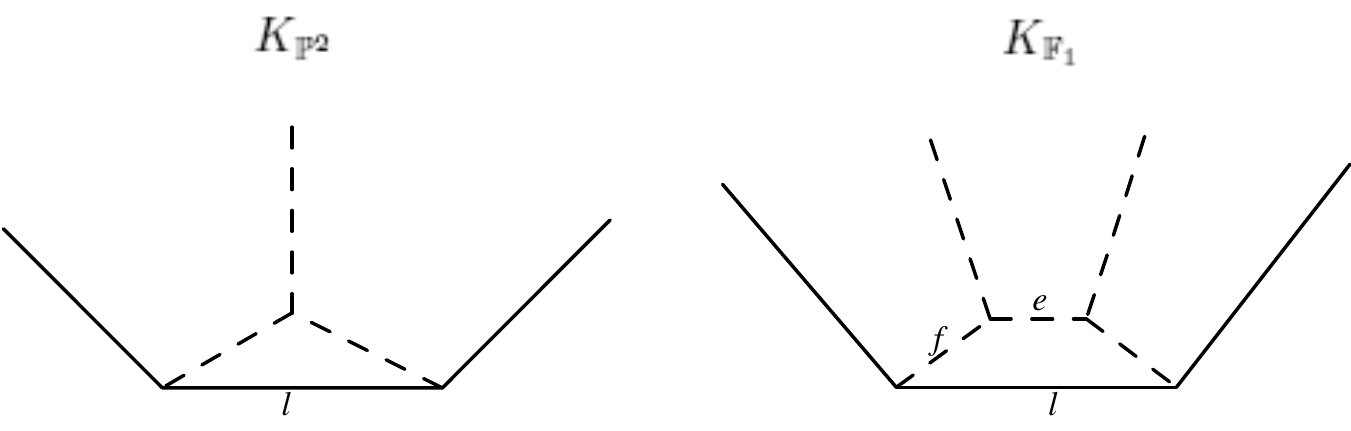}
\end{center}
\caption{Polytope picture for $K_{\proj^2}$ and $K_{\mathbb{F}_1}$.} \label{KP2_KF1}
\end{figure}

The local BPS invariants $\mathrm{GW}_{0,0}^{Y,kf+(k-1)e}$ have been computed by Chiang-Klemm-Yau-Zaslow and the results can be found on the "sup-diagonal" of Table 10 in \cite{CKYZ}:
\begin{align*}
n_{\beta_0+l}  & = -2,\\
n_{\beta_0+2l} & = 5,\\
n_{\beta_0+3l} & = -32,\\
n_{\beta_0+4l} & = 286,\\
n_{\beta_0+5l} & = -3038,\\
n_{\beta_0+6l} & = 35870,\\
               &\vdots&
\end{align*}
Using these results, we can write the instanton-corrected mirror explicitly as
$$\check{X}=\left\{(u,v,z_1,z_2)\in\cpx^2\times(\cpx^\times)^2:uv=1+\delta_0(q)+z_1+z_2+\frac{q}{z_1z_2}\right\},$$
where
$$\delta_0(q)=-2q+5q^2-32q^3+286q^4-3038q^5+\ldots.$$

Now, both $\mathcal{M}_C(\check{X})$ and $\mathcal{M}_K(X)$ can be identified with the punctured unit disk $\Delta^*$. Our map $\phi:\Delta^*\to\Delta^*$ is therefore given by
$$q\mapsto\check{q}(q):=q(1-2q+5q^2-32q^3+286q^4-3038q^5+\ldots)^{-3}.$$

On the other hand, the mirror map and its inverse have been computed by Graber-Zaslow in \cite{graber-zaslow01}. First of all, the Picard-Fuchs equation associated to $K_{\proj^2}$ is
$$[\theta_{\check{q}}^3+3\check{q}\theta_{\check{q}}(3\theta_{\check{q}}+1)(3\theta_{\check{q}}+2)]\Phi(\check{q})=0,$$
where $\theta_{\check{q}}$ denotes $\check{q}\frac{\partial}{\partial\check{q}}$, the solution of which with a single logarithm is given by
$$\Phi(\check{q})=-\log\check{q}-\sum_{k=1}^\infty\frac{(-1)^k}{k}\frac{(3k)!}{(k!)^3}\check{q}^k.$$
Hence, the mirror map $\psi:\Delta^*\to\Delta^*$ can be written explicitly as
$$\check{q}\mapsto q(\check{q})=\exp(-\Phi(\check{q}))=\check{q}\exp\left(\sum_{k=1}^\infty\frac{(-1)^k}{k}\frac{(3k)!}{(k!)^3}
\check{q}^k\right).$$
The inverse mirror map can be computed and is given by
\begin{eqnarray*}
q & \mapsto & q+6q^2+9q^3+56q^4+300q^5+3942q^6+\ldots\\
  & = & q(1-2q+5q^2-32q^3+286q^4-3038q^5+\ldots)^{-3}.
\end{eqnarray*}
This shows that $\phi$ coincides with the inverse mirror map up to degree $5$ which provides evidence to Conjecture \ref{can_coords2} for $K_{\proj^2}$.

\subsubsection{$K_{\proj^1\times\proj^1}$}

For $X=K_{\proj^1\times\proj^1}$, the primitive generators of the 1-dimensional cones of the defining fan $\Sigma$ can be chosen to be $v_0=(0,0,1), v_1=(1,0,1), v_2=(0,1,1), v_3=(-1,0,1)$ and $v_4=(0,-1,1)$. We equip $X$ with a toric K\"{a}hler structure $\omega$ so that the associated moment polytope $P$ is defined by the following inequalities
$$x_3\geq0,x_1+x_3\geq0,x_2+x_3\geq0,-x_1+x_3\geq-t_1',-x_2+x_3\geq-t_2'.$$
Here, $t_1'=\int_{l_1}\omega, t_2'=\int_{l_2}\omega>0$ and $l_1,l_2\in H_2(X,\integer)=H_2(\proj^1\times\proj^1,\integer)$ are the classes of the $\proj^1$-factors in $\proj^1\times\proj^1$. To complexify the K\"ahler class, we set $\omega^\cpx=\omega+2\pi\sqrt{-1}B$, where $B$ is a real two-form (the $B$-field). We let $t_1=\int_{l_1}\omega^\cpx, t_2=\int_{l_2}\omega^\cpx\in\cpx$.

There is only one compact toric prime divisor $\mathscr{D}_0$ which is the zero section $\proj^1\times\proj^1\hookrightarrow K_{\proj^1\times\proj^1}$. By Proposition \ref{no_cpt_div} and Theorem \ref{mir_thm}, the instanton-corrected mirror $\check{X}$ is given by
$$\check{X}=\left\{(u,v,z_1,z_2)\in\cpx^2\times(\cpx^\times)^2:uv=1+\delta_0(q_1,q_2)+z_1+z_2
+\frac{q_1}{z_1}+\frac{q_2}{z_2}\right\},$$
where $q_a=\exp(-t_a)$ ($a=1,2$) and
$$1+\delta_0(q_1,q_2)=\sum_{k_1,k_2\geq0} n_{\beta_0+k_1l_1+k_2l_2}q_1^{k_1}q_2^{k_2}.$$

For simplicity, denote $n_{\beta_0+k_1l_1+k_2l_2}$ by $n_{k_1,k_2}$.  By Corollary \ref{mirror_KS}, we have
$$n_{k_1,k_2}=\mathrm{GW}_{0,0}^{Y,k_1L_1+k_2L_2+(k_1+k_2-1)e},$$
where $Y=K_{dP_2}$, $dP_2$ is the blowup of $\proj^1\times\proj^1$ at one point or, equivalently, the blowup of $\proj^2$ at two points, $e\in H_2(dP_2,\integer)$ is the class of the exceptional divisor of the blowup $dP_2\to\proj^1\times\proj^1$ and $L_1,L_2\in H_2(dP_2,\integer)$ are the strict transforms of $l_1,l_2\in H_2(\proj^1\times\proj^1,\integer)$ respectively. See Figure \ref{K_P1^2} below.

\begin{figure}[htp]
\begin{center}
\includegraphics{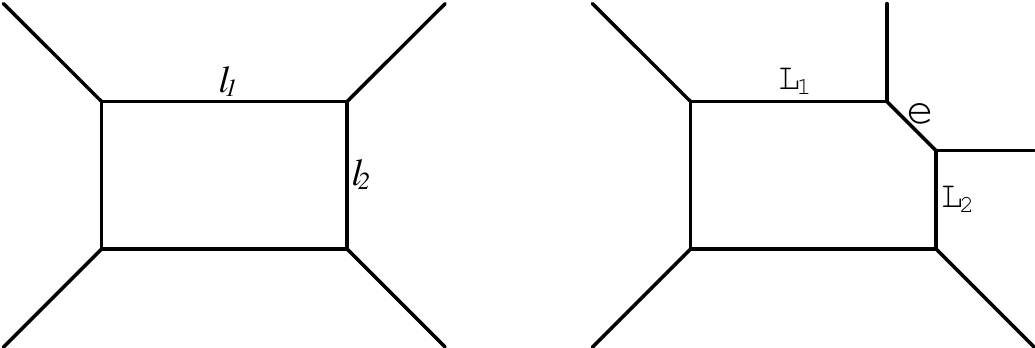}
\end{center}
\caption{Polytope picture for $K_{\proj^1\times\proj^1}$ and $K_{dP_2}$.} \label{K_P1^2}
\end{figure}

The local BPS invariants $\mathrm{GW}_{0,0}^{Y,k_1L_1+k_2L_2+(k_1+k_2-1)e}$ have again been computed by Chiang-Klemm-Yau-Zaslow and the results can be read from "anti-diagonals" of Table 3 on p. 42 in \cite{CKYZ}:
\begin{align*}
n_{0,0}=1,\\
n_{1,0}=n_{0,1}=1,\\
n_{2,0}=n_{0,2}=0, n_{1,1}=3,\\
n_{3,0}=n_{0,3}=0, n_{2,1}=n_{1,2}=5,\\
n_{4,0}=n_{0,4}=0, n_{3,1}=n_{1,3}=7, n_{2,2}=35, \\
n_{5,0}=n_{0,5}=0, n_{4,1}=n_{1,4}=9, n_{3,2}=n_{2,3}=135,\\
&\vdots&
\end{align*}
Hence,
\begin{eqnarray*}
\delta_0(q_1,q_2)& = & q_1+q_2+3q_1q_2+5q_1^2q_2+5q_1q_2^2+7q_1^3q_2+35q_1^2q_2^2+7q_1q_2^3\\
& & +9q_1^4q_2+135q_1^3q_2^2+135q_1^2q_2^3+9q_1q_2^4+\ldots.
\end{eqnarray*}

Now, both $\mathcal{M}_C(\check{X})$ and $\mathcal{M}_K(X)$ can be identified with $(\Delta^*)^2$, and we have $Q^1=(-2,1,0,1,0), Q^2=(-2,0,1,0,1)$. So our map $\phi:(\Delta^*)^2\to(\Delta^*)^2$ is given by $$(q_1,q_2)\mapsto(q_1(1+\delta_0(q_1,q_2))^{-2},q_2(1+\delta_0(q_1,q_2))^{-2}).$$

On the other hand, we can compute the mirror map and its inverse by solving the following Picard-Fuchs equations:
\begin{eqnarray*}
(\theta_1^2-2\check{q}_1(\theta_1+\theta_2)(1+2\theta_1+2\theta_2))\Phi(\check{q}_1,\check{q}_2)=0,\\
(\theta_2^2-2\check{q}_2(\theta_1+\theta_2)(1+2\theta_1+2\theta_2))\Phi(\check{q}_1,\check{q}_2)=0,
\end{eqnarray*}
where $\theta_a$ denotes $\check{q}_a\frac{\partial}{\partial\check{q}_a}$ for $a=1,2$. The two solutions to these equations with a single logarithm are given by
$$\Phi_1(\check{q}_1,\check{q}_2)=-\log\check{q}_1-f(\check{q}_1,\check{q}_2),\
\Phi_2(\check{q}_1,\check{q}_2)=-\log\check{q}_2-f(\check{q}_1,\check{q}_2),$$
where
\begin{eqnarray*}
&&f(\check{q}_1,\check{q}_2)\\
& = &
2\check{q}_1+2\check{q}_2+3\check{q}_1^2+12\check{q}_1\check{q}_2+3\check{q}_2^2
+\frac{20}{3}\check{q}_1^3+60\check{q}_1^2\check{q}_2+60\check{q}_1\check{q}_2^2+\frac{20}{3}\check{q}_2^3\\
& &
+\frac{35}{2}\check{q}_1^4+280\check{q}_1^3\check{q}_2+630\check{q}_1^2\check{q}_2^2
+280\check{q}_1\check{q}_2^3+\frac{35}{2}\check{q}_2^4\\
& &
+\frac{252}{5}\check{q}_1^5+1260\check{q}_1^4\check{q}_2+5040\check{q}_1^3\check{q}_2^2
+5040\check{q}_1^2\check{q}_2^3+1260\check{q}_1\check{q}_2^4+\frac{252}{5}\check{q}_2^5\\
& &+\ldots.
\end{eqnarray*}
This gives the mirror map $\psi:(\Delta^*)^2\to(\Delta^*)^2$:
$$(\check{q}_1,\check{q}_2)\mapsto(\check{q}_1\exp(f(\check{q}_1,\check{q}_2)),
\check{q}_2\exp(f(\check{q}_1,\check{q}_2))).$$
We can then invert this map and the result is given by
$$(q_1,q_2)\mapsto(q_1(1-F(q_1,q_2)),q_2(1-F(q_1,q_2)))$$
where
\begin{eqnarray*}
F(q_1,q_2) & = &
2q_1+2q_2-3q_1^2-3q_2^2+4q_1^3+4q_1^2q_2+4q_1q_2^2+4q_2^3\\
& & -5q_1^4+25q_1^2q_2^2-5q_2^4+\ldots.
\end{eqnarray*}
Now, we compute
\begin{eqnarray*}
(1-F(q_1,q_2))^{-1/2} & = &
1+q_1+q_2+3q_1q_2+5q_1^2q_2+5q_1q_2^2\\
& & +7q_1^3q_2+35q_1^2q_2^2+7q_1q_2^3\\
& & +9q_1^4q_2+135q_1^3q_2^2+135q_1^2q_2^3+9q_1q_2^4+\ldots\\
& = & 1+\delta_0(q_1,q_2).
\end{eqnarray*}
This shows that the inverse mirror map agrees with the map $\phi$ we define up to degree 5, and this gives evidence to Conjecture \ref{can_coords2} for $K_{\proj^1\times\proj^1}$.

\end{document}